\documentclass[11pt]{article}
\usepackage{amsfonts}
\usepackage{lineno,hyperref}
\usepackage{bbm}
\usepackage{amsmath}
\usepackage{mathrsfs}
\usepackage{amssymb}
\usepackage{amsmath,color}
\usepackage{amsthm}
\usepackage{amstext}
\usepackage{amsopn}
\usepackage{texdraw}
\usepackage{graphicx}
\usepackage{multirow}
\usepackage{epsfig,subfig,color}
\usepackage{lscape}
\usepackage{rotating}
\usepackage{float}
\usepackage{cite}
\usepackage{srcltx}
\usepackage[square, comma, sort&compress, numbers]{natbib}
\DeclareGraphicsExtensions{.pdf,.eps,.png,.jpg,.mps}
\usepackage{epstopdf}
\newtheorem{theorem}{Theorem}[section]
\newtheorem{lemma}[theorem]{Lemma}
\newtheorem{proposition}[theorem]{Proposition}
\newtheorem{corollary}[theorem]{Corollary}
\newtheorem{definition}[theorem]{Definition}
\newtheorem{example}[theorem]{Example}
\newtheorem{remark}[theorem]{Remark}
\numberwithin{equation}{section}
\newcommand{\ba}{\begin{array}}
\newcommand{\ea}{\end{array}}

\newcommand{\ds}{\displaystyle}
\newcommand{\Om}{\Omega}

\newcommand{\la}{\lambda}
\newcommand{\R}{{\mathbb R}}
\newcommand{\N}{{\mathbb N}}

\newcommand{\lan}{\langle}

\newcommand{\noi}{\noindent}

\begin{document}
\captionsetup[figure]{labelfont={bf},name={Fig.},labelsep=period}
\date{}
\title{ \bf\large{Effect of spatial average on the spatiotemporal pattern formation of reaction-diffusion systems}\footnote{Partially supported by a grant from China Scholarship Council, US-NSF grant DMS-1715651, National Natural Science Foundation of China (No.11571257), Zhejiang Provincial Natural Science Foundation of China (No.LY19A010010).}}
\author{Qingyan Shi\textsuperscript{1},\ \ Junping Shi\textsuperscript{2}\footnote{Corresponding Author, Email: jxshix@wm.edu},\ \ Yongli Song\textsuperscript{3}
 \\
{\small \textsuperscript{1} School of Science, Jiangnan University,\hfill{\ }}\\
\ \ {\small Wuxi, Jiangsu, 214122, China.\hfill{\ }}\\
{\small \textsuperscript{2} Department of Mathematics, William \& Mary,\hfill{\ }}\\
\ \ {\small Williamsburg, Virginia, 23187-8795, USA.\hfill {\ }}\\
{\small \textsuperscript{3} Department of Mathematics, Hangzhou Normal University,\hfill{\ }}\\
\ \ {\small Hangzhou, Zhejiang, 311121, China.\hfill {\ }}}
\maketitle

\begin{abstract}
Some quantities in the reaction-diffusion models from cellular biology or ecology depend on the spatial average of density functions instead of local density functions. We show that such nonlocal spatial average can induce instability of constant steady state, which is different from classical Turing instability. For a general scalar equation with spatial average, the occurrence of the steady state bifurcation is rigorously proved, and the formula to determine the bifurcation direction and the stability of the bifurcating steady state is given. For the two-species model, spatially non-homogeneous time-periodic orbits could arise due to spatially non-homogeneous Hopf bifurcation from the constant equilibrium. Examples from a nonlocal cooperative Lotka-Volterra model  and a nonlocal Rosenzweig-MacArthur predator-prey model are used to demonstrate the bifurcation of spatially non-homogeneous patterns.
\end{abstract}



%\baselineskip=16pt

%=========================================================
\section{Introduction} \label{Sec:Introduction}

Spatiotemporal pattern formation in the natural world has been a fascinating subject for scientific research in recent years.
One well acknowledged theory is proposed by Turing \cite{Turing1952} who suggested that the random movement of chemicals can destabilize the system and results in the spatially non-homogeneous distribution of chemicals. Different types of Turing-type spatiotemporal patterns have been discovered in chemistry \cite{Lengyel1991,Ouyang1991}, developmental biology \cite{Kondo1995,Sheth2012,Sick2006}, and ecology \cite{Klausmeier1999,Juergens2013,Rietkerk2004}. Turing's theory of diffusion-driven instability or Turing instability has been credited as the main mechanism of these realistic pattern formation phenomena \cite{Kondo2010,maini1997spatial}.

While Turing's instability theory has profound influence on the studies of many spatial chemical or biological models, its scope of application is also restricted. For a system of two interacting chemical/biological species, the occurrence of Turing instability requires (i) an interaction of species of activator-inhibitor type; and (ii) diffusion coefficients of two species in different scales. Hence Turing type pattern formation cannot occur for a two-species reaction-diffusion system if the system is competitive or cooperative type, or the two diffusion coefficients are nearly identical. Indeed it is known that a stable steady state  of a diffusive cooperative (or two-species competitive) system under no-flux boundary condition on a convex domain must be a constant \cite{Kishimoto1985}, and the stability of a constant steady state of a reaction-diffusion system does not change if the diffusion coefficients of variables are identical. It is also known that a stable steady state  of a scalar reaction-diffusion equation under no-flux boundary condition on a convex domain must be a constant \cite{Casten1978,Matano1979}. On the other hand, other types of dispersals have been suggested as possible mechanisms of pattern formation (usually for two-species diffusive competition models), such as cross-diffusion \cite{Lou1996,Mimura1980}, density-dependent diffusion \cite{Mimura1984}, advection towards better resource \cite{ChenXF2012,ChenXF2008,Cosner2003}, or nonlocal competition \cite{Ni2018}. Spatial pattern formation is also possible for scalar equation or  two-species diffusive competition model on a dumbbell-shaped domain (which is not convex) \cite{Matano1979,Matano1983}.

In this paper we explore the effect of spatial average of density functions on the dynamics of reaction-diffusion systems, in particular on the spatiotemporal pattern formation. Here the density function $u(x,t)$ depends on spatial variable $x$ and time $t$, and the spatial average is $\bar{u}=\ds \frac{1}{|\Omega|}\int_{\Omega} u(y,t) dy$ where $\Omega$ is the bounded spatial domain and $|\Omega|$ is the Lebesgue measure (volume) of $\Omega$. This is a special form of integral average like $\ds \int_{\Omega} K(x,y)u(y,t) dy$ with an integral kernel $K(x,y)$. Such nonlocal effect appears in various reaction-diffusion models.  In \cite{Britton1989,Gourley2001}, such a nonlocal term represents the aggregation  induced by grouping behavior, for example, the aggregation of insects for the purpose of social work or the herd behavior for defense. The integral form also appears as nonlocal competition for the resource or a nonlocal crowding effect in a scalar model of bacteria colonies  \cite{ChenShi2012,Fuentes2003,Furter1989,Sun2013}, and further studies have been conducted for diffusive competition model \cite{Ni2018} or predator-prey model with nonlocal crowding effect in prey population \cite{ChenYu2017,Merchant2011}. Another reaction-diffusion model with effect of spatial average was proposed in \cite{Altschuler2008} where the integral term represents the total amount of cytoplasmic molecules in a feedback loop, see also \cite{Tian2018} for a more recent study.

Motivated by  previous examples, we consider the following general form of two-species reaction-diffusion system with spatial average:
\begin{equation}\label{equ-main}
\begin{cases}
u_{t}=d_1\Delta u+f(u,v,\bar{u},\bar{v},r), & x\in\Om,\ t>0,\\
v_{t}=d_2\Delta v+g(u,v,\bar{u},\bar{v},r), & x\in\Om,\ t>0,\\
\partial_{\nu}u=\partial_{\nu}v=0, & x\in\partial\Om,\ t>0,
\end{cases}
\end{equation}
where $u(x,t)$, $v(x,t)$ are the density functions of two interacting chemical/biological species, $\bar{u}=\dfrac{1}{|\Om|}\ds\int_{\Om}u(x,t)dx$, $\bar{v}=\dfrac{1}{|\Om|}\ds\int_{\Om}v(x,t)dx$ are the spatial averages of $u$ and $v$ respectively; $\Omega$ is a bounded domain in $\R^m$ ($m\ge 1$) with smooth boundary $\partial\Om$; a no-flux boundary condition is imposed so the system is a closed one; the interactions are described by smooth functions $f,g:\R^5\to\R$; and  $d_1,~d_2>0$ are the diffusion coefficients and $r>0$ is a possible kinetic system parameter. Assume that $(u_*,v_*)$ is a non-negative spatially constant steady state, and it is linearly stable with respect to a spatially homogeneous perturbation. We show that $(u_*,v_*)$ can be unstable under a spatially non-homogeneous perturbation, that is, the constant steady state $(u_*,v_*)$ is unstable for the system \eqref{equ-main}. While this has been shown to be possible under the Turing instability scheme, our instability result does not necessarily require the activator-inhibitor interaction, nor it requires the different scales of diffusion coefficients. Also our approach can not only produce spatially non-homogeneous steady state pattern through \textit{steady state instability}, but it also can produce spatially non-homogeneous time-periodic oscillatory patterns through \textit{wave instability}. All these patterns can be generated through varying the diffusion coefficients, and bifurcation theory can be used to prove the existence of small amplitude non-constant steady states or periodic orbits.  Note that classical Turing mechanism cannot lead to wave instability for systems with only two interacting species.

More specifically, let the Jacobian matrices at $(u_*,v_*)$ be defined as
\begin{equation}
J_{U}=\begin{pmatrix}f_u & f_v\\ g_u & g_v\end{pmatrix},~J_{\bar{U}}=\begin{pmatrix}f_{\bar{u}} & f_{\bar{v}}\\ g_{\bar{u}} & g_{\bar{v}}\end{pmatrix}.
\end{equation}
We assume that the matrix $J_U+J_{\bar{U}}$ is stable with all eigenvalues with negative real parts, but $J_U$ is not stable, then we have the following scenarios for the pattern formation of system \eqref{equ-main}: (see Theorem \ref{theorem-instability} for more details)
\begin{description}
\item (i) if $Tr(J_U)<0$, then steady state instability may occur but not the wave instability;
\item (ii) if  $Tr(J_U)>0$, then both wave and steady state instability may occur.
\end{description}
Here $Tr(J_U)=f_u+g_v$ is the trace of $J_U$. The studies here is induced by the dependence of dynamics on the spatial average of variable, which is reflected in $J_{\bar{U}}$. A similar study in \cite{Chen2013} considered the  dependence of dynamics on the time-delayed variables. The diffusion-induced pattern formation found in  \eqref{equ-main} here
does not occur in the corresponding ``localized system" of \eqref{equ-main}:
\begin{equation}\label{equ-mainlocal}
\begin{cases}
u_{t}=d_1\Delta u+f(u,v,u,v,r), & x\in\Om,\ t>0,\\
v_{t}=d_2\Delta v+g(u,v,u,v,r), & x\in\Om,\ t>0,\\
\partial_{\nu}u=\partial_{\nu}v=0, & x\in\partial\Om,\ t>0,
\end{cases}
\end{equation}
which is the standard two-species reaction-diffusion system where the reaction is completely localized, or in the corresponding ordinary differential equation model in which the reaction is completely homogenized. Hence both the localized reaction and the homogenized reaction  pattern contribute to the formation occurred in  \eqref{equ-main}. This shows that not only spatial heterogeneity can induce rich spatial patterns, but sometimes partial homogeneity can also lead to spatiotemporal patterns.

As example of this new pattern formation mechanism, we show in  Section 4 that in a reaction-diffusion Lotka-Volterra cooperative system with a nonlocal intraspecific competition, stable spatially non-homogeneous steady state pattern can occur when one of diffusion coefficients decreases, while the constant coexistence steady state is globally asymptotically stable in its corresponding localized system. In this case, the interaction between the two species is clearly not activator-inhibitor type, but a cooperative or mutualistic  one. In various spatially heterogenous ecosystems, alternative stable states or self-organized patterns have been found \cite{kefi2016can}, and the mechanism introduced here could be the cause of spatially non-homogeneous patterns. In Section 5, we demonstrate the occurrence of both steady state and wave instability in a reaction-diffusion Rosenzweig-MacArthur predator-prey model with a  nonlocal intraspecific competition in the prey population. Again in the corresponding localized system, the constant coexistence steady state is globally asymptotically stable. But the addition of the spatial average intraspecific competition can lead to either a spatially non-homogeneous steady state or a spatially non-homogeneous time-periodic pattern. The latter one can be viewed as stable pattern generated from Turing-Hopf bifurcation, which is rarely achieved in two-variable reaction-diffusion models \cite{maini1997spatial}.

Our result also has  a  version for the scalar counter part of  \eqref{equ-main}:
\begin{equation}\label{equ-single}
\begin{cases}
u_{t}=d\Delta u+r f(u,\bar{u}), & x\in\Om,\ t>0,\\
\partial_{\nu}u=0, & x\in\partial\Om,\ t>0.
\end{cases}
\end{equation}
Assume that $u_*$ is a constant steady state, and it is stable for the non-spatial model in the sense that $f_u+f_{\bar{u}}<0$ at $u=u_*$. In  Section 2 we show that
\begin{description}
\item (i) if $f_{u}<0$, then $u_*$ is locally asymptotically stable for all $d,r>0$;
\item (ii) if $f_{u}>0$, then there exists $d_1>0$ such that $u_*$ is locally asymptotically stable for $d>d_1$, and it is unstable for $0<d<d_1$. A spatially non-homogeneous  steady state pattern emerges at $d=d_1$.
\end{description}
Here $f_u=f_u(u_*,u_*)$. The above results for the scalar equation \eqref{equ-single} have been implied in \cite{Furter1989}, and our results for the two-species model \eqref{equ-main} are generalizations of these results in a sense. But for scalar equations, wave instability cannot occur and there are more possible cases to consider for the two-species model \eqref{equ-main}.

This paper is organized as follows. First the pattern formation for a general scalar equation with spatial average in studied in Section 2. In Section 3, the possible scenarios for pattern formation in a  general two-species reaction-diffusion model with spatial average subjected to the homogeneous Neumann boundary condition are considered. The general theory is applied to  two specific  biological system:  a diffusive  Lotka-Volterra cooperative model and a  diffusive Rosenzweig-MacArthur predator-prey model each with effect of spatial average,  in Section 4 and Section 5 respectively.  In Section 6, we conclude our work and compare our results with the classic Turing pattern formation. For the convenience of the following analysis, we introduce some notations: the real-valued Sobolev space corresponding to the Neumann boundary value problem is denoted as $X=\{u\in W^{2,p}(\Om):\partial_{\nu}u=0\}$ and $Y=L^p(\Omega)$ denotes the real-valued $L^p$ space, where $p>m$. Also, it is well known that the eigenvalue problem
\begin{equation}\label{eig0}
\begin{cases}
\Delta \varphi+\la\varphi=0, & x\in\Om,\\
\partial_{\nu}\varphi=0, & x\in\partial\Om,
\end{cases}
\end{equation}
has infinitely many eigenvalues satisfying
\begin{equation*}
0=\la_0<\la_1\le \la_2\le\cdots\le\la_i\le\la_{i+1}\le \cdots<+\infty,
\end{equation*}
with the corresponding eigenfunction $\varphi_i$ ($i\geq0$) satisfying  $\ds\int_{\Omega}\varphi_i^2dx=1$.

\section{Pattern formation in scalar models}

In this section we consider the pattern formation in the scalar reaction-diffusion model \eqref{equ-single}. We recall that from \cite{Casten1978,Matano1979}, the localized model
\begin{equation*}
\begin{cases}
u_{t}=d\Delta u+rf(u,u), & x\in\Om,\ t>0,\\
\partial_{\nu}u=0, & x\in\partial\Om,\ t>0,
\end{cases}
\end{equation*}
has no non-constant stable steady state solutions if $\Omega$ is convex.

We assume that there exists at least one positive constant steady state $u=u_*$ of \eqref{equ-single} such that $f(u_*,u_*)=0$.
Linearizing Eq. \eqref{equ-single} at $u=u_*$, we obtain an eigenvalue problem
\begin{equation}\label{eig1}
\begin{cases}
d\Delta \phi+r(f_{u}\phi+f_{\bar{u}}\bar{\phi})=\mu\phi, & x\in\Om,\\
\partial_{\nu}\phi=0, & x\in\partial\Om,
\end{cases}
\end{equation}
where $f_u=f_u(u_*,u_*)$ and $f_{\bar{u}}=f_{\bar{u}}(u_*,u_*)$. The eigenvalues of \eqref{eig1} are easy to determine as follows:
\begin{lemma}\label{lem:2.1}
Let $\la_i$ be  eigenvalues of \eqref{eig0} and let $\varphi_i$ be the corresponding eigenfunctions for $i\in \N_0$. Then the eigenvalues of \eqref{eig1} are $\mu_0=r(f_u+f_{\bar{u}})$ with eigenfunction $\phi_0=1$, and $\mu_i=-d\la_i+rf_u$ for $i\ge 1$ with eigenfunction $\phi_i=\varphi_i$.
\end{lemma}
\begin{proof}
Integrating \eqref{eig1}, we have that $r(f_{u}+f_{\bar{u}})\bar{\phi}=\mu\bar{\phi}$. When $\bar{\phi}\ne 0$, we obtain $\mu_0=rf_u$ and $\phi_0=\bar{\phi_0}=1$; and when $\bar{\phi}= 0$, we obtain $\mu_i=-d\la_i+rf_u$ and $\phi_i=\varphi_i$ for $i\ge 1$.
\end{proof}

The stability of a constant steady state $u=u_*$ and possible emergence of spatial patterns of \eqref{equ-single} now can be stated as follows.
\begin{theorem}\label{theorem-single}
Suppose that $r>0$, $f\in C^1(\R^2,\R)$ satisfying $f(u_*,u_*)=0$ for some $u_*\ge 0$,  $f_u=f_u(u_*,u_*)$, $f_{\bar{u}}=f_{\bar{u}}(u_*,u_*)$, and  $f_{u}+f_{\bar{u}}<0$. %And we assume that $\la_i$ are simple eigenvalues of \eqref{eig0},  then for ,
\begin{description}
\item (i) if $f_{u}<0$, then $u_*$ is locally asymptotically stable with respect to \eqref{equ-single} for all $d,r>0$;
\item (ii)  if $f_{u}>0$, then there exist $d_1:=rf_{u}/\la_1$ such that $u_*$ is locally asymptotically stable for $d>d_1$, and it is unstable for $0<d<d_1$.
\end{description}
\end{theorem}
\begin{proof}
The condition $f_{u}+f_{\bar{u}}<0$ guarantees that $u=u_*$ is locally asymptotically stable in the absence of diffusion and $\mu_0<0$. When $f_{u}<0$ is satisfied, from Lemma \ref{lem:2.1}, we see that $\mu_i<0$ holds for any $i\in\mathbb{N}$, thus $u_*$ is locally asymptotically stable for system \eqref{equ-single}, thus (i) is proved. If $f_{u}>0$, it is possible for $\mu_i=-d\la_i+rf_u=0$ and it occurs at $d=d_i:=rf_u/\la_i$. Also, we know that the constant equilibrium loses its stability at the first bifurcation point $d=d_1$. This completes the proof of part (ii).
\end{proof}

In the following theorem, we give a more detailed bifurcation result for the following steady state (nonlocal elliptic) problem:
\begin{equation}\label{equ-singleelliptic}
\begin{cases}
d\Delta u+rf(u,\bar{u})=0, & x\in\Om,\\
\partial_{\nu}u=0, & x\in\partial\Om.
\end{cases}
\end{equation}

\begin{theorem}\label{theorem-scalarbifurcation}
Suppose that $r>0$, $f\in C^1(\R^2,\R)$ satisfying $f(u_*,u_*)=0$ for some $u_*\ge 0$ and  $f_{u}+f_{\bar{u}}<0$. And we assume that for some $i\in \N$, $\la_i$ is a simple eigenvalue of \eqref{eig0}, and $f_{u}>0$.
\begin{description}
\item (i) Near $(d_i,u_*)$, Eq. \eqref{equ-singleelliptic} has a line of trivial solutions  $\Gamma_0=\{(d,u_*):d>0\}$  and a family of nontrivial solutions bifurcating from $\Gamma_0$ at $d=d_i$:
\begin{equation}\label{equ-singleGamman}
\Gamma_i=\left\{\left(d_i(s), u_i(s,x)\right):\ -\delta<s<\delta\right\},
\end{equation}
where $\delta>0$, $u_i(s,x)=u_*+s\varphi_i(x)+sg_i(s,x)$ and $d_i(s),~g_i(s,\cdot)$ are continuous functions defined for $s\in (-\delta,\delta)$ such that $d_i(0)=d_i$, and $g_i(0,\cdot)=0$; and there are no other solutions of \eqref{equ-singleelliptic} than the ones on $\Gamma_0$ and $\Gamma_i$ near $(d,u)=(d_i,u_*)$.

\item (ii) If $f\in C^2$ near $(u_*,u_*)$, then $d_i(s),~g_i(s,\cdot)$ are $C^1$ for $s\in (-\delta,\delta)$, and
\begin{equation}\label{d'}
d_i'(0)=\dfrac{d_if_{uu}\ds\int_{\Om}\varphi_i^3dx}{2f_u\ds\int_{\Om}\varphi_i^2dx},
\end{equation}
If $d_i'(0)\not=0$, then the steady state bifurcation at $d=d_i$ is transcritical type. Moreover  the solution $(d_1(s),u_1(s,\cdot))$ with $d_1(s)<d_1$ is locally asymptotically stable, and the one with $d_1(s)>d_1$ is unstable; and all solutions of $\Gamma_i$ with $i\ge 2$ are unstable.%, and the bifurcation is supercritical when $d'(0)<0$ and subcritical when $d'(0)>0$.
\item (iii) If $d_i'(0)=0$ and $f\in C^3$ near $(u_*,u_*)$, then $d_i(s),~g_i(s,\cdot)$ are $C^2$ for $s\in (-\delta,\delta)$, and
\begin{equation}\label{d''}
d_i''(0)=\dfrac{d_if_{uuu}\ds\int_{\Omega}\varphi_i^4dx+3d_if_{uu}\ds\int_{\Omega}w\varphi_i^2dx+3d_if_{u\bar{u}}\ds\int_{\Omega}\bar{w}\varphi_i^2dx}{3f_{u}\ds\int_{\Omega}\varphi_i^2dx},
\end{equation}
where $w=w(x)$ is the unique solution of
\begin{equation}\label{equ-singletheta}
\begin{cases}
d_i\Delta w+rf_u w+rf_{\bar{u}}\bar{w}=-rf_{uu}\varphi_1^2, & x\in\Omega, \\
\partial_{\nu}w=0, & x\in\partial\Omega,\\
\ds\int_{\Omega}w(x)dx=0.
\end{cases}
\end{equation}
 If  $d_i''(0)\not=0$, then the steady state bifurcation at $d=d_i$ is pitchfork type. Moreover, the solution $(d_1(s),u_1(s,\cdot))$ with all $s\ne 0$ is locally asymptotically stable  when $d_1''(0)<0$ (the bifurcation is supercritical), and the solution $(d_1(s),u_1(s,\cdot))$ with all $s\ne 0$ is unstable when $d_1''(0)>0$ (the bifurcation is subcritical).
\end{description}
\end{theorem}

\begin{proof}
For Eq. \eqref{equ-singleelliptic},  $u=u_*$ is a constant steady state of \eqref{equ-singleelliptic} for all $r,d>0$. Fixing $r>0$, we define a nonlinear mapping $F:\ \mathbb{R}^+\times X\rightarrow Y$ by
\begin{equation}\label{equ-singleF}
 F(d,u)=d\Delta u+rf(u,\bar{u}).
\end{equation}
It is clear that $F(d,u_*)=0$.

Then, we have
\begin{equation}\label{equ-singleL}
 F_{u}\left(d_i, u_*\right)[\psi]=d_i\Delta\psi+rf_u\psi+rf_{\bar{u}}\bar{\psi}:=L[\psi].
\end{equation}

\noi\textbf{Step 1.} First, we determine the null space $\mathcal{N}(L)$ of $L$. If $\psi\in\mathcal{N}(L)$, then we have
\begin{equation}\label{equ-singlekernel}
d_i\Delta\psi+rf_u\psi+rf_{\bar{u}}\bar{\psi}=0,
\end{equation}
or equivalently, $\Delta\psi+\la_i\psi+rf_{\bar{u}}/d_i\bar{\psi}=0$.
Integrating Eq. \eqref{equ-singlekernel}, we obtain
\begin{equation*}
r(f_u+f_{\bar{u}})\bar{\psi}=0,
\end{equation*}
which implies that $\bar{\psi}=0$ as $f_u+f_{\bar{u}}<0$ and $r>0$, so $\psi$ satisfies that $\Delta\psi+\la_i\psi=0$, then $\psi=\varphi_i$. And $\mathcal{N}(L)=\textrm{Span}\left\{\varphi_i\right\}$
as $\la_i$ is assumed to be simple, thus $\dim {\mathcal{N}(L)}=1$.

\noi\textbf{Step 2.} We next consider the range space $\mathcal{R}(L)$  of $L$. If $q\in\mathcal{R}(L)$, then there exist $\psi\in X$ such that
\begin{equation}\label{equ-singlerange}
d_i\Delta\psi+rf_u\psi+rf_{\bar{u}}\bar{\psi}=q.
\end{equation}
Multiplying the equation \eqref{equ-singlerange} by $\varphi_i$ and integrating over $\Omega$, we obtain
\begin{equation*}
0=rf_{\bar{u}}\bar{\psi}\int_{\Omega}\varphi_idx=\int_{\Omega}q\varphi_idx.
\end{equation*}
On the other hand, if $\ds\int_{\Omega}q\varphi_idx=0$, then the solution of \eqref{equ-singlerange} is
\begin{equation*}
\psi=\dfrac{\bar{q}}{r(f_u+f_{\bar{u}})}+\sum\limits_{j\neq i}\dfrac{a_j}{rf_u-d_i\la_j}\varphi_j+k\varphi_i,~\textrm{if}~q=\bar{q}+\sum\limits_{j\neq i}a_j\varphi_j,
\end{equation*}
where $k\in\mathbb{R}$ is arbitrary. Hence $\ds\mathcal{R}(L)=\left\{q\in Y: \int_{\Omega}q\varphi_idx=0\right\}$,
which is co-dimensional $1$ in $Y$.

\noi\textbf{Step 3.} We prove that $F_{du}(d_i,u_*)[\varphi_i]\not\in\mathcal{R}(L)$. From \eqref{equ-singleF}, we have
\begin{equation}\label{equ-derivatived}
 F_{du}\left(d,u_*\right)[\varphi_i]=\Delta\varphi_i=-\la_i\varphi_i\not\in\mathcal{R}(L),
\end{equation}
as $\ds\int_{\Omega}\la_i\varphi^2_idx\neq0$.
By applying Theorem 1.7 in \cite{Rabinowitz1971}, we conclude that there exists an open interval $(-\delta, \delta)$ with $\delta>0$
 and continuous functions $d_i(s): (-\delta, \delta)\to \mathbb{R},~g_i(\cdot, s) : (-\delta, \delta)\to Z$,
 where $Z$ is any complement of $\text{Span}\{\varphi_i\}$, such that the solution set of \eqref{equ-singleelliptic} near $\left(d_i, u_*\right)$ consists precisely of the curves $\Gamma_0$ and $\Gamma_i$ defined by \eqref{equ-singleGamman}. This completes the proof of part (i).

\noi\textbf{Step 4.} Now we consider the bifurcation direction and stability of the bifurcating solutions in $\Gamma_i$. According to the results in \cite{Crandall1973,Shi1999}, the direction of the steady state bifurcation is determined by $d_i'(0)$ and $d_i''(0)$. For $y\in Y^*$ (the conjugate space of $Y$) defined by
$\lan y,q\rangle=\ds\int_{\Omega}q\varphi_idx$,
we have \cite{Shi1999}
\begin{equation}\label{equ-singled'}
\aligned
d_i'(0)=&-\frac{\left\langle y,\ F_{uu}\left(d_i, u_*\right)[\varphi_i,\varphi_i]\right\rangle}
 {2\left\langle y,\ F_{du}\left(d_i,u_*\right)[\varphi_i]\right\rangle}.
\endaligned\end{equation}
By \eqref{equ-derivatived} and the definition of $y$, we have
\begin{equation*}
\left\langle y,\ F_{du}\left(d_i, u_*\right)[\varphi_i]\right\rangle=-\la_i\int_{\Omega}\varphi_i^2dx.
\end{equation*}
From \eqref{equ-singleL}, it can be obtained that
\begin{equation*}
F_{uu}\left(d_i,u_*\right)[\varphi_i,\varphi_i]=rf_{uu}\varphi_i^2.
\end{equation*}
Therefore,
\begin{equation*}
d_i'(0)=\dfrac{rf_{uu}\ds\int_{\Omega}\varphi_i^3dx}{2\la_i\ds\int_{\Omega}\varphi_i^2dx}=\dfrac{d_if_{uu}\ds\int_{\Omega}\varphi_i^3dx}{2 f_u\ds\int_{\Omega}\varphi_i^2dx},
\end{equation*}
where $d_i=rf_u/\la_i$ is applied. Then, according to \cite{Crandall1973,Shi1999}, a transcritical steady state bifurcation occurs at $d=d_i$ if $d_i'(0)\not=0$. %Moreover, the bifurcation is supercritical if $d_i'(0)<0$ and subcritical if $d_i'(0)>0$.

If $d_i'(0)=0$, then we need to calculate $d_i''(0)$ to determine the bifurcation direction. %In \cite{Jin2013}, a calculation for the second derivative of the bifurcation parameter is presented for a general reaction-diffusion system without the spatial average effect. Now, we will perform the calculation when the spatial average effect is present in the system.
According to \cite{Shi1999}, $d_i''(0)$ takes the following form:
\begin{equation}\label{equ-singled''}
d_i''(0)=-\dfrac{\left\langle y,\ F_{uuu}\left(d_i,u_*\right)[\varphi_i,\varphi_i,\varphi_i]\right\rangle+3\left\langle y,\ F_{uu}\left(d_i, u_*\right)[\varphi_i,\eta]\right\rangle}{3\left\langle y,\ F_{du}\left(d_i,u_*\right)[\varphi_i]\right\rangle},
\end{equation}
where $\eta$ is the unique solution of
\begin{equation}\label{equ-singleTheta}
F_{uu}\left(d_i,u_*\right)[\varphi_i,\varphi_i]+F_{u}\left(d_i,u_*\right)[\eta]=0,
\end{equation}
which is equivalent to \eqref{equ-singletheta}.
By \eqref{equ-singleL}, we have
\begin{equation*}
F_{uu}(d_i,u_*)[\varphi_i,\eta]=rf_{uu}\varphi_i\eta+rf_{u\bar{u}}\varphi_i\bar{\eta},
\end{equation*}
and
\begin{equation*}
F_{uuu}(d_i,u_*)[\varphi_i,\varphi_i,\varphi_i]=rf_{uuu}\varphi_i^3.
\end{equation*}
Substituting them into \eqref{equ-singled''}, we obtain \eqref{d''}.
From \cite{Shi1999}, $d_i''(0)<0$ implies a supercritical pitchfork type bifurcation occurs and $d_i''(0)>0$ implies a subcritical pitchfork type bifurcation occurs.

\noi\textbf{Step 5.} The bifurcating solutions on $\Gamma_i$ with $i\ge 2$ are all unstable as the trivial solution $(d,u_*)$ is unstable for $0<d<d_1$ (Lemma \ref{lem:2.1}). The stability of bifurcating non-constant steady state solutions on $\Gamma_1$  can be determined by the two eigenvalue problems (see \cite{Crandall1973})
\begin{equation*}
\begin{aligned}
F_{u}(d,u_*)[\psi(d)]=M(d)K[\psi(d)],~\textrm{for}~d\in(d_1-\epsilon,d_1+\epsilon),\\
F_{u}(d_1(s),u_1(s,\cdot))[\Psi(s)]=\mu(s)K[\Psi(s)],~\textrm{for}~s\in(-\delta,\delta),
\end{aligned}
\end{equation*}
where $K:X\to Y$ is inclusion map $K(u)=u$, $M(d)$ and $\mu(s)$ satisfy $M(d_1)=\mu(0)=0$ and $\psi(d_1)=\Psi(0)=\varphi_1$.
By applying Corollary 1.13 and Theorem 1.16 in \cite{Crandall1973} or Theorem 5.4 in \cite{LiuShi2018}, the stability of $(d_1(s),u_1(s,\cdot))$ can be determined by the sign of $\mu(s)$ which satisfies
\begin{equation}\label{formula}
\lim_{s\rightarrow0}\frac{-sd_1'(s)M'(d_1)}{\mu(s)}=1.
\end{equation}

It is easy to calculate that $M(d)=rf_u-d\la_1$ with $\psi(d)=\varphi_1$, so $M'(d_1)=-\la_1<0$. Thus \eqref{formula} implies that ${\rm Sign}(\mu(s))={\rm Sign}(sd_1'(s))$. When $d_1'(0)=0$ and $d_1''(0)<0$, we have $sd_1'(s)<0$  so $\mu(s)<0$ for all $s\ne 0$, hence a supercritical pitchfork bifurcation occurs.  Similarly when $d_1'(0)=0$ and $d_1''(0)>0$, all bifurcating steady states are unstable for $s\ne 0$. The case for $d_1'(0)\ne 0$ can be obtained in a similar way as well.
\end{proof}

We apply the results in Theorems \ref{theorem-single} and \ref{theorem-scalarbifurcation} to the following two examples.
\begin{example}\label{exm2.1}
The following diffusive population model was considered in \cite{Furter1989}:
\begin{equation}\label{equ-exampleone}
\begin{cases}
u_{t}=d\Delta u+r u(1+au-b\bar{u}), & x\in\Om,\ t>0,\\
\partial_{\nu}u=0, & x\in\partial\Om,\ t>0,
\end{cases}
\end{equation}
where $a,~b,~r>0$ are constants, and $d$ is the diffusion coefficient. The growth rate per capita $r(1+au-b\bar{u})$ in \eqref{equ-exampleone} has a nonlocal crowding effect $-rb\bar{u}$ but also a localized positive dependent term $rau$.
When $b>a$, Eq. \eqref{equ-exampleone} has a unique positive constant equilibrium $u_*=1/(b-a)$, and we can calculate that $f_u(u_*,u_*)=\dfrac{ a}{b-a}>0$ and $f_u(u_*,u_*)+f_{\bar{u}}(u_*,u_*)=-1<0$. So from Theorem \ref{theorem-single}, $u=u_*$ is locally asymptotically stable when $d>d_1$ and it is unstable when $0<d<d_1$, where $d_1=\dfrac{ ar}{(b-a)\la_1}$. Assume $m=1$ and $\Omega=(0,l\pi)$ for some $l>0$, $d_i=\dfrac{arl^2}{(b-a)i^2}$ for $i\in \N$ and the corresponding eigenfunction at $d=d_i$ is $\cos(ix/l)$. From Theorem \ref{theorem-scalarbifurcation} and the fact that $\ds\int_{0}^{l\pi}\cos^3(x/l)dx=0$, we find that $d_1'(0)=0$; and from \eqref{d''}, \begin{equation*}\label{equ-ex1fderi}
f_{uu}=2a,~f_{u\bar{u}}=-b,~f_{uuu}=0,
\end{equation*}
and
\begin{equation*}\label{equ-ex1theta}
\theta=-\dfrac{f_{uu}}{2(f_u+f_{\bar{u}})}+\dfrac{f_{uu}}{6f_u}\cos\left(\frac{2ix}{l}\right),
\end{equation*}
we obtain that $d_1''(0)=\ds\dfrac{b}{(a-b)u_*}<0$. Then Theorem \ref{theorem-scalarbifurcation} shows that a supercritical pitchfork type steady state bifurcation occurs for system \eqref{equ-exampleone} at $d=d_1$, and the bifurcating non-homogeneous steady states are locally asymptotically stable (see Fig. \ref{fig:1} for numerical simulation).
\end{example}

\begin{figure}[htp]
\centering
\subfloat[]{\includegraphics[scale=0.45]{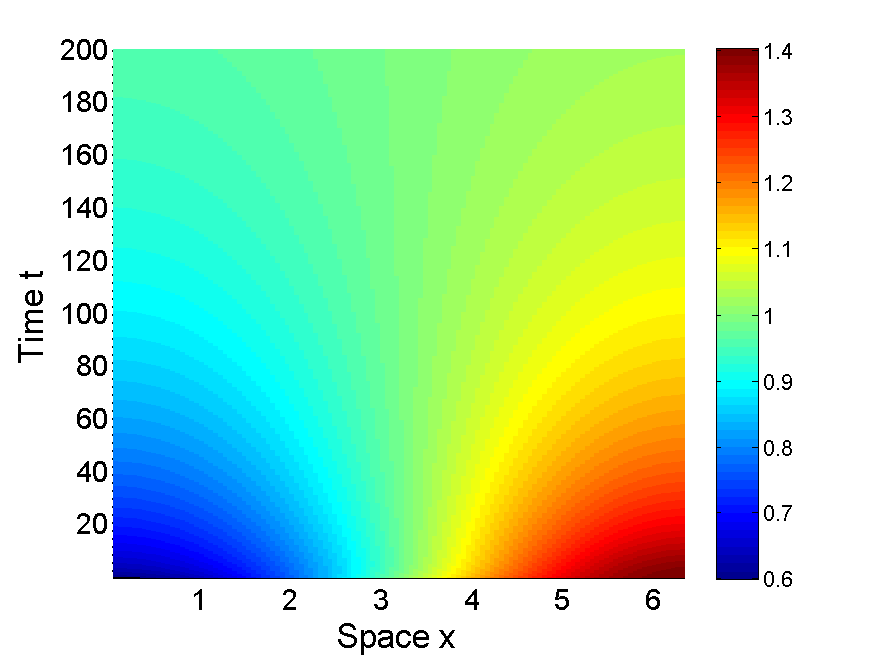}}
\subfloat[]{\includegraphics[scale=0.45]{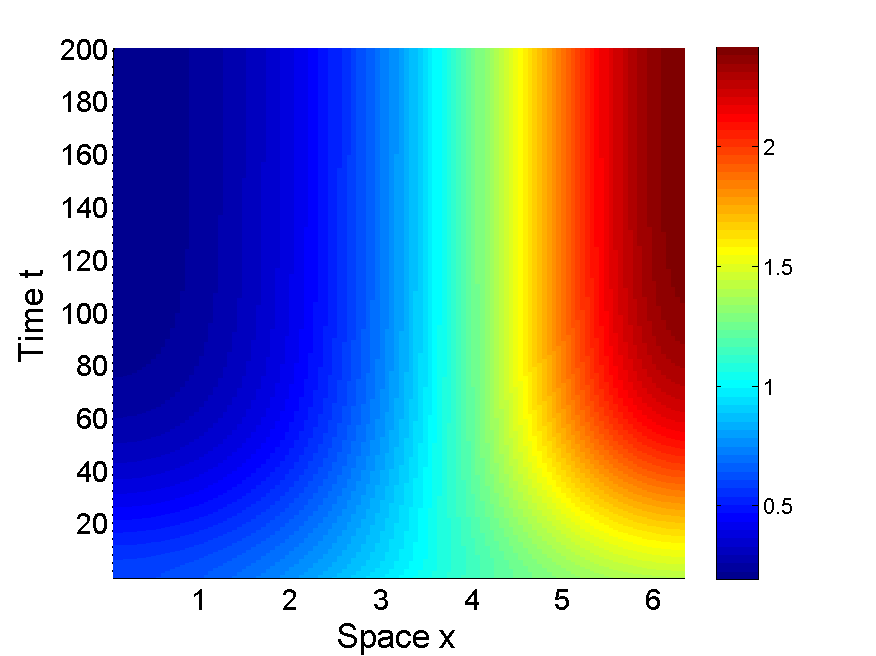}}
  \caption{Dynamics of Eq. \eqref{equ-exampleone} with  $a=0.1,~b=1.1,~r=1$ and $\Om=(0,2\pi)$: (Left) convergence to a constant steady state when $d=0.45>d_1=0.4$; (Right) convergence to a non-constant steady state when $d=0.3<d_1$.}\label{fig:1}
\end{figure}

\begin{example}\label{example-singletwo}
Consider the logistic type model:
\begin{equation}\label{equ-exampletwogeneral}
\begin{cases}
u_{t}=d\Delta u+a-b\bar{u}-cu-d\bar{u}^2-e u\bar{u}, & x\in\Om,\ t>0,\\
\partial_{\nu}u=0, & x\in\partial\Om,\ t>0,
\end{cases}
\end{equation}
where $a,~b,~c,~d,~e$ are all constants. We assume that
\begin{equation}\label{ade}
  a>0, \;\; d+e>0.
\end{equation}
It is clear that under \eqref{ade}, there is a unique positive constant steady state $u=u_*$ satisfying
$a-(b+c)u-(d+e)u^2=0$. Since $f_u=-(c+eu_*)<0$ and $f_{\bar{u}}=-(b+2du_*+eu_*)<0$, $u=u_*$ is locally asymptotically stable for all $d>0$ from Theorem \ref{theorem-single} (i) when \eqref{ade} is satisfied.
Indeed the constant steady state  $u=u_*$ is globally asymptotically stable as in the following proposition, and its proof is included in the Appendix.

\begin{proposition}\label{theorem-global}
The unique positive constant steady state $u=u_*$ of Eq. \eqref{equ-exampletwogeneral} is globally asymptotically stable for all non-negative initial conditions when \eqref{ade} is satisfied.
\end{proposition}
\end{example}

As an application of \eqref{equ-exampletwogeneral} and Proposition \ref{theorem-global}, we consider the following model proposed in \cite{Altschuler2008}:
\begin{equation}\label{equ-exampletwo}
\begin{cases}
u_{t}=d\Delta u+k_{on}(1-\bar{u})+k_{fb}(1-\bar{u})u-k_{off}u, & x\in\Om,\ t>0,\\
\partial_{\nu}u=0, & x\in\partial\Om,\ t>0.
\end{cases}
\end{equation}
Here $u$ is the density of membrane-bound molecules and $\bar{u}:=|\Om|^{-1}\ds\int_{\Om}u(x,t)dx$ denotes the total density of cytoplasmic molecules. From \cite{Altschuler2008}, the four terms in Eq. \eqref{equ-exampletwo} can be interpreted as: (1) $D$ is the lateral diffusion rate of molecules; (2) $k_{on}$ stands for the spontaneous association of cytoplasmic molecules to random locations on the membrane; (3) $k_{fb}$ represents the recruitment of cytoplasmic molecules to the locations of membrane-bound signalling; (4) $k_{off}$ is the rate of random disassociation of molecules from the membrane. Then from Proposition \ref{theorem-global}, there is a unique positive constant steady state $u=u_{*}$ satisfying $k_{pf}u^2+(k_{on}+k_{off}-k_{pf})u-k_{on}=0$, and it is globally asymptotically stable thus there is no spatial pattern in \eqref{equ-exampletwo}.

\section{Pattern formation in two-species system}

For model \eqref{equ-main}, we assume that $f,~g$ are $C^{k}$ functions with ($k\ge 1$)  satisfying
\begin{equation*}
f(u_*,v_*,u_*,v_*,r)=0,~~g(u_*,v_*,u_*,v_*,r)=0,
\end{equation*}
which means that $(u_*,v_*)$ is a constant steady state of system \eqref{equ-main} for all $r>0$ as well as the localized system \eqref{equ-mainlocal}. We linearize Eq. \eqref{equ-main} at $(u_*,v_*)$ and obtain:
\begin{equation}\label{equ-linear}
\begin{pmatrix}\phi_t \\ \psi_t\end{pmatrix}= D\begin{pmatrix}\Delta \phi \\ \Delta \psi\end{pmatrix}+J_{U}\begin{pmatrix}\phi \\ \psi\end{pmatrix}+J_{\bar{U}}\begin{pmatrix}\bar{\phi} \\ \bar{\psi}\end{pmatrix},
\end{equation}
where
\begin{equation}\label{matrix}
D=\begin{pmatrix}d_1 & 0\\ 0 & d_2\end{pmatrix},~J_{U}=\begin{pmatrix}f_u & f_v\\ g_u & g_v\end{pmatrix},~J_{\bar{U}}=\begin{pmatrix}f_{\bar{u}} & f_{\bar{v}}\\ g_{\bar{u}} & g_{\bar{v}}\end{pmatrix},
\end{equation}
and $\bar{\phi}=\dfrac{1}{|\Om|}\ds\int_{\Om}\phi(x)dx,~\bar{\psi}=\dfrac{1}{|\Om|}\ds\int_{\Om}\psi(x)dx$.
On the other hand, the linearized equation of the localized system \eqref{equ-mainlocal} at $(u_*,v_*)$ is
\begin{equation}\label{equ-linear2}
\begin{pmatrix}\phi_t \\ \psi_t\end{pmatrix}= D\begin{pmatrix}\Delta \phi \\ \Delta \psi\end{pmatrix}+(J_{U}+J_{\bar{U}})\begin{pmatrix}\phi \\ \psi\end{pmatrix}.
\end{equation}

By using Fourier series, we have the following results regarding the eigenvalues of linearized systems \eqref{equ-linear} and
\eqref{equ-linear2}. The proof is similar to the one of \cite[Lemma 4.1]{Tian2018}, thus we omit the proof  here.
\begin{lemma}\label{lemma-Jacob}
Let $\la_i$ be  eigenvalues of \eqref{eig0} and let $\varphi_i$ be the corresponding eigenfunctions for $i\in \N_0$.  Define
\begin{equation}\label{equ-Jn}
J_0=\tilde{J}_0=J_{U}+J_{\bar{U}}, \;\;
J_i=-\la_i D+J_{U}, \;\;
\tilde{J}_i=-\la_i D+J_{U}+J_{\bar{U}}, \;\; i\in \N,
\end{equation}
then we have
\begin{description}
\item (i) if $\mu$ is an eigenvalue of \eqref{equ-linear} (or \eqref{equ-linear2}), then there exists $i\in\mathbb{N}_0$ such that $\mu$ is an eigenvalue of $J_i$ (or $\tilde{J}_i$) with the associated eigenvector $(a_i,b_i)\varphi_i$ (or $(\tilde{a}_i,\tilde{b}_i)\varphi_i$) which is not identically zero;
\item (ii) the local stability of the constant steady state $(u_*,v_*)$ is determined by the eigenvalues of $J_i$ (or $\tilde{J}_i$) for $i\in\mathbb{N}_0$; to be more specific: $(u_*,v_*)$ is locally asymptotically stable with respect to \eqref{equ-linear} (or \eqref{equ-linear2}) when all the eigenvalues of $J_i$ (or $\tilde{J}_i$) have negative real parts, and it is unstable with respect to \eqref{equ-linear} (or \eqref{equ-linear2}) when there exist some $i\in\mathbb{N}_0$ such that $J_i$ (or $\tilde{J}_i$) has at least one eigenvalue with positive real part.
\end{description}
\end{lemma}

Lemma \ref{lemma-Jacob} reduces the stability with respect to PDE model \eqref{equ-main} or \eqref{equ-mainlocal} to the stability of infinitely many $2\times 2$ matrices $J_i$ or $\tilde{J}_i$, which can be determined by the trace ($Tr$) and determinant ($Det$) of the matrix:
\begin{equation}\label{TD}
  T_i=Tr(J_i), \;\; D_i=Det(J_i), \;\; \tilde{T}_i=Tr(\tilde{J}_i), \;\; \tilde{D}_i=Det(\tilde{J}_i), \;\; i\in \N_0.
\end{equation}
For the convenience of later discussion, we present $T_i$ and $D_i$ as continuous functions of $p$ here:
\begin{equation}\label{TD2}
T(p)=f_u+g_v-(d_1+d_2)p,~D(p)=d_1d_2p^2-(d_1g_v+d_2f_u)p+f_ug_v-g_uf_v.
\end{equation}
And we define
\begin{equation}\label{equ-Delta}
\Delta=(d_1g_v+d_2f_u)^2-4d_1d_2(f_ug_v-g_uf_v),
\end{equation}
and denote the roots of $T(p)$ and $D(p)$ (when $\Delta>0$) as
\begin{equation}\label{equ-TDroots}
p_*=\dfrac{f_u+g_v}{d_1+d_2},~p_{\pm}=\dfrac{(d_1g_v+d_2f_u)\pm\sqrt{\Delta}}{2d_1d_2},
\end{equation}
note that $T(\la_i)=T_i$ and $D(\la_i)=D_i$  for $i\in\mathbb{N}_0$.

To state a general criterion for the pattern formation of system \eqref{equ-main}, we recall some definitions and results about real-valued square matrices, which will help us to determine the stability of the constant steady state $(u_*,v_*)$. Denote $M_n(\mathbb{R})$ as the set of all $n\times n$ real matrices for $n\ge 2$, then we introduce the following definitions for the stability/instability of a real-valued matrix.

\begin{definition}\label{def-instability}
Let $A,~D\in M_n(\mathbb{R})$, and assume that $D$ is diagonal with positive entries. For $p\ge 0$, we denote the eigenvalues of $A-p D$ by $\mu_j(p)$ for each $1\leq j\leq n$.
\begin{description}
 \item (i) $A$ is \textbf{stable} if $\mathcal{R}e(\mu_j(0))<0$ for all $1\leq j\leq n$;
 \item (ii) $A$ is \textbf{strongly stable} if $\mathcal{R}e(\mu_j(p))<0$ for all $1\leq j\leq n$ and $p>0$, that is $A-p D$ is stable for all $p>0$;
 \item (iii) $A$ has \textbf{steady state instability }if $A$ is stable and there exists $p>0$ such that $\mu_j(p)>0$ for some $1\leq j\leq n$;
 \item (iv) $A$ has \textbf{wave instability} if $A$ is stable and there exists $p>0$ such that $\mu_j(p)=\alpha+i\beta$ with $\alpha>0$ and $\beta\ne 0$ for some $1\leq j\leq n$.
 \end{description}
\end{definition}
When applying these definitions to the linearized system of \eqref{equ-main} for some $J_i$ with $i\ge 1$ and diffusion matrix $D$, spatial or spatiotemporal patterns could emerge if $J_i$ is unstable. The steady state instability corresponds to generation of mode-$i$ spatial patterns through a symmetry-breaking bifurcation of spatially non-constant steady states, and the wave instability corresponds to creation of mode-$i$ time-periodic spatiotemporal patterns through a symmetry-breaking Hopf bifurcation of spatially non-constant periodic orbits. Indeed the roots $p_*,~p_{\pm}$  in \eqref{equ-TDroots} define two intervals of wave-number for pattern formation: steady state wave number
\begin{equation}\label{IS}
  I_S=\{p>0:Det(J_U-pD)<0\}=(p_{-},p_{+})\cap(0,+\infty),
\end{equation}
and cycle wave number
\begin{equation}\label{IH}
  I_H=\{p>0:Tr(J_U-pD)>0, \; Det(J_U-pD)>0\}=(0,p_*)\backslash [p_{-},p_{+}].
\end{equation}
A mode-$i$ steady state pattern may exist if $\la_i\in I_S$, and a mode-$i$ periodic orbit may exist if $\la_i\in I_H$.

We have the following classification results on the possible instability occurring in \eqref{equ-main}.

\begin{theorem}\label{theorem-instability} Suppose that $(u_*,v_*)$ is a constant steady state of \eqref{equ-main}.
Let $J_U,~J_{\bar{U}},~\Delta$ be defined in \eqref{matrix},\eqref{equ-Delta}, and let $p_*,~p_{\pm}$ be defined in \eqref{equ-TDroots}. We denote the two intervals of wave-number for pattern formation by  $I_S$ and  $I_H$ as in \eqref{IS} and \eqref{IH}. Suppose that $J_U+J_{\bar{U}}$ is stable and $J_U$ is not strongly stable, then we have the following scenarios for the pattern formation of system \eqref{equ-main} from the stability of matrix $J_U-pD$ (based on the assumption that the spatial domain is properly chosen):
\begin{description}
\item (i) $Det(J_U)<0$ and $Tr(J_U)<0$:  the steady state instability may occur but not the wave instability with $I_{S}=(0,p_+)$;
\item (ii) $Det(J_U)>0$ and $Tr(J_U)>0$: (a) if $\Delta\leq0$, or $\Delta>0$ and $d_1g_v+d_2f_u<0$, the wave instability may occur but not the steady state instability with $I_{H}=(0,p_*)$; (b) if $\Delta>0$, $d_1g_v+d_2f_u>0$ and $p_{*}>p_{+}$, both the  wave and the steady state instability may occur with $I_{S}=(p_-,p_+)$ and $I_{H}=(0,p_-)\cup(p_+,p_*)$; (c) if $\Delta>0$, $d_1g_v+d_2f_u>0$ and $p_{-}<p_{*}<p_{+}$, both the  wave and the steady state instability may occur with $I_{S}=(p_-,p_+)$ and $I_{H}=(0,p_-)$; (d) if $\Delta>0$, $d_1g_v+d_2f_u>0$ and $p_{*}<p_{-}$, both the  wave and the steady state instability may occur with $I_{S}=(p_-,p_+)$ and $I_{H}=(0,p_*)$;
\item (iii) $Det(J_U)<0$ and $Tr(J_U)>0$: (a) if $p_*\leq p_+$, the steady state instability may occur but not the wave instability with $I_{S}=(0,p_+)$; (b) if $p_*>p_+$, both the wave and the steady state instability may occur with $I_{S}=(0,p_+)$ and $I_{H}=(p_+,p_*)$;
\item (iv) $Det(J_U)>0$ and $Tr(J_U)<0$: (a) if $\Delta\leq0$, or $\Delta>0$ and $d_1g_v+d_2f_u<0$, neither the steady state nor the wave instability occurs; (b) if $\Delta>0$ and $d_1g_v+d_2f_u>0$,  the steady state instability may occur but not the wave instability with $I_{S}=(p_-,p_+)$.
\end{description}
\end{theorem}

\begin{proof}
According to the values of $Det(J_U)$ and $Tr(J_U)$, we discuss the possible bifurcation scenarios shown in Fig. \ref{fig-2}.

For case (i), that is, $Det(J_U)<0$ and $Tr(J_U)<0$, we see that $T(p)<0$ holds for all $p>0$, thus the wave instability is impossible. The function $D(p)$ is  quadric in $p$, and as $Det(J_U)<0$, it has a unique positive root $p_+$. If there exists some $\la_i\in(0,p_+)$, then steady instability may occur, and it is clear that the steady state wave number interval $I_S=(0,p_+)$, the situation is demonstrated in Fig. \ref{fig-2} (i).

When it comes to case (ii), that is, $Det(J_U)>0$ and $Tr(J_U)>0$, the situation is more complicated. First, if either $\Delta\leq0$, or $\Delta>0$ and $d_1g_v+d_2f_u<0$ holds, then from Fig. \ref{fig-2} (ii-a1) and (ii-a2), we can see that $D(p)$ has no positive roots, thus the steady state instability cannot occur. However, in both situations, $T(p)$ has a positive root $p_*$, thus the wave instability is possible and the cycle wave number interval is $I_H=(0,p_*)$. If $\Delta>0$ and $d_1g_v+d_2f_u>0$ holds, $D(p)$ has two positive roots $p_{\pm}$, thus the steady state instability is possible and the steady state wave number interval is $I_S=(p_-,p_+)$. Though the wave instability can still occur, but the cycle wave number interval will be influenced by the distribution of $p_+$ and $p_*$: if $p_*>p_+$, that is the situation in Fig. \ref{fig-2} (iib), we have $I_{H}=(0,p_-)\cup (p_+,p_*)$; if $p_{-}<p_*<p_+$, that is the situation in Fig. \ref{fig-2} (iic), now $I_{H}=(0,p_-)$; and if $p_*<p_-$ (see Fig. \ref{fig-2} (iid)),  we have $I_{H}=(0,p_*)$.

For case (iii), that is, $Det(J_U)<0$ and $Tr(J_U)>0$. It is clear that both $D(p)$ and $T(p)$ have a unique positive root. When $p_*< p_+$ (see Fig. \ref{fig-2} (iii-a)), we can see that only the steady state instability can occur with $I_S=(0,p_+)$; when $p_*>p_+$ (see Fig. \ref{fig-2} (iii-b)), we see that both the wave and the steady state instability may occur with $I_{S}=(0,p_+)$ and $I_{H}=(p_+,p_*)$.

Finally, for case (iv), when $Det(J_U)>0$ and $Tr(J_U)<0$, if $\Delta\leq0$, or $\Delta>0$ and $d_1g_v+d_2f_u<0$ is satisfied, then both $T(p)$ and $D(p)$ have no positive roots, thus the constant equilibrium $(u_*,v_*)$ stays  stable and no instability occurs (see the demonstration in Fig.  \ref{fig-2} (iv-a1) and  (iv-a2)); if $\Delta>0$ and $d_1g_v+d_2f_u>0$ (see Fig. \ref{fig-2}), $D(p)$ has two positive roots, thus the steady state instability may occur for $I_S=(p_-,p_+)$.
\end{proof}

\begin{figure}[htp]
\centering
\subfloat[(i)]{\includegraphics[scale=0.4]{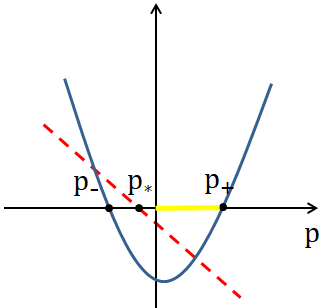}}
\subfloat[(ii-a1)]{\includegraphics[scale=0.4]{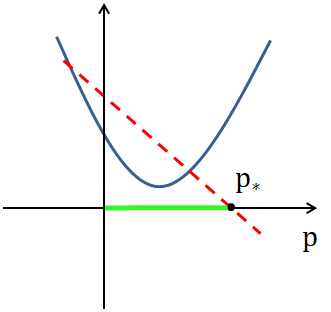}}
\subfloat[(ii-a2)]{\includegraphics[scale=0.4]{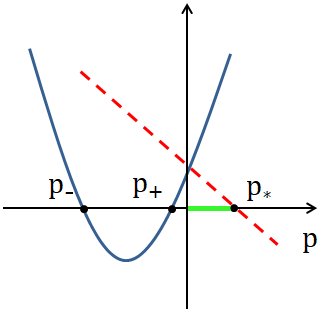}}
\subfloat[(ii-b)]{\includegraphics[scale=0.4]{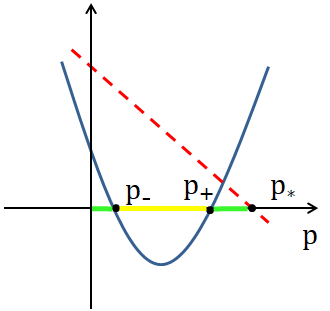}}\quad
\subfloat[(ii-c)]{\includegraphics[scale=0.4]{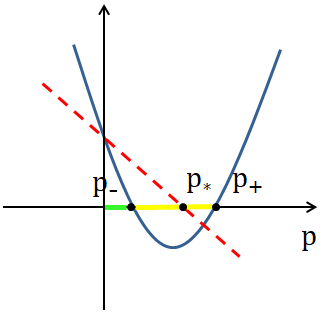}}
\subfloat[(ii-d)]{\includegraphics[scale=0.4]{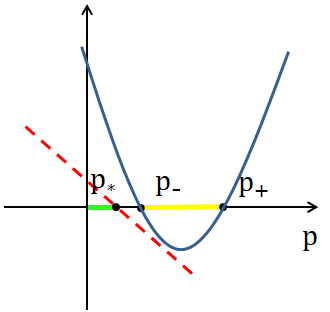}}
\subfloat[(iii-a)]{\includegraphics[scale=0.4]{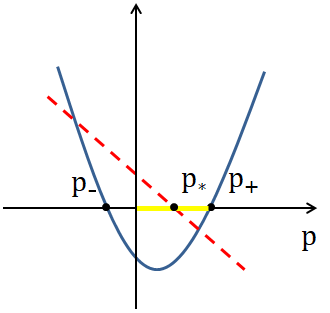}}
\subfloat[(iii-b)]{\includegraphics[scale=0.4]{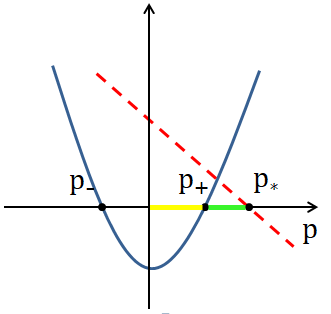}}\quad
\subfloat[(iv-a1)]{\includegraphics[scale=0.4]{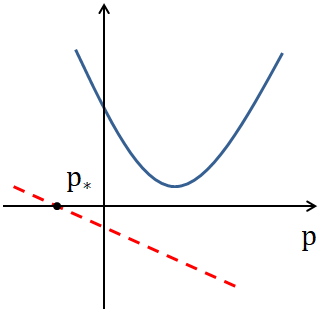}}
\subfloat[(iv-a2)]{\includegraphics[scale=0.4]{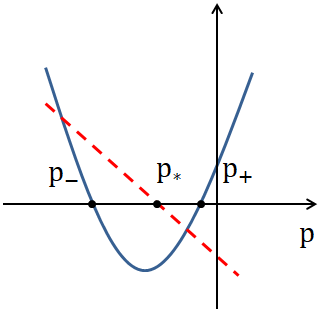}}
\subfloat[(iv-b)]{\includegraphics[scale=0.4]{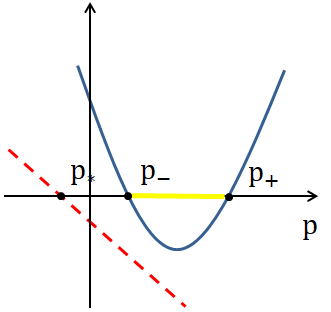}}\quad
\caption{The demonstration for the possible scenarios of the pattern formation in system \eqref{equ-main}. In each figure, $T(p)$ and $D(p)$ are described by blue solid curve and red dashed line, respectively. And, on the horizontal axis, the interval $I_S$ is marked by yellow color and $I_H$ is marked by green color. \label{fig-2}}
\end{figure}

As a comparison, we recall the classical Turing diffusion-induced instability result for a standard two-species reaction-diffusion system:
\begin{equation}\label{equ-main2}
\begin{cases}
u_{t}=d_1\Delta u+f(u,v,r), & x\in\Om,\ t>0,\\
v_{t}=d_2\Delta v+g(u,v,r), & x\in\Om,\ t>0,\\
\partial_{\nu}u=\partial_{\nu}v=0, & x\in\partial\Om,\ t>0,
\end{cases}
\end{equation}
and we use the same notation as above (or simply assuming $f,g$ are independent of $\bar{u},\;\bar{v}$), then we have the following results (as Turing \cite{Turing1952}).
 \begin{theorem}\label{theorem-instability2} Suppose that $(u_*,v_*)$ is a constant steady state of \eqref{equ-main2}.
Let $J_U,~\Delta$ be defined in \eqref{matrix},\eqref{equ-Delta}. Suppose that $J_U$ is stable (so $Det(J_U)>0$ and $Tr(J_U)<0$) and $J_U$ is not strongly stable, then  (a) if $\Delta\leq0$, or $\Delta>0$ and $d_1g_v+d_2f_u<0$, neither the steady state nor the wave instability occurs; (b) if $\Delta>0$ and $d_1g_v+d_2f_u>0$,  the steady state instability may occur but not the wave instability with $I_{S}=(p_-,p_+)$.
\end{theorem}

The proof of Theorem \ref{theorem-instability2} is similar to that of Theorem \ref{theorem-instability} so it is omitted. Comparing these two results, one can see  that only the case $(iv)$ in Theorem \ref{theorem-instability} occurs for   Theorem \ref{theorem-instability2}, so the system with spatial average \eqref{equ-main} allows more possible pattern formation scenarios  than the classical reaction-diffusion system \eqref{equ-main2}. Also Theorem \ref{theorem-instability2} (and indeed Turing \cite{Turing1952}) shows that the wave stability is not possible for the classical two-species reaction-diffusion system \eqref{equ-main2}, but it is possible for the two species reaction-diffusion system with spatial average \eqref{equ-main}.

\begin{remark}\label{rem3}
\begin{enumerate}
  \item The conditions in Theorem \ref{theorem-instability} are necessary for pattern formation but not sufficient: These conditions  determine if $I_S$ or $I_H$ is non-empty, but whether the interval $I_S$ or $I_H$ contains eigenvalues $\la_i$ depends on the spatial domain $\Om$. When $I_S$ or $I_H$ is non-empty, one can rescale the domain $\Om$ through a dilation $\Om\mapsto l\Om:=\{lx:x\in\Om\}$ for $k>0$, then $I_S$ or $I_H$ must contain some eigenvalue $\la_i(l\Om)=l^{-2}\la_i(\Om)$ if $l$ is sufficiently large so all instability described in Theorem \ref{theorem-instability} can be achieved for the dilated domain $l\Om$.
  \item Results in Theorem \ref{theorem-instability} are stated for a fixed diffusion matrix $D$, but varying $D=diag(d_1,d_2)$ will change the value of $p_{\pm}$, $p_*$, $\Delta$ and $d_1g_v+d_2f_u$, which determine the type of instability in case $(ii)$, $(iii)$ and $(iv)$.
  \item A more detailed result of bifurcation of non-constant steady states or periodic orbits like Theorem \ref{theorem-scalarbifurcation} can also be stated for system \eqref{equ-main} by using either diffusion coefficients $d_1,~d_2$, or kinetic parameter $r$, or domain scaling parameter $l$ as the bifurcation parameter. But it is too tedious to state the results for every case in Theorem \ref{theorem-instability} so we will not give the whole list. Instead we demonstrate such detailed bifurcation analysis through two specific examples: cooperative Lotka-Volterra model (case $(i)$) and Rosenzweig-MacArthur predator-prey model (case $(ii)$) in the following sections.
\end{enumerate}
\end{remark}

\section{A nonlocal two-species cooperative Lotka-Volterra model}

In this section, we show that the spatial average can induce spatial patterns in a diffusive cooperative Lotka-Volterra system with nonlocal competition in one of the species. Here for simplicity, we assume that the spatial dimension $m=1$ and $\Om=(0,l\pi)$ for $l>0$, and the corresponding eigenvalues/eigenfunctions for the diffusion operator are $\la_j=j^2/l^2$ and $\varphi_j(x)=\cos(jx/l)$. Note that $l$ is a scaling parameter for the spatial domain as in Remark \ref{rem3}. The model on $(0,l\pi)$ is
\begin{equation}\label{equ-coopnonlocal}
\begin{cases}
u_{t}=\beta u_{xx}+u\left(1-\dfrac{a}{l\pi}\ds\int_0^{l\pi}u(x,t)dx+bv\right), & x\in(0,l\pi),\ t>0,\\
v_{t}=v_{xx}+v\left(1+cu-dv\right), & x\in(0,l\pi),\ t>0,\\
u_{x}(0,t)=u_{x}(l\pi,t)=0,~v_{x}(0,t)=v_{x}(l\pi,t)=0, & t>0,\\
u(x,0)=u_0(x)\ge0,~~v(x,0)=v_0(x)\ge0, & x\in[0,l\pi].
\end{cases}
\end{equation}
Here $u(x,t)$ and $v(x,t)$ are the densities of two cooperating populations, and all parameters $a,b,c,d,\beta$ are positive.

Before our study for the nonlocal system \eqref{equ-coopnonlocal}, first we give a brief description for its corresponding local system:
\begin{equation}\label{equ-cooplocal}
\begin{cases}
u_{t}=\beta u_{xx}+u\left(1-au+bv\right), & x\in(0,l\pi),\ t>0,\\
v_{t}=v_{xx}+v\left(1+cu-dv\right), & x\in(0,l\pi),\ t>0,\\
u_{x}(0,t)=u_{x}(l\pi,t)=0,~v_{x}(0,t)=v_{x}(l\pi,t)=0, & t>0,\\
u(x,0)=u_0(x)\ge0,~~v(x,0)=v_0(x)\ge0, & x\in[0,l\pi].
\end{cases}
\end{equation}
It is clear that system \eqref{equ-cooplocal} has three unstable constant equilibria: $(0,0),~(0,1/d),~(1/a,0)$
and a unique positive constant equilibrium $(u_*,v_*)$ which is locally asymptotically stable with
\begin{equation}\label{equ-coopequilibrium}
u_*=\dfrac{d+b}{ad-bc},~v_*=\dfrac{a+c}{ad-bc},
\end{equation}
when $ad-bc>0$ is satisfied. Furthermore, the global stability of $(u_*,v_*)$ with respect to \eqref{equ-cooplocal} for all $\beta>0$ can be obtained by the monotone dynamical systems theory or Lyapunov method \cite{Takeuchi1996,Smith1995}. It is also known that if \eqref{equ-cooplocal} has a stable equilibrium $(u(x),v(x))$ on a higher dimensional convex domain, then $(u(x),v(x))$ must be a constant one \cite{Kishimoto1985}.

For the nonlocal system \eqref{equ-coopnonlocal}, the linearization  at $(u_*,v_*)$ gives
\begin{equation}\label{matrix2}
D=\begin{pmatrix}\beta & 0\\ 0 & 1\end{pmatrix},~J_{U}+J_{\bar{U}}=\begin{pmatrix}-au_* & bu_*\\ cv_* & -dv_*\end{pmatrix},~J_{U}=\begin{pmatrix}0 & bu_*\\ cv_* & -dv_*\end{pmatrix}.
\end{equation}
Then $J_{U}+J_{\bar{U}}$ is stable as $ad-bc>0$, and $J_U$ satisfies $Tr(J_U)<0$ and $Det(J_U)<0$ so this example belongs to the case $(i)$ in Theorem \ref{theorem-instability}.

Following Lemma \ref{lemma-Jacob}, we obtain the characteristic equation  with the diffusion ratio $\beta$ taken as a parameter:
\begin{equation}\label{equ-coopchara}
\mu^{2}-T_{j}(\beta)\mu+D_{j}(\beta)=0,~j\in\mathbb{N}_{0},
\end{equation}
where
\begin{equation*}
T_{0}(\beta)=au_*+dv_*,~D_{0}(\beta)=(ad-bc)u_*v_*,
\end{equation*}
and for $j\ge1$,
\begin{equation*}
T_{j}(\beta)=(\beta+1)\frac{j^2}{l^2}+dv_*,~D_{j}(\beta)=\beta\frac{j^4}{l^4}+\beta dv_*\frac{j^2}{l^2}-bcu_*v_*
\end{equation*}
with $u_*$ and $v_*$ defined by \eqref{equ-coopequilibrium}. By letting $p=\dfrac{j^2}{l^2}$, we define the trace and determinant functions by
\begin{equation}\label{equ-coopDn}
T(\beta,p)=(\beta+1)p+dv_*,~D(\beta,p)=\beta p^{2}+\beta dv_*p-bcu_*v_*.
\end{equation}
From \eqref{equ-coopchara}, we know that $T_j(\beta)>0$ holds for any $j\in\mathbb{N}_0$ and $D_0(\beta)>0$, while the sign of $D_j(\beta)$ could change which may lead to steady state instability in the system \eqref{equ-coopnonlocal} but not wave instability (see Theorem \ref{theorem-instability} case $(i)$).

The following lemma about the property of the root of $D(\beta,p)$ is easy to obtain.

\begin{lemma}\label{lemma-coop}
Let $D(\beta,p)$ be defined in \eqref{equ-coopDn}, then it has a unique positive zero $p=p_{\sharp}(\beta)$ such that $D(\beta,p)<0$ for $p\in(0,p_{\sharp}(\beta))$ and $D(\beta,p)>0$ for $p\in(p_{\sharp}(\beta), +\infty)$. Moreover, $p_{\sharp}(\beta)$ is strictly decreasing in $\beta>0$,  $\lim\limits_{\beta\rightarrow 0}p_{\sharp}(\beta)=\infty$ and $\lim\limits_{\beta\rightarrow\infty}p_{\sharp}(\beta)=0$.
\end{lemma}

\begin{proof}
The existence and uniqueness of $p_{\sharp}(\beta)$ is obvious as  $D(\beta,p)=0$ is quadric in $p$, $\beta>0$, $\beta dv_*>0$ and $-bcu_*v_*<0$. By taking derivative with respect to $\beta$ in  $D(\beta,p_{\sharp}(\beta))=0$, we obtain
\begin{equation*}
p_{\sharp}'(\beta)=-\frac{p_{\sharp}^2+dv_*p_{\sharp}}{\beta(2p_{\sharp}+dv_*)}<0.
\end{equation*}
Therefore $p_{\sharp}$ is strictly decreasing with respect to $\beta$ and the limits can be obtained by a direct calculation.
\end{proof}

Now we have the main result on the stability/instability of $(u_*,v_*)$ and bifurcation of non-constant solutions for system \eqref{equ-coopnonlocal}.

\begin{theorem}\label{theorem-coop}
For system \eqref{equ-coopnonlocal} with fixed parameters $a,~b,~c,~d,~l>0$ satisfying $ad-bc>0$, we have the following results about the stability and bifurcation of constant equilibrium $(u_*,v_*)$:
\begin{description}
\item (i) There exists a decreasing  sequence $\beta_{j}=\dfrac{bcu_*v_*}{\la_j(\la_j+dv_*)}$ with $\la_j=j^2/l^2$ such that system \eqref{equ-coopnonlocal} undergoes a steady state bifurcation at $\beta=\beta_j$ near $(u_*,v_*)$;
\item (ii)  $(u_*,v_*)$ is locally asymptotically stable for $\beta\in(\beta_1,\infty)$ and unstable for $\beta\in(0,\beta_1)$;
\item (iii) there exists a positive constant $\delta>0$ such that the set of non-constant steady state solutions $\Gamma_1$ of \eqref{equ-coopnonlocal} near $(\beta_1,u_*,v_*)$ has the form:
\begin{equation}\label{equ-Gamman}
\Gamma_1=\left\{\left(\beta_1(s), U(s,x), V(s,x)\right):\ -\delta<s<\delta\right\},
\end{equation}
where $U(s,x)=u_*+s\varphi_1(x)+sg_{1}(s,x), ~V(s,x)=v_*+sh\varphi_1(x)+sg_{2}(s,x)$,
and $\beta(s),~g_{i}(s,\cdot) (i=1,2)$ are smooth functions defined for $s\in (0,\delta)$ such that $\beta(0)=\beta_1$, and $g_{i}(0,\cdot)=0\ (i=1,2)$ and $\ds h=\frac{cv_*}{\la_1+dv_*}$;

\item (iv) $\beta_1'(0)=0,~\beta_1''(0)\not=0$, thus the bifurcation is of pitchfork type. Moreover, if $\beta_1''(0)>0$, the bifurcation is supercritical and the bifurcating steady states are unstable for $s\in(-\delta, \delta)$; and if $\beta_1''(0)<0$, the bifurcation is subcritical and the bifurcating steady states are locally asymptotically stable for $s\in(-\delta, \delta)$.
\end{description}
\end{theorem}

\begin{proof}
For part (i), the steady state bifurcation occurs at $\beta=\beta_j$ if there exist some $j\in\mathbb{N}$ such that $p_{\sharp}(\beta_j)=\la_j$ which is equivalent to $D_j(\beta_j)=0$. By Lemma \ref{lemma-coop}, $p_{\sharp}$ is strictly decreasing in $\beta>0$, thus we obtain a decreasing sequence $\beta_j$ such that system \eqref{equ-coopnonlocal} undergoes a steady state bifurcation at $\beta_j$. Part (ii) is a corollary of (i) since the equilibrium $(u_*,v_*)$ loses its stability at the first bifurcation value.

For part (iii) and (iv), we again use the abstract bifurcation theory in \cite{Rabinowitz1971,Shi1999}, which is similar to the  proof of Theorem \ref{theorem-scalarbifurcation}. The steady states of \eqref{equ-coopnonlocal} satisfy the following nonlocal elliptic system:
\begin{equation}\label{equ-coopsteady}
\begin{cases}
\beta u_{xx}+u\left(1-\dfrac{a}{l\pi}\ds\int_0^{l\pi}u(x)dx+bv\right)=0, & x\in(0,l\pi),\\
v_{xx}+v\left(1+cu-dv\right)=0, & x\in(0,l\pi),\\
u_{x}(0)=u_{x}(l\pi)=0,~v_{x}(0)=v_{x}(l\pi)=0.
\end{cases}
\end{equation}

We define a nonlinear mapping $G:\ \mathbb{R}^+\times X^2\rightarrow Y^2$ by
\begin{equation}\label{equ-G}
 G(\beta,u,v)=\left(\begin{array}{cccc}
\medskip
&\beta u_{xx}+u\left(1-\dfrac{a}{l\pi}\ds\int_{0}^{l\pi}udx+bv\right)\\
\medskip
&v_{xx}+v(1+cu-dv)\\
\end{array}\right).
\end{equation}
It is clear that $G(\beta,u_*,v_*)=0$.
We have
\begin{equation}\label{equ-tildeL}
\begin{aligned}
 &G_{(u,v)}\left(\beta_1, u_*,v_*\right)[\varphi,\psi]
 =\left(\begin{array}{cc}
\medskip
 &\beta_1\varphi_{xx}-\dfrac{au_*}{l\pi}\ds\int_{0}^{l\pi}\varphi dx+bu_*\psi\\
\medskip
 &\psi_{xx}+cv_*\varphi-dv_*\psi\\
\end{array}\right):=\tilde{L}.\end{aligned}
\end{equation}
Then the kernel is $\mathcal{N}(\tilde{L})=\textrm{Span}\left\{\tilde{q}=(1, h)\varphi_1 \right\}$ where $h=\ds\frac{cv_*}{\la_1+dv_*}$, thus $\dim\left({\mathcal{N}(\tilde{L})}\right)=1$. The range space of $\tilde{L}$ is $\mathcal{R}(\tilde{L})=\left\{(f_1,f_2)\in Y^2:\ \lan y,(f_1,f_2)\rangle=0\right\}$, where $y$ is defined by
\begin{equation*}
\lan y,(f_1,f_2)\rangle=\ds\int_0^{l\pi}\left(f_1+\frac{bu_*}{\la_1+dv_*}f_2\right)\varphi_1dx,
\end{equation*}
and thus $\textrm{codim}\left(\mathcal{R}(\tilde{L})\right)=1$. Also, from \eqref{equ-G}, we have
\begin{equation}\label{equ-Gderibeta}
 G_{\beta(u,v)}\left(\beta_1, u_*,v_*\right)[\tilde{q}]=(1,h)^T\Delta\varphi_1=-(1,h)^T\la_1\varphi_1\not\in\mathcal{R}(\tilde{L}),
\end{equation}
as $\ds\int_{\Omega}\left(1+\dfrac{\beta_1\la_1}{\la_1+dv_*}\right)\varphi^2_idx>0$. By applying Theorem 1.7 in \cite{Rabinowitz1971}, we obtain the result in part (iii).

Then we calculate $\beta'_1(0)$. From \cite{Jin2013}, we know that $\beta'_1(0)$ has the following form:
\begin{equation}\label{equ-beta'}
\aligned
\beta'_1(0)=&-\frac{\left\langle y,\ G_{(u,v)(u,v)}\left(\beta_1, u_*, v_*\right)[\tilde{q},\tilde{q}]\right\rangle}
 {2\left\langle y,\ G_{\beta(u,v)}\left(\beta_1, u_*,v_*\right)[\tilde{q}]\right\rangle}.
\endaligned\end{equation}
Also, from \eqref{equ-Gderibeta}, we have
\begin{equation*}
\left\langle y,\ G_{\beta(u,v)}\left(\beta_1, u_*,v_*\right)[\tilde{q}]\right\rangle=-\left(1+\dfrac{\beta_1\la_1}{\la_1+dv_*}\right)\int_{0}^{l\pi}\varphi_1^2dx=-\frac{l\pi}{2}\left(1+\dfrac{\beta_1\la_1}{\la_1+dv_*}\right),
\end{equation*}
and
\begin{equation}\label{equ-Guu}
\left\langle y,\ G_{(u,v)(u,v)}\left(\beta_1, u_*,v_*\right)[\tilde{q},\tilde{q}]\right\rangle\\
=\left[2bh\beta_{1}^{-1}+2h(c-dh)\frac{\beta_1\la_1}{h(\la_1+dv_*)}\right]\int_{0}^{l\pi}\varphi_1^3dx=0,
\end{equation}
since $\ds\int_{0}^{l\pi}\varphi_1^3dx=\int_{0}^{l\pi}\cos^3\left(\frac{x}{l}\right)dx=0$. Therefore, Theorem 1.7 in \cite{Rabinowitz1971} can be applied to obtain the existence of the branch of non-constant solutions $\Gamma_1$ as in \eqref{equ-Gamman}, and $\beta_1'(0)=0$.

Continuing to calculate $\beta_1''(0)$, which reads \cite{Jin2013},
\begin{equation}\label{equ-beta''}
\beta_1''(0)=-\dfrac{\left\langle y,\ G_{(u,v)(u,v)(u,v)}\left(\beta_1, u_*,v_*\right)[\tilde{q}, \tilde{q}, \tilde{q}]\right\rangle+3\left\langle y,\ G_{(u,v)(u,v)}\left(\beta_1, u_*,v_*\right)[\tilde{q},\Theta]\right\rangle}{3\left\langle y,\ G_{\beta(u,v)}\left(\beta_1, u_*,v_*\right)[\tilde{q}]\right\rangle}.
\end{equation}
First we have
$\left\langle y,\ G_{(u,v)(u,v)(u,v)}\left(\beta_1, u_*,v_*\right)[\tilde{q}, \tilde{q}, \tilde{q}]\right\rangle=0$
since all the third derivatives are zero.
Then,
\begin{equation*}
\left\langle y,\ G_{(u,v)(u,v)}\left(\beta_1, u_*,v_*\right)[\tilde{q},\Theta]\right\rangle=\frac{l\pi}{2}A+\frac{3l\pi}{8}B,
\end{equation*}
where
\begin{equation}\label{equ-AnBn}
\begin{aligned}
A&=\left[-a\Theta^1_0+bh(\Theta^1_0-\Theta^1_{2})+b(\Theta^2_0-\Theta^2_{2})\right]\\
&+\frac{\beta_1\la_{1}}{h(\la_1+dv_*)}\left[ch(\Theta^1_0-\Theta^1_{2})+(c-2dh)(\Theta^2_0-\Theta^2_{2})\right],\\
B&=\left[2bh\Theta^1_{2}+2b\Theta^2_{2}\right]+\frac{\beta_1\la_{1}}{h(\la_1+dv_*)}\left[2ch\Theta^1_{2}+2(c-2dh)\Theta^2_{2}\right],
\end{aligned}
\end{equation}
with
\begin{equation}\label{equ-Thetacoefficient}
\begin{aligned}
\Theta^{1}_0&=\frac{bh(-dhu_*+c_*+dv_*)}{(ad-bc)u_*v_*}, \;\; \Theta^{1}_{2}=-\frac{bh(-dhu_*+cu_*+dv_*+4\la_1)}{bcu_*v_*-4d\beta_1\la_1v_*-16\beta_1\la_1^2},\\
\Theta^{2}_0&=\frac{au_*(c-dh)+bcv_*}{(ad-bc)u_*v_*},\;\;
\Theta^{2}_{2}=-\frac{h(bcv_*+4c\beta_1\la_1-4dh\beta_1\la_1)}{bcu_*v_*-4d\beta_1\la_1v_*-16\beta_1\la_1^2}.
\end{aligned}
\end{equation}
Thus we have
\begin{equation}\label{equ-betaformula}
\beta_1''(0)=\frac{1}{\beta_1\la_1}\left(A+\frac{3}{4}B\right)=\frac{P}{Q},
\end{equation}
with
\begin{equation*}
\begin{aligned}
P=&6bcd^4v_*^5-4(ad-bc)d^4v_*^5-9ac^2d^2\la_1u_*^2v_*^2+10acd^3\la_1u_*v_*^3-37ad^4\la_1v_*^4\\
&+9bc^3d\la_1u_*^2v_*^2+2bc^2d^2\la_1u_*v_*^3+79bcd^3\la_1v_*^4-19ac^2d\la_1^2u_*^2v_*+20acd^2\la_1^2u_*v_*^2\\
&-87ad^3\la_1^2v_*^3+25bc^3\la_1^2u_*^2v_*+52bc^2d\la_1^2u_*v_*^2+153bcd^2\la_1^2v_*^3+30ac^2\la_1^3u_*^2\\
&+10acd\la^3u_*v_*-79ad^2\la_1^3v_*^2+50bc^2\la_1^3u_*v_*+109bcd\la_1^3v_*^2-25(ad-bc)\la_1^4v_*,\\
Q=&6(d^3v_*^3+7d^2\la_1v_*^2+11d\la^2v_*+5\la^3)(dv_*(ad-bc)+(ad-bc)\la_1)v_*u_*^2.
\end{aligned}
\end{equation*}
The assertion on the stability follows from the same way as the proof of Theorem \ref{theorem-scalarbifurcation}.
\end{proof}

\begin{remark}
In \cite{Jin2013}, a detailed bifurcation analysis for steady state bifurcation  is carried out in a regular reaction-diffusion system. Here, we give a calculation for bifurcation direction when spatial average is introduced into a reaction-diffusion system. The main difference lies in the calculation for the derivatives of nonlinear operator $G$. For instance, we see the first Fr\'echet derivative $G_{(u,v)}\left(\beta_1, u_*,v_*\right)[q]$ in \eqref{equ-Gderibeta}, because the integral $\ds\int_{0}^{l\pi}\varphi dx=0$, so the parameter $a$ does not play role in determination of bifurcation direction (the similar for the second derivative \eqref{equ-Guu}). However, if we replace the nonlocal term with a local one, parameter $a$ will certainly affects the direction of steady state bifurcation.
\end{remark}

We do not have a more definite conclusion on the sign of  $\beta_1''(0)$ in \eqref{equ-betaformula} due to the complex form of $P$ and $Q$, but for a given set of parameters $a,~b,~c,~d,~l$, it can be calculated. For example, when the parameters in \eqref{equ-coopnonlocal} are
\begin{equation}\label{equ-coopparameter}
a=1,~b=0.1,~c=0.1,~d=1,~l=1,
\end{equation}
we can compute that $\beta''_1(0)=-1.6759<0$ from \eqref{equ-betaformula}, thus the pitchfork bifurcation is subcritical and the bifurcating non-constant steady states are locally asymptotically stable near $\beta=\beta_1$. Here, we plot the graph of $D(\beta,p)=0$ in $(\beta,p)$ plane (see Fig. \ref{Fig-coopdiagram}), and the steady state bifurcation points are
\begin{equation*}
\beta_1=0.00585,~\beta_2=0.000605.
\end{equation*}

\begin{figure}[htp]
\centering
\includegraphics[scale=0.4]{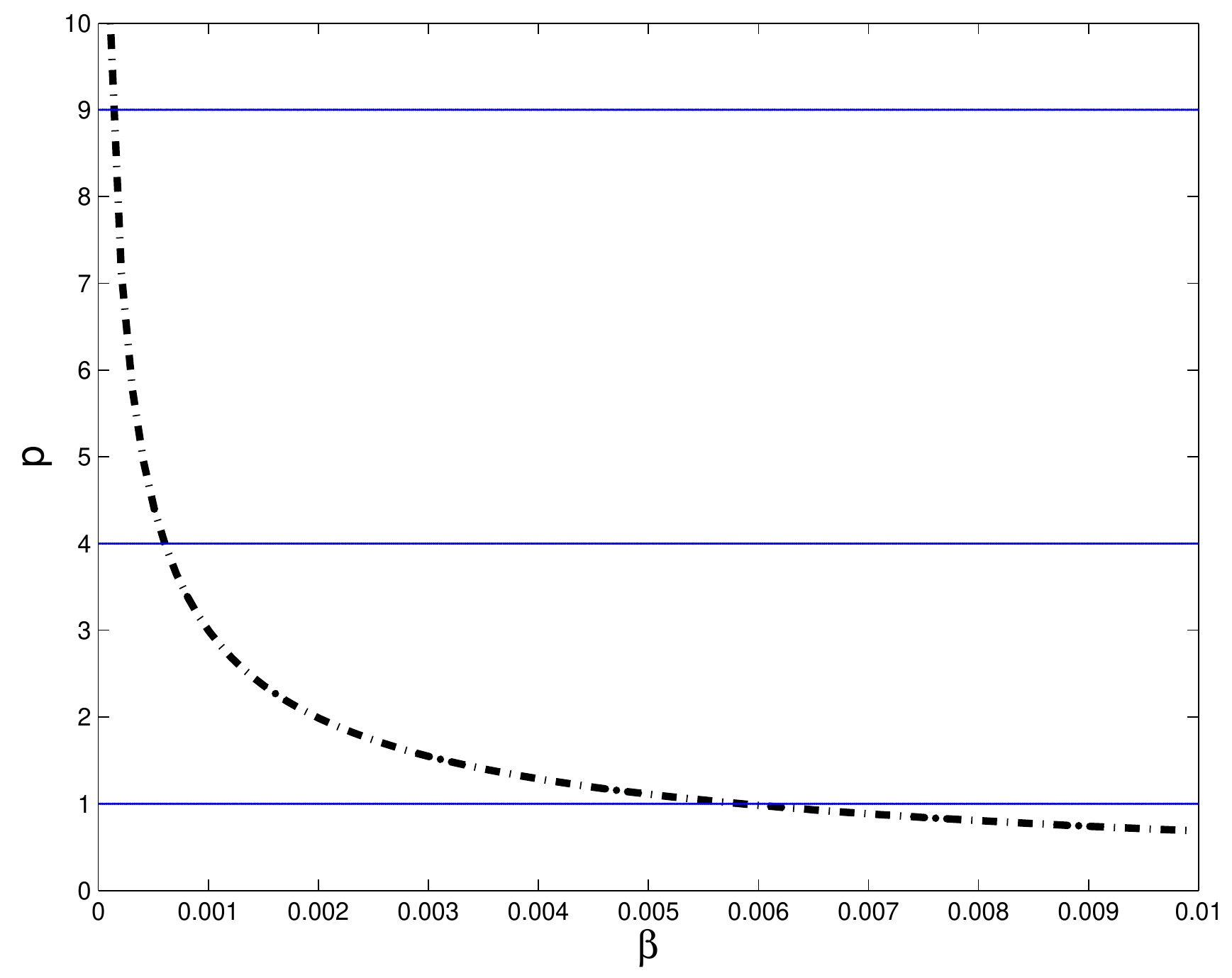}
 \caption{The plot of $D(\beta,p)=0$ (black dash-doted curve) with the parameters from \eqref{equ-coopparameter}, and the blue solid horizontal lines are $p=j^2/l^2$ with $j\in\mathbb{N}$.
\label{Fig-coopdiagram}}
\end{figure}

\begin{figure}[htp]
\centering
\subfloat[species u]{\includegraphics[scale=0.45]{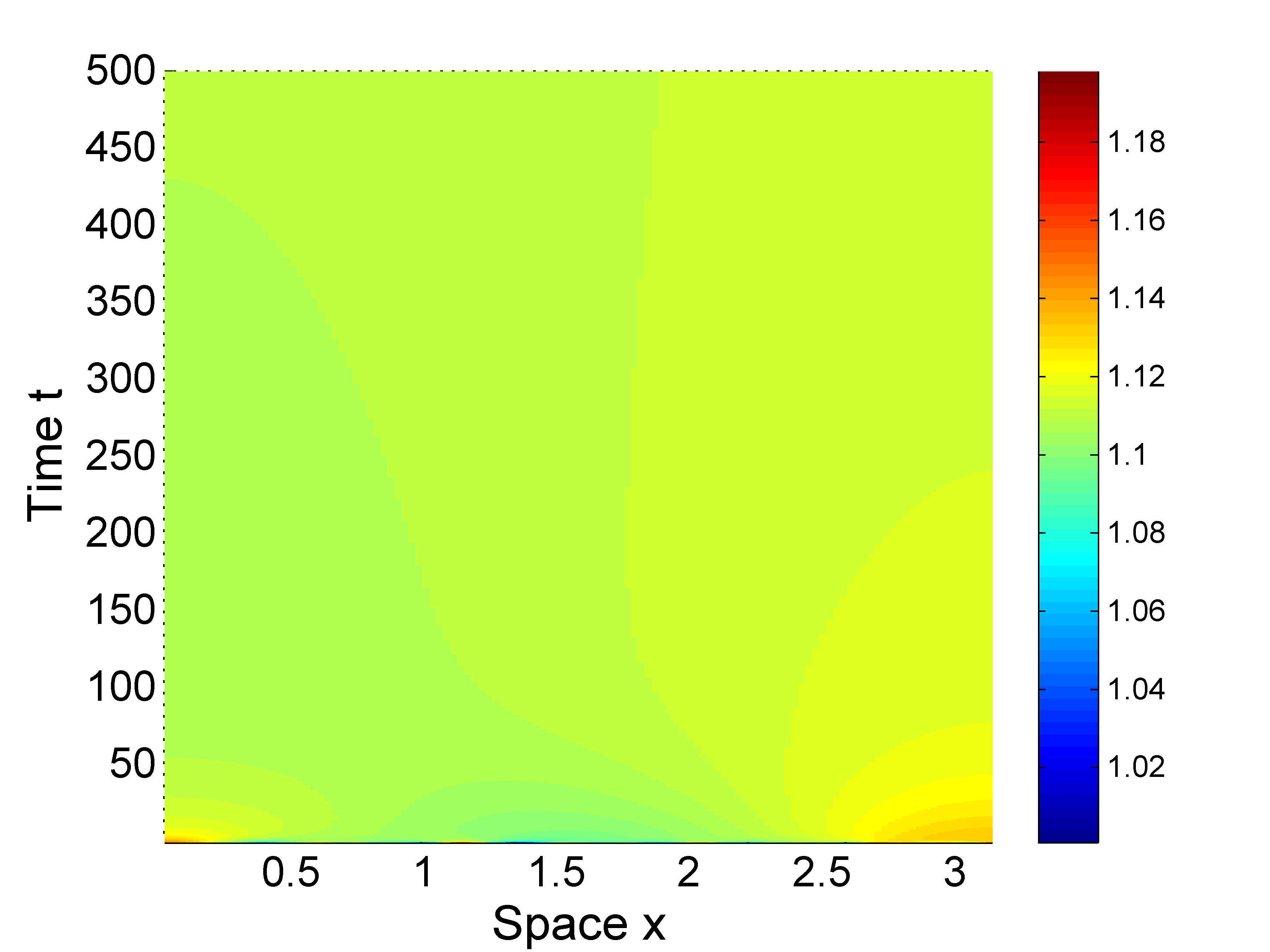}}
\subfloat[species v]{\includegraphics[scale=0.45]{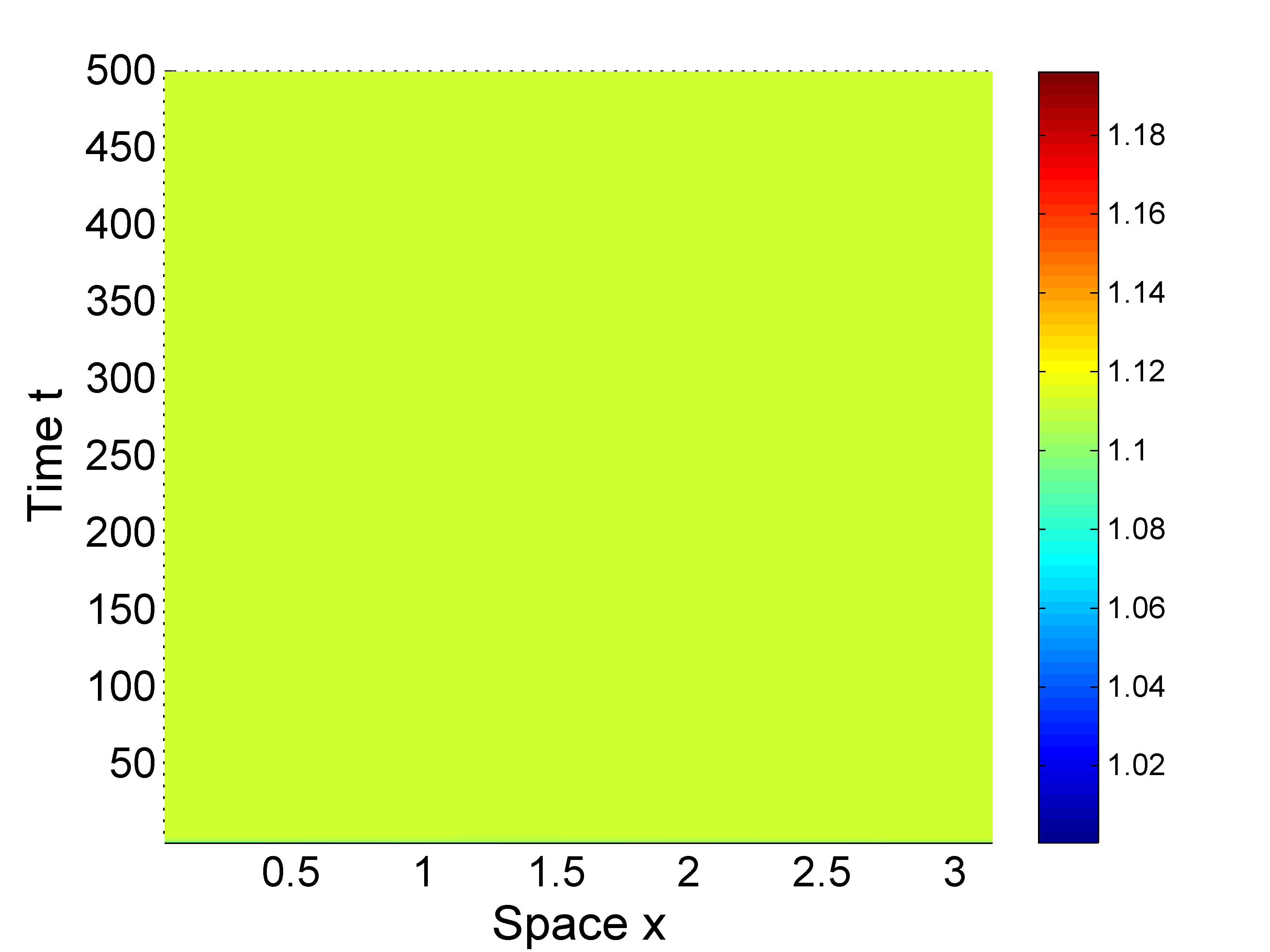}}\quad
\subfloat[species u]{\includegraphics[scale=0.45]{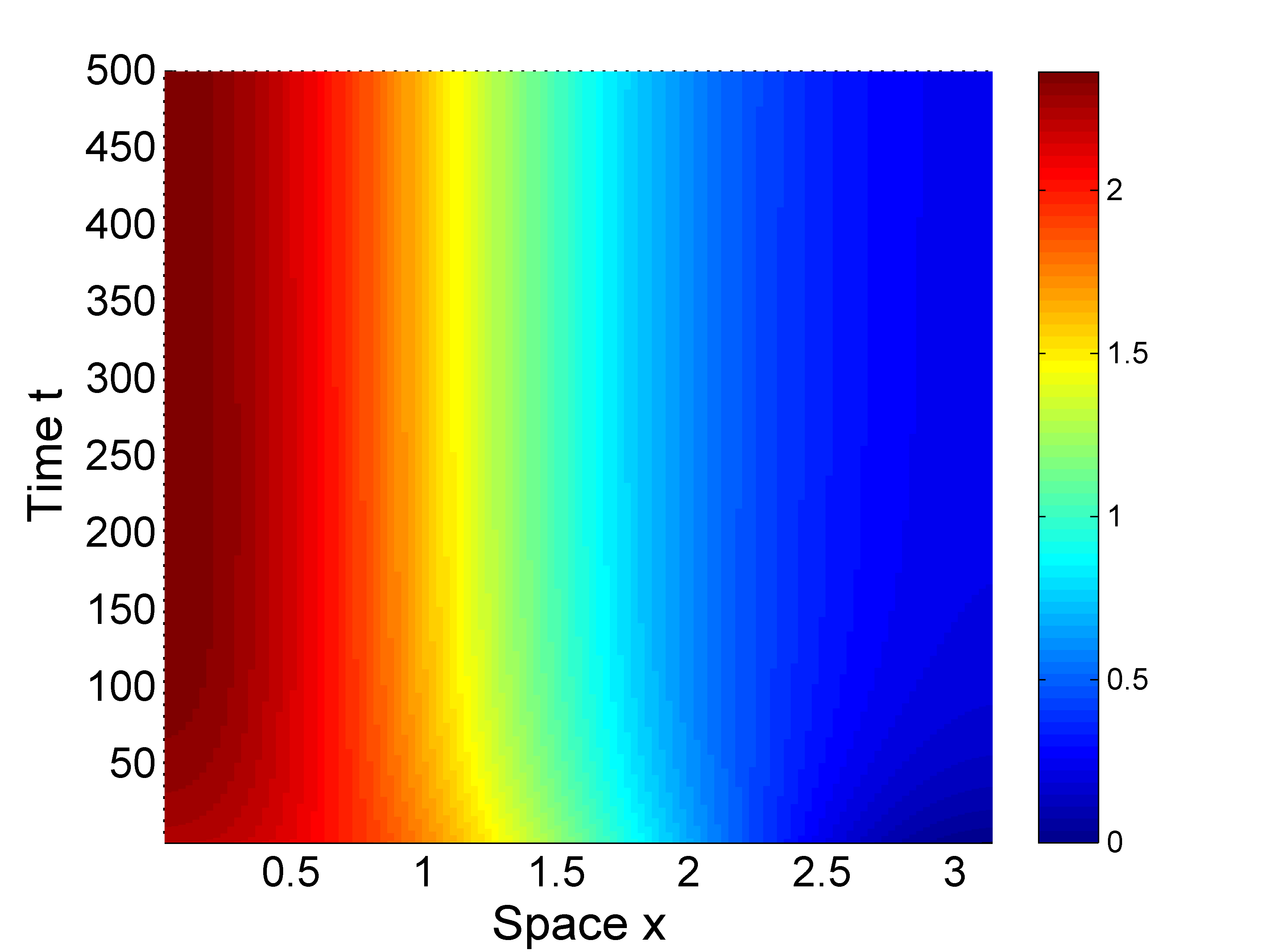}}
\subfloat[species v]{\includegraphics[scale=0.45]{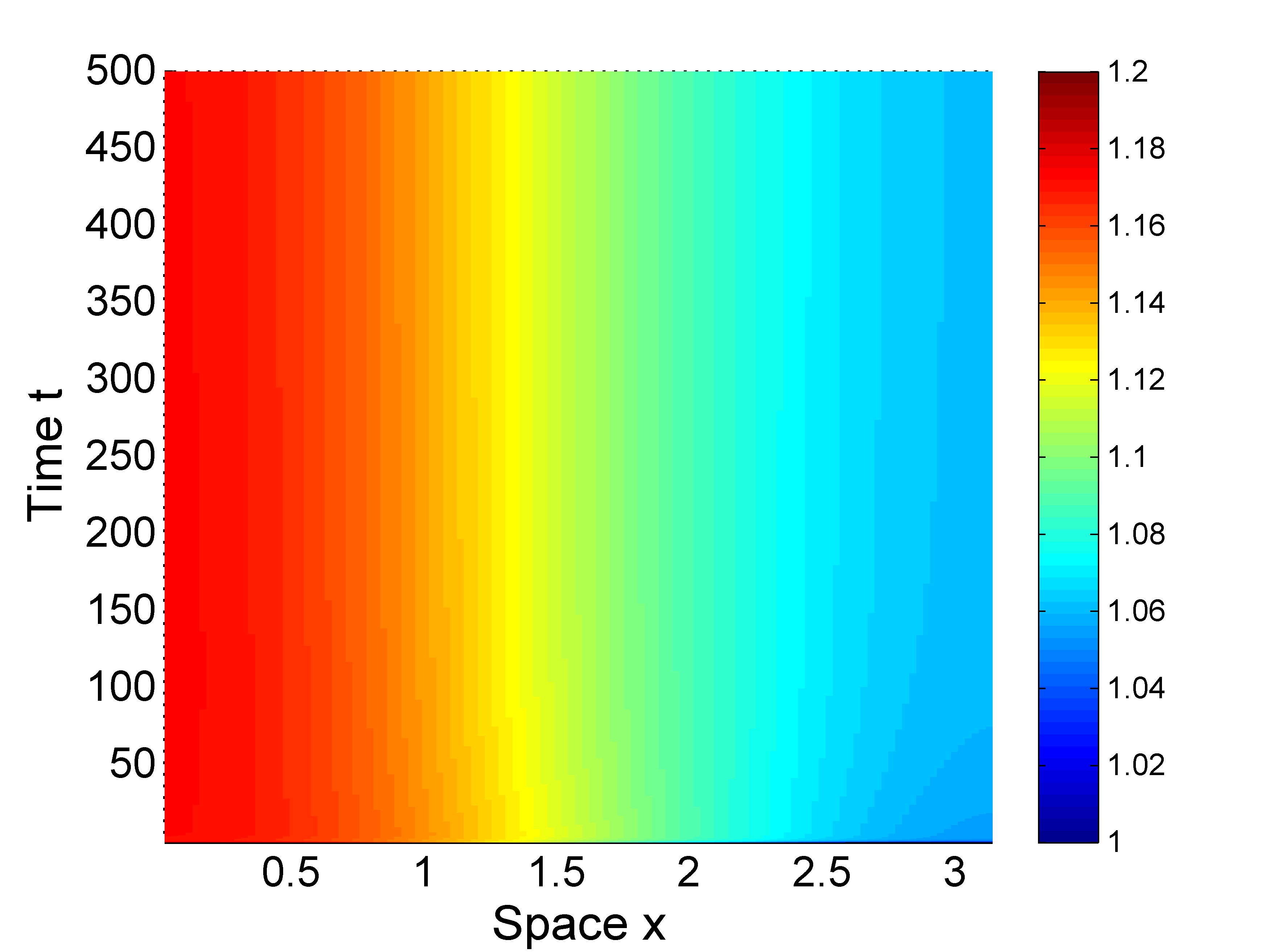}}\quad
\subfloat[species u]{\includegraphics[scale=0.45]{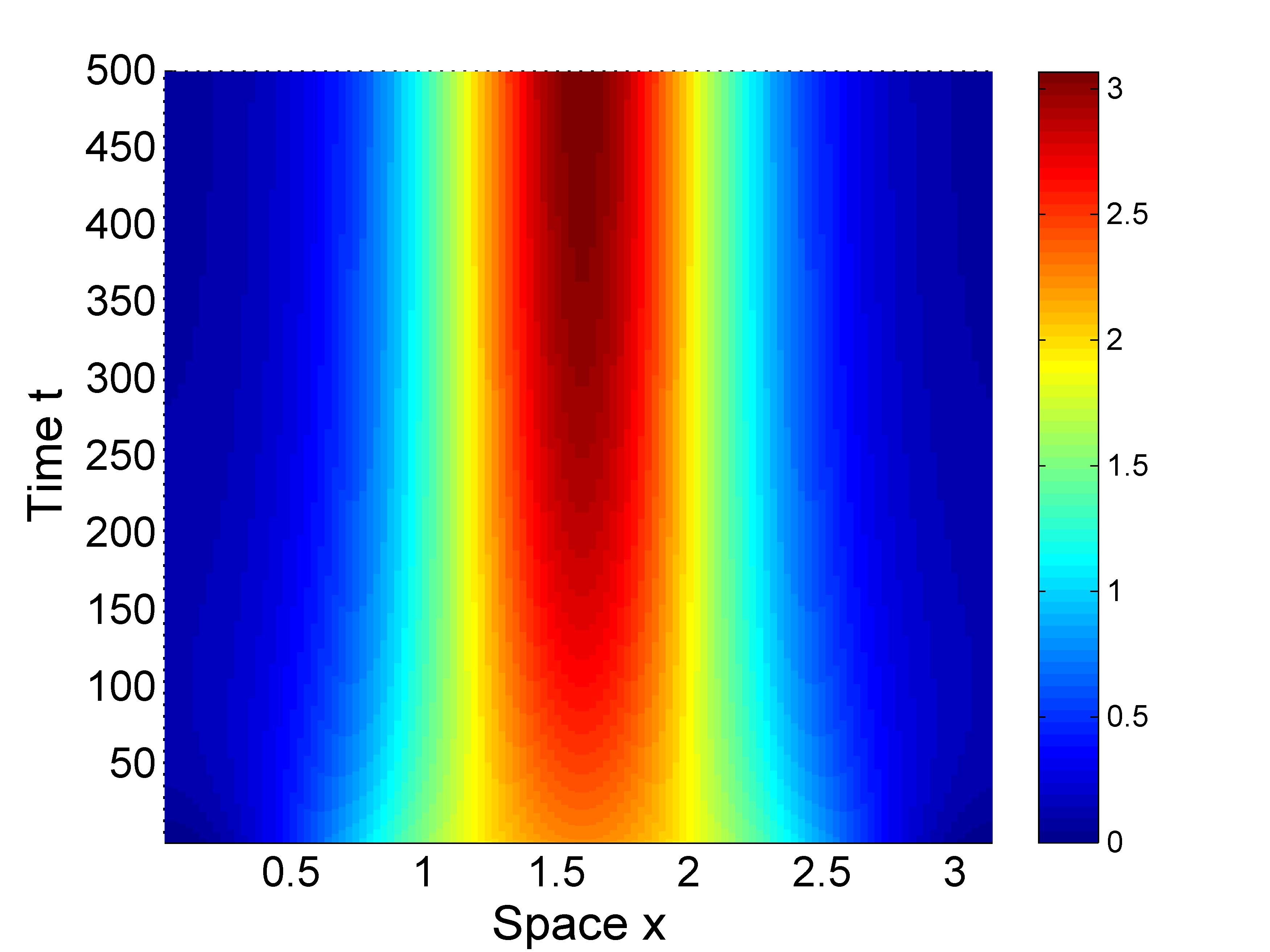}}
\subfloat[species v]{\includegraphics[scale=0.45]{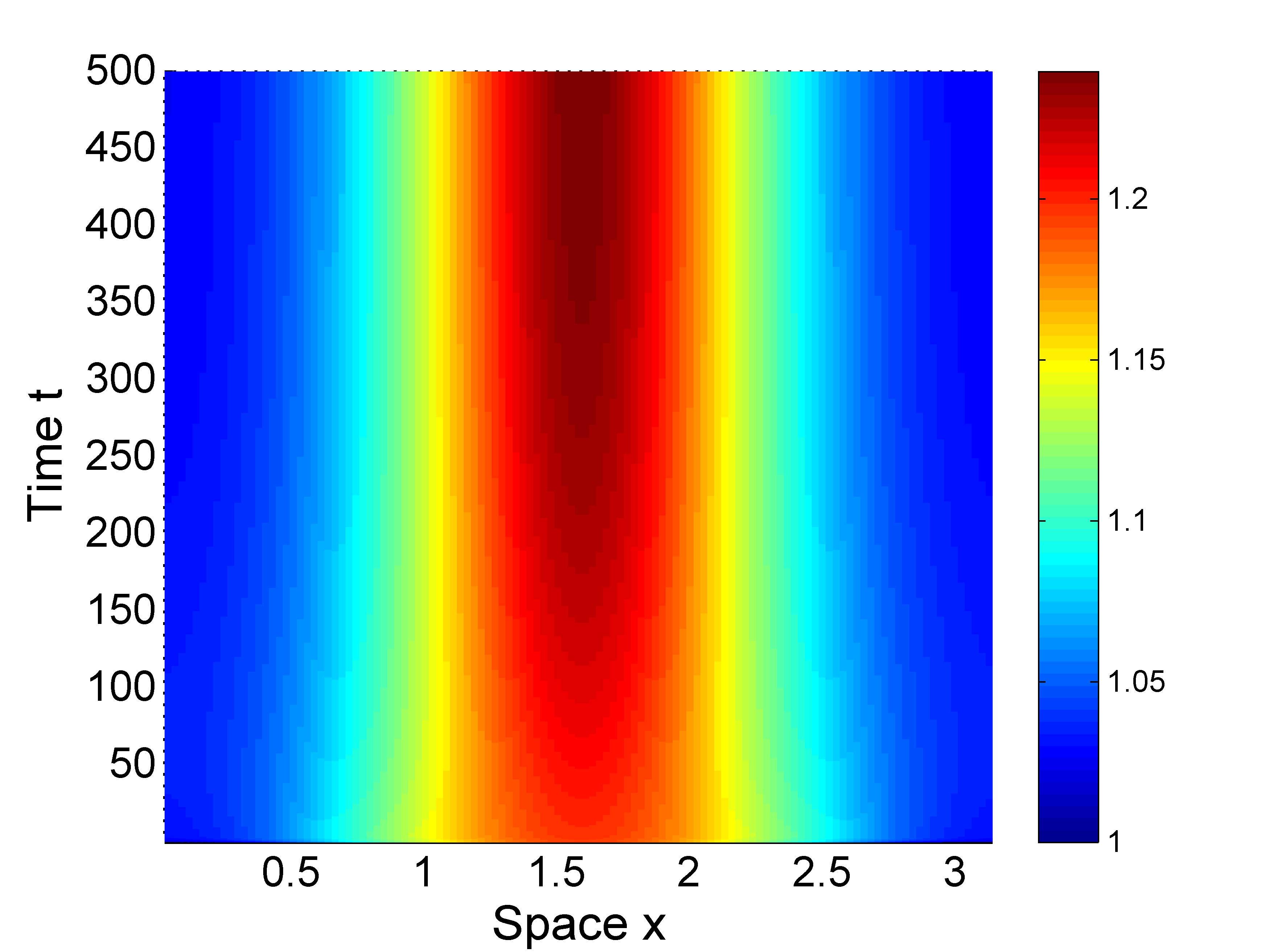}}
  \caption{The dynamics of Eq. \eqref{equ-coopnonlocal} with parameters are in \eqref{equ-coopparameter}. (Top row): $\beta=0.008$, the constant steady state $(u_*,v_*)=(0.085,1.008)$ is locally asymptotically stable; (Middle row): $\beta=0.004<\beta_1$, a mode-1 Turing pattern can be observed; (Bottom row): $\beta=0.0005<\beta_2$, a mode-2 Turing pattern can be observed. \label{Fig-coop}}
\end{figure}

Guided by the above stability and bifurcation analysis, we choose three different $\beta$ values for numerical simulations: $\beta=0.008,~\beta=0.004,~\beta=0.0005$ to observe the dynamical behavior of Eq. \eqref{equ-coopnonlocal}. When $\beta=0.008>\beta_1$, according to Theorem \ref{theorem-coop}, we know that $(u_*,v_*)$ is still locally stable. In Fig. \ref{Fig-coop} (top row),  we see that the solution of Eq. \eqref{equ-coopnonlocal} converges to the stable equilibrium $(u_*,v_*)=(0.085,1.008)$.
Then we decrease $\beta$ such that $\beta<\beta_1$. First, when $\beta=0.004$ satisfying $\beta_2<0.004<\beta_1$, a  mode-1 Turing pattern is observed in Fig. \ref{Fig-coop} (middle row). Next we take $\beta=0.0005<\beta_2$, then we observe a mode-2 Turing patterns in Fig. \ref{Fig-coop} (bottom row). Our theoretical result in Theorem \ref{theorem-coop} confirms the observations at $\beta=0.008$ and $\beta=0.004$, but the mode-2 Turing pattern observed at $\beta=0.0005$ probably is due to a secondary bifurcation not primary one from the constant steady state.

As we mentioned in the beginning of this section, the model \eqref{equ-coopnonlocal} is an example of case $(i)$ in Theorem \ref{theorem-instability}. If a nonlocal competition also exists in the system, then  the following system
\begin{equation}\label{equ-coop2nonlocal}
\begin{cases}
u_{t}=\beta u_{xx}+u\left(1-\dfrac{a}{l\pi}\ds\int_0^{l\pi}u(x,t)dx+bv\right), & x\in(0,l\pi),\ t>0,\\
v_{t}=v_{xx}+v\left(1+cu+dv-\dfrac{e}{l\pi}\ds\int_0^{l\pi}v(x,t)dx\right), & x\in(0,l\pi),\ t>0,\\
u_{x}(0,t)=u_{x}(l\pi,t)=0,~v_{x}(0,t)=v_{x}(l\pi,t)=0, & t>0,\\
u(x,0)=u_0(x)\ge0,~~v(x,0)=v_0(x)\ge0, & x\in[0,l\pi],
\end{cases}
\end{equation}
provides an example of case $(iii)$ in Theorem \ref{theorem-instability}. Here $e$ satisfying $e>d$ is the nonlocal competition parameter in species $v$ and $d$ is the growth of $v$, and other parameters have the same meaning as that in \eqref{equ-coopnonlocal}. It is clear that system \eqref{equ-coop2nonlocal} has a unique positive constant equilibrium $(u_*,v_*)$:
\begin{equation}\label{equ-coopequilibrium}
u_*=\dfrac{e-d+b}{a(e-d)-bc},~v_*=\dfrac{a+c}{a(e-d)d-bc},
\end{equation}
when $a(e-d)-bc>0$ is satisfied. And the linearization at $(u_*,v_*)$ gives
\begin{equation}\label{matrix3}
D=\begin{pmatrix}\beta & 0\\ 0 & 1\end{pmatrix},~J_{U}+J_{\bar{U}}=\begin{pmatrix}-au_* & bu_*\\ cv_* & -(e-d)v_*\end{pmatrix},~J_{U}=\begin{pmatrix}0 & bu_*\\ cv_* & dv_*\end{pmatrix}.
\end{equation}
Then $J_{U}+J_{\bar{U}}$ is stable as $a(e-d)-bc>0$, and $J_U$ satisfies $Tr(J_U)>0$ and $Det(J_U)<0$ so this example belongs to the case $(iii)$ in Theorem \ref{theorem-instability}. Through a tedious calculation, we find that $p_*<p_+$ always holds for this model, so this is an example of case $(iii-a)$ and Turing patterns can be generated similar to \eqref{equ-coopnonlocal}.

\section{A diffusive predator-prey model with nonlocal competition}

In this section, we consider  following reaction-diffusion  predator-prey (consumer-resource) model with nonlocal prey competition:
\begin{equation}\label{equ-prenonlocal}
\begin{cases}
\frac{}{}u_{t}=d_1u_{xx}+u\left(1-\dfrac{1}{kl\pi}\ds\int_0^{l\pi}u(x,t)dx\right)-\dfrac{muv}{u+1}, & x\in(0,l\pi),\ t>0,\\
v_{t}=d_2v_{xx}-\theta v+\dfrac{muv}{u+1}, & x\in(0,l\pi),\ t>0,\\
u_{x}(0,t)=u_{x}(l\pi,t),~v_{x}(0,t)=v_{x}(l\pi,t), & t>0,\\
u(x,0)=u_0(x)\ge0,~~v(x,0)=v_0(x)\ge0, & x\in[0,l\pi],
\end{cases}
\end{equation}
where $u(x,t)$, $v(x,t)$ stand for the prey and predator population densities respectively, the spatial domain is assumed to be one-dimensional interval $(0,l\pi)$, $k>0$ is carrying capacity, $m>0$ is the predation parameter and $\theta>0$ is the mortality rate of predator. The intraspecific competition of the prey is assumed to be nonlocal. The model \eqref{equ-prenonlocal} was first proposed in \cite{Merchant2011,Merchant2015} for wave propagation an unbounded domain. In \cite{ChenYu2017}, the existence of the nonlocality-induced stable spatially non-homogeneous periodic orbits was proved via Hopf bifurcation theory (see Theorems 3.2, 3.4, 3.5 in \cite{ChenYu2017}). In \cite{WuSong2019}, the same model was investigated for the Turing-Hopf bifurcation. Here we revisit this nonlocal model \eqref{equ-prenonlocal} and show that spatially non-homogeneous steady state can be induced by the nonlocal competition through bifurcation.

For the corresponding model with local competition
\begin{equation}\label{equ-prelocal}
\begin{cases}
u_{t}=d_1u_{xx}+u\left(1-\dfrac{u}{k}\right)-\dfrac{muv}{u+1}, & x\in(0,l\pi),\ t>0,\\
v_{t}=d_2v_{xx}-\theta v+\dfrac{muv}{u+1}, & x\in(0,l\pi),\ t>0,\\
u_{x}(0,t)=u_{x}(l\pi,t),~v_{x}(0,t)=v_{x}(l\pi,t), & t>0,\\
u(x,0)=u_0(x)\ge0,~~v(x,0)=v_0(x)\ge0, & x\in[0,l\pi],
\end{cases}
\end{equation}
a thorough bifurcation analysis was carried out in \cite{YiWeiShi2009}:
The system \eqref{equ-prelocal} (or equivalently \eqref{equ-prenonlocal}) has there constant non-negative equilibrium: $(0,0)$, $(k,0)$ and $(\la,v_{\la})$ with
\begin{equation}\label{equ-la}
\la=\dfrac{\theta}{m-\theta},~~v_{\la}=\dfrac{(k-\la)(1+\la)}{km},~0<\la<k.
\end{equation}
In the following we assume that $0<k\le1$ and $0<\la<k$ which ensures that the positive equilibrium $(\la,v_{\la})$ is globally asymptotically stable for the local system \eqref{equ-prelocal}(see \cite[Theorem 2.3]{YiWeiShi2009}). From the results in \cite{YiWeiShi2009}, it is also known that neither spatial steady state nor spatiotemporal patterns can appear in Eq. \eqref{equ-prelocal} under the assumptions that $0<k\le1$ and $0<\la<k$. Here we demonstrate that the spatial average  in system \eqref{equ-prenonlocal} can induce non-constant spatial patterns.

The linearization of Eq. \eqref{equ-prenonlocal} at $(\la,v_{\la})$ gives the following diffusion and Jacobian matrices
\begin{equation}\label{matrix4}
D=\begin{pmatrix}d_1 & 0\\ 0 & d_2\end{pmatrix},~J_{U}+J_{\bar{U}}=\begin{pmatrix}\dfrac{\la(k-1-2\la)}{k(1+\la)} & -\theta \vspace{0.2cm}\\ \dfrac{k-\la}{k(1+\la)} & 0\end{pmatrix},~J_{U}=\begin{pmatrix}\dfrac{\la(k-\la)}{k(1+\la)} & -\theta \vspace{0.2cm}\\ \dfrac{k-\la}{k(1+\la)} & 0\end{pmatrix}.
\end{equation}
Then, when $0<k\leq1$, $J_{U}+J_{\bar{U}}$ is stable as $0<\la<k$, and $J_U$ satisfies $Tr(J_U)>0$ and $Det(J_U)>0$ so this example belongs to the case $(ii)$ in Theorem \ref{theorem-instability}. Thus, the characteristic equation for the linearized system \eqref{equ-linear} is
\begin{equation}\label{equ-prechara}
\mu^{2}-T_{i}(\la)\mu+D_{i}(\la)=0,~i\in\mathbb{N}_{0},
\end{equation}
where
\begin{equation*}\label{equ-prechara0}
T_{0}(\la)=\dfrac{\la(k-1-2\la)}{k(1+\la)},~D_{0}(\la)=\frac{\theta(k-\la)}{k(1+\la)},
\end{equation*}
and for $i\ge1$,
\begin{equation*}\label{equ-precharan}
T_{i}(\la)=\dfrac{\la(k-\la)}{k(1+\la)}-\frac{(d_1+d_2)i^2}{l^2},~D_{i}(\la)=\frac{\theta(k-\la)}{k(1+\la)}-\frac{d_2\la(k-\la)}{k(1+\la)}\frac{i^2}{l^2}+\frac{d_1d_2i^4}{l^4}.
\end{equation*}
By letting $p=i^2/l^2$, we define the trace and determinant functions to be
\begin{equation}\label{Di}
T(\la,p)=C_1(\la)-(d_1+d_2)p,~D(\la,p)=d_1d_2p^{2}-d_2C_1(\la)p+\frac{\theta(k-\la)}{k(1+\la)},
\end{equation}
where
\begin{equation}\label{equ-C1}
C_1(\la):=\frac{\la(k-\la)}{k(1+\la)}.
\end{equation}

For the further discussion, we also define $C_2(\la):=\la C_1(\la)$. For the properties of functions $C_1(\la)$ and $C_2(\la)$, the following results are given in \cite[Lemma 3.1]{ChenYu2017}:
\begin{lemma}\label{lemma-preC1C2}
For $k>0$, the following statements are true:
\begin{description}
  \item (i) there exists $\la_*:=\sqrt{k+1}-1$ such that $C_1^{\prime}(\la_*)=0$ and $C_1^{\prime}(\la)>0$ for $\la\in(0,\la_*)$ and $C_1^{\prime}(\la)<0$ for $\la\in(\la_*,k)$ and $\ds\max_{\la\in[0,k]}C_1(\la)=C_1(\la_*)$;
  \item (ii) there exists $\la_{\sharp}:=\dfrac{k-3+\sqrt{(k-3)^2+16k}}{4}$ such that $C_2^{\prime}(\la_{\sharp})=0$ and $C_2^{\prime}(\la)>0$ for $\la\in(0,\la_{\sharp})$ and $C_2^{\prime}(\la)<0$ for $\la\in(\la_{\sharp},k)$ and $\max_{\la\in[0,k]}C_2(\la)=C_2(\la_{\sharp})$.
\end{description}
\end{lemma}
Also the result on spatially non-homogeneous Hopf bifurcations are obtained.
\begin{proposition}\label{lemma-preyHopf}
(\cite[Theorem 3.2]{ChenYu2017}) Let $\la_{\sharp}$ and $\la_*$ be defined in Lemma \ref{lemma-preC1C2}, and suppose that $d_1,~d_2,~m,~\theta>0$ and $0<k\le1$ satisfy
\begin{equation}\label{5.10}
  \dfrac{d_1}{d_2}>\dfrac{C_2(\la_{\sharp})}{4\theta},
\end{equation}
 and define
\begin{equation}\label{equ-lnHopf}
l^{H}_i:=i\sqrt{\dfrac{d_1+d_2}{C_1(\la_*)}}, ~\textrm{with}~i\in\mathbb{N}.
\end{equation}
Then, the following two statements are true.
\begin{description}
\item (i) If $l\in(0,l^{H}_1)$, then $(\la,v_{\la})$ is locally asymptotically stable for $\la\in(0,k)$.
\item (ii) If $l\in(l^{H}_1, \infty)$, then there exist finitely many critical points satisfying
\begin{equation*}
0<\la^{H}_{1,-}(l)<\cdots<\la^{H}_{n,-}(l)<\la_*<\la^{H}_{n,+}<\cdots<\la^{H}_{1,+}(l)<k,
\end{equation*}
such that $(\la,v_{\la})$ is locally asymptotically stable for $\la\in\left(0,\la^{H}_{1,-}(l)\right)\cup\left(\la^{H}_{1,+}(l),\infty\right)$ and unstable for $\la\in\left(\la^{H}_{1,-}(l),\la^{H}_{1,+}(l)\right)$. Moreover, system \eqref{equ-prenonlocal} undergoes Hopf bifurcation at $\la=\la^{H}_{n,\pm}(l)$, and the bifurcating periodic solutions near $\la^{H}_{n,-}(l)$ or $\la^{H}_{n,+}(l)$ are spatially non-homogeneous.
\end{description}
\end{proposition}

We now consider the steady state bifurcations for \eqref{equ-prenonlocal}. The steady state solutions of \eqref{equ-prenonlocal} satisfy the following elliptic problem:
\begin{equation}\label{equ-prenonlocalsteady}
\begin{cases}
d_1u_{xx}+(u+\la)\left(1-\dfrac{1}{kl\pi}\ds\int_0^{l\pi}(u(x)+\la)dx\right)-\dfrac{m(u+\la)(v+v_{\la})}{u+\la+1}=0, & x\in(0,l\pi),\\
d_2v_{xx}-\theta (v+v_{\la}) +\dfrac{m(u+\la)(v+v_{\la})}{u+\la+1}=0, & x\in(0,l\pi),\\
u_x(0)=u_x(l\pi)=0,~v_x(0)=v_x(l\pi)=0, \\
\end{cases}
\end{equation}
which has a trivial equilibrium $(0,0)$, and we want to find its non-trivial solution.
%where $m$ is replaced with $\dfrac{\theta(1+\la)}{\la}$ according to the relation in \eqref{equ-la}.
Then, the condition for steady state bifurcation is that $D(\la,i^2/l^2)=D_i(\la)=0$ which is defined in \eqref{Di}. It is clear that $\ds\frac{\theta(k-\la)}{k(1+\la)}>0$ and $d_2C_1(\la)>0$ for any $\la\in(0,k)$. So if we assume that
\begin{equation}\label{equ-preDelta}
\frac{d_1}{d_2}<\frac{C_2(\la_{\sharp})}{4\theta},
\end{equation}
then from Lemma \ref{lemma-preC1C2}, there exist $\underline{\la},\bar{\la}$ satisfying $0<\underline{\la}<\la_{\sharp}<\bar{\la}<k$ such that
 \begin{equation}\label{lam}
   \frac{d_1}{d_2}=\frac{C_2(\underline{\la})}{4\theta}=\frac{C_2(\bar{\la})}{4\theta},
\end{equation}
and for any $\la\in (\underline{\la},\bar{\la})$, $D(\la,\cdot)=0$ has two positive roots:
\begin{equation}\label{equ-prep+-}
p_{\pm}(\la)=\dfrac{d_2C_1(\la)\pm\sqrt{C_1(\la)(d_2^2C_2(\la)-4d_1d_2\theta)}}{2d_1d_2}.
\end{equation}
By using similar arguments  in the proof of \cite[Lemma 3.9]{YiWeiShi2009}, we obtain the following properties of $p_{\pm}(\la)$.
\begin{proposition}\label{proposition-presteady}
Suppose that $0<k\le1$ and \eqref{equ-preDelta} holds, there exist $\la_{-},~\la_{+}\in[\underline{\la},\bar{\la}]$ such that
\begin{description}
\item (i) $p_{+}(\la)$ is increasing in $(\underline{\la},\la_{+})$ and decreasing in $(\la_{+},\bar{\la})$, and $\max{p_{+}(\la)}=p_{+}(\la_{+})$;
\item (ii) $p_{-}(\la)$ is decreasing in $(\underline{\la},\la_{-})$ and increasing in $(\la_{-},\bar{\la})$, and $\min{p_{-}(\la)}=p_{-}(\la_{-})$.
\end{description}
\end{proposition}

Now we have the following results on the steady state bifurcations and Hopf bifurcations for system \eqref{equ-prenonlocal} under the condition \eqref{equ-preDelta}.

\begin{theorem}\label{theorem-preysteadyHopf}
Suppose that $d_1,~d_2,~\theta,~0<k\le1$ satisfy \eqref{equ-preDelta}.
 \begin{description}
   \item (i) Let $p_{\pm}(\la)$ be defined by \eqref{equ-prep+-} and $\la_{\pm}$ in Proposition \ref{proposition-presteady}, for  $i\in\mathbb{N}$ we define
\begin{equation*}
l^{S}_{i,-}:=\frac{i}{\sqrt{p_{+}(\la_+)}},~~l^{S}_{i,+}:=\frac{i}{\sqrt{p_{-}(\la_-)}}.
\end{equation*}
When $l\in\left(l^{S}_{i,-},l^{S}_{i,+}\right)$,  there exist exactly two points $\la^{S}_{i,\pm}\in[\underline{\la},\bar{\la}]$  such that $p_{\pm}\left(\la^{S}_{i,\pm}\right)=\dfrac{i^2}{l^2}$. If $\la^{S}_{i,\pm}\ne \la^{S}_{j,\pm}$ for $j\ne i$, then there is a  smooth curve $\Gamma_{i,\pm}$ of positive solutions of \eqref{equ-prenonlocalsteady} bifurcating from the line of constant solutions
 $(\la,u,v)=(\la,\la,u_{\la})$ at $\la=\la_{i,\pm}^{S}$. Moreover, near $\left\{(\la_{i,\pm}^{S},\la_{i,\pm}^{S},v_{\la_{i,\pm}^{S}})\right\}$,
there exists a positive constant $\delta>0$ such that $\Gamma_{i,\pm}\in C^{\infty}$ has the following form:
\begin{equation}\label{equ-Gamman}
\Gamma_{i,\pm}=\left\{\left(\la(s), u(s), v(s)\right):\ -\delta<s<\delta\right\},
\end{equation}
where
\begin{equation*}
\begin{aligned}
u(s)&=\la_{i,\pm}^{S}+s\cos\left(\frac{ix}{l}\right)+sz_{1}(s,x),\\
v(s)&=v_{\la_{i,\pm}^{S}}+\frac{l^2\left(k-\la_{i,\pm}^{S}\right)}{d_2i^2k\left(1+\la_{i,\pm}^{S}\right)}s\cos\left(\frac{ix}{l}\right)+sz_{2}(s,x),
\end{aligned}
\end{equation*}
with $\la(s),~z_{j}(s,x)$ are smooth functions defined for $s\in (-\delta,\delta)$ such that $\la(0)=\la_{i,\pm}^{S}$, and $z_{j}(0,x)=0$ ($j=1,2$).
   \item (ii) Let $l_i^{H},~\la_{n,\pm}^{H}$ be defined in Lemma \ref{lemma-preyHopf} and let $C_2(\la),~\la_{\sharp}$ defined in Lemma \ref{lemma-preC1C2}. Then system \eqref{equ-prenonlocal} undergoes a Hopf bifurcation at $\la=\la_{n,\pm}^{H}$ if $\la_{n,\pm}^{H}\not\in[\underline{\la},\bar{\la}]$, where $\underline{\la},\bar{\la}$ are defined in \eqref{lam}.
 \end{description}
\end{theorem}
\begin{proof}
The proof of part (i) is similar to the  proof of Theorems \ref{theorem-scalarbifurcation} and \ref{theorem-coop}, and we again use the abstract bifurcation theory in \cite{Rabinowitz1971,Shi1999}.

Following the similar setting in \cite{YiWeiShi2009}, we define a nonlinear mapping $H:\ \mathbb{R}^+\times X^2\rightarrow Y^2$ by
\begin{equation}\label{equ-H}
H(\la,u,v)=\left(\begin{array}{cccc}
\medskip
&d_1u_{xx}+(u+\la)\left(1-\dfrac{1}{kl\pi}\ds\int_0^{l\pi}(u(x)+\la)dx\right)-\dfrac{m(u+\la)(v+v_{\la})}{u+\la+1}\\
\medskip
&d_2v_{xx}-\theta (v+v_{\la}) +\dfrac{m(u+\la)(v+v_{\la})}{u+\la+1}\\
\end{array}\right).
\end{equation}
It is clear that $H(\la,0,0)=(0,0)$.
At $\la=\la_{i,\pm}^{S}$, we have
\begin{equation}\label{equ-hatL}
\begin{aligned}
 &H_{(u,v)}\left(\la_{i,\pm}^{S},0,0\right)[\varphi,\psi]
 =\left(\begin{array}{cc}
\medskip
 &d_1\varphi_{xx}-\dfrac{\la_{i,\pm}^{S}}{kl\pi}\ds\int_{0}^{l\pi}\varphi dx+C_1(\la_{i,\pm}^{S})\varphi-\theta\psi,\\
\medskip
 &d_2\psi_{xx}+A(\la_{i,\pm}^{S})\varphi.\\
\end{array}\right):=\hat{L}[\varphi,\psi],\end{aligned}
\end{equation}
where $A(\la):=\dfrac{k-\la}{k(1+\la)}$ and $C_1(\la)$ is defined in \eqref{equ-C1}. We assume that $\la^{S}_{i,\pm}\ne \la^{S}_{j,\pm}$ for $j\ne i$, then the kernel is $\mathcal{N}(\hat{L})=\textrm{Span}\left\{\hat{q}=(1, \hat{h})\varphi_1 \right\}$ where $\hat{h}=\dfrac{l^2 A\left(\la_{i,\pm}^{S}\right)}{d_2i^2}$, thus $\dim{\mathcal{N}(\hat{L})}=1$. The range space of $\hat{L}$ is $\mathcal{R}(\hat{L})=\left\{(f_1,f_2)\in Y^2:\ \lan y,(f_1,f_2)\rangle=0\right\}$, where $y$ is defined by
\begin{equation*}
\lan y,(f_1,f_2)\rangle=\ds\int_0^{l\pi}\left(f_1-\dfrac{l^2\theta}{d_2 i^2}f_2\right)\varphi_1dx,
\end{equation*}
and thus $\textrm{codim}\mathcal{R}(\hat{L})=1$.

Next, We prove that $H_{\la(u,v)}\left(\la_{i,\pm}^{S}, 0,0\right)[\hat{q}]\not\in\mathcal{R}(\hat{L})$. From \eqref{equ-H}, we have
\begin{equation}\label{equ-Gderibeta}
 H_{\la(u,v)}\left(\la_{i,\pm}^{S},0,0\right)[\hat{q}]=\left(C_1'(\la_{i,\pm}^{S}), A'(\la_{i,\pm}^{S})\right)^T\varphi_i.
\end{equation}
Using the definition of $\mathcal{R}(\hat{L})$, we obtain the integral
\begin{equation}\label{equ-transversality}
T:=\ds\int_0^{l\pi}\left(f_1-\dfrac{l^2\theta}{d_2 i^2}f_2\right)\varphi_idx=\ds\int_0^{l\pi}\left(C_1'(\la_{i,\pm}^{S})-\dfrac{l^2\theta A'(\la_{i,\pm}^{S})}{d_2i^2}\right)\varphi^2_idx.
\end{equation}
%Because $\la_{i,\pm}^{S}$ satisfies $p_{\pm}\left(\la_{i,\pm}^{S}\right)=\dfrac{i^2}{l^2}$ such that $D_i\left(\la_{i,\pm}^{S}\right)=0$, which reads
Since $p_{\pm}(\la)$ satisfies $D(\la,p_{\pm}(\la))=0$, we have
\begin{equation}\label{equ-Di}
d_1d_2p_{\pm}^{2}(\la)-d_2C_1(\la)p_{\pm}(\la)+A(\la)=0.
\end{equation}
Differentiating \eqref{equ-Di} with respect to $\la$ at $\la=\la_{i,\pm}^{S}$, we obtain
\begin{equation*}
p'_{\pm}\left(\la_{i,\pm}^{S}\right)=\dfrac{\theta A'\left(\la_{i,\pm}^{S}\right)-d_2C_1'\left(\la_{i,\pm}^{S}\right)p_{\pm}\left(\la_{i,\pm}^{S}\right)}{2d_1d_2p_{\pm}\left(\la_{i,\pm}^{S}\right)-d_1C_1'\left(\la_{i,\pm}^{S}\right)}.
\end{equation*}
Since $l\in\left(l^{S}_{i,-},l^{S}_{i,+}\right)$, thus $p'_{\pm}\left(l^{S}_{i,\pm}\right)\ne 0$ which implies that
\begin{equation*}
\theta A'\left(\la_{i,\pm}^{S}\right)-d_2C_1'\left(\la_{i,\pm}^{S}\right)p_{\pm}\left(\la_{i,\pm}^{S}\right)\ne 0,
\end{equation*}
together with $p_{\pm}\left(\la_{i,\pm}^{S}\right)=\dfrac{i^2}{l^2}$, we have $T\ne 0$ with $T$ defined in \eqref{equ-transversality}, thus $H_{\la(u,v)}\left(\la_{i,\pm}^{S}, 0,0\right)[\hat{q}]\not\in\mathcal{R}(\hat{L})$ is proved. By applying Theorem 1.7 in \cite{Rabinowitz1971}, we obtain the result in part (i).

As for part (ii), from \eqref{equ-prep+-},  when $\la\in[\underline{\la},\bar{\la}]$, we have $D(\la,\cdot)\leq0$ which does not satisfy the condition of Hopf bifurcation. But when $\la_{n,\pm}^{H}\not\in[\underline{\la},\bar{\la}]$, we have $D(\la_{n,\pm}^{H},\cdot)>0$ thus Hopf bifurcations can occur.
\end{proof}

Proposition \ref{lemma-preyHopf} and Theorem \ref{theorem-preysteadyHopf} define two minimal domain size(patch length) $l^{H}_1$ for Hopf bifurcation and $l^{S}_{1,-}$ for steady state bifurcation. Together with the threshold conditions \eqref{5.10} and \eqref{equ-preDelta} for the diffusion coefficients, we have the following classification of different scenarios of steady state and Hopf bifurcations when using $\la$ as the bifurcation parameter.

\begin{corollary}\label{coro-prescenario}
Let $C_2(\la),~\la_{\sharp}$, $l^{H}_1$ be defined in Lemma \ref{lemma-preyHopf} and $l^{S}_{1,-}$ be defined in Theorem \ref{theorem-preysteadyHopf}. Denote $M=\dfrac{C_2(\la_{\sharp})}{4\theta}$, for the bifurcation scenarios in system \eqref{equ-prenonlocal}, we have the following results:
\begin{description}
\item (i) when $d_1>d_2M,~l<l^{H}_1$ or $d_1<d_2M,~l<l^{S}_{1,-}$, the constant steady state $(\la, v_{\la})$ is locally asymptotically stable, and both steady state and Hopf bifurcation will not occur;
\item (ii) when $d_1>d_2M,~l>l^{H}_1$ or $d_1<d_2M,~l^{H}_{1}<l<l^{S}_{1,-}$, Hopf bifurcation can occur, but steady state bifurcation cannot occur;
\item (iii) when $d_1<d_2M,~l<l^{H}_1$ and $l>l^{S}_{1,-}$, steady state bifurcation can occur, but Hopf bifurcation can not occur;
\item (iv) when $d_1<d_2M,~l>l^{H}_1$ and $l>l^{S}_{1,-}$, both steady state and Hopf bifurcation can occur.
\end{description}
\end{corollary}
A similar classification was given in \cite{ChenShi2014} on the pattern formation conditions for a diffusive Gierer-Meinhardt system. The results in  Corollary \ref{coro-prescenario} are depicted numerically in Fig.  \ref{Fig-prescenario}.

\begin{figure}[htp]
\centering
\subfloat[]{\includegraphics[scale=0.6]{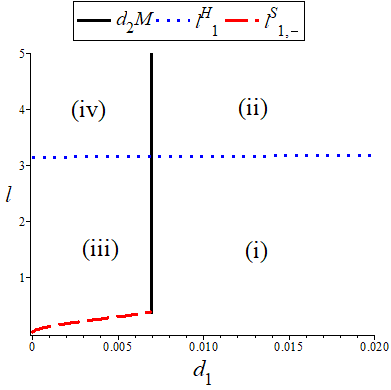}}\hspace{0.3cm}
\subfloat[]{\includegraphics[scale=0.6]{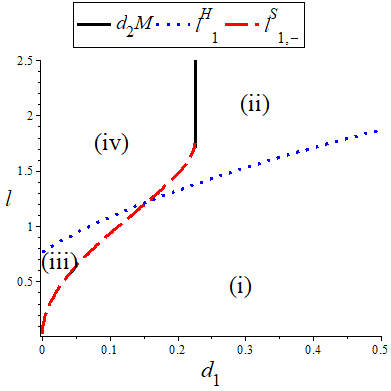}}\quad
  \caption{Illustration of possible bifurcation scenarios in system \eqref{equ-prenonlocal}, which  correspond to the four cases in Corollary \ref{coro-prescenario}. The parameters used: (a) $d_2=1,~\theta=1,~k=0.5$; (b) $d_2=0.1,~\theta=0.01,~k=1$.
  \label{Fig-prescenario}}
\end{figure}

\begin{remark}\label{rem4}
(i) In Theorem \ref{theorem-preysteadyHopf}, the direction of the steady state bifurcations in system \eqref{equ-prenonlocal} can be  determined similarly as in Theorem \ref{theorem-coop}.

%(ii) In Corollary \ref{coro-prescenario}, $l^{S}_{1,-}$ is the minimal domain length (patch size) for steady state bifurcation to occur and $l^{H}_{1}$ is the minimal domain length for Hopf bifurcation.

(ii) For the stability of periodic orbits bifurcated through a Hopf bifurcation, we refer readers for the calculation of normal form in \cite{YiWeiShi2009} which is for a classical reaction-diffusion system as the calculation for our model with spatial average is similar. Because of the introduction of the spatial average term, so some differences happen for the derivatives of the nonlinear functions $f(\la,u,v)$ and $g(\la,u,v)$ at $(\la,v_{\la})$:
\begin{equation}\label{equ-prederivatives}
\begin{split}
f_{uu}&=\dfrac{2(k-\la)}{k(1+\la)^2},~f_{uv}=-\frac{\theta}{\la(1+\la)},~f_{vv}=0,\\
f_{uuu}&=-\frac{6(k-\la)}{k(1+\la)^3},~f_{uuv}=\frac{2\theta}{\la(1+\la)^2},~f_{uvv}=0,~f_{vvv}=0,\\
g_{uu}&=-\dfrac{2(k-\la)}{k(1+\la)^2},~g_{uv}=\frac{\theta}{\la(1+\la)},~g_{vv}=0,\\
g_{uuu}&=\frac{6(k-\la)}{k(1+\la)^3},~g_{uuv}=-\frac{2\theta}{\la(1+\la)^2},~g_{uvv}=0,~g_{vvv}=0.
\end{split}
\end{equation}
Other calculations are similar, so we will not repeat here. Also, in \cite{Chen2018}, the Hopf bifurcation direction in a diffusive Holling-Tanner predator-prey model with spatial average is computed similarly.
\end{remark}

\begin{figure}[htp]
\centering
\subfloat[(i)]{\includegraphics[scale=0.35]{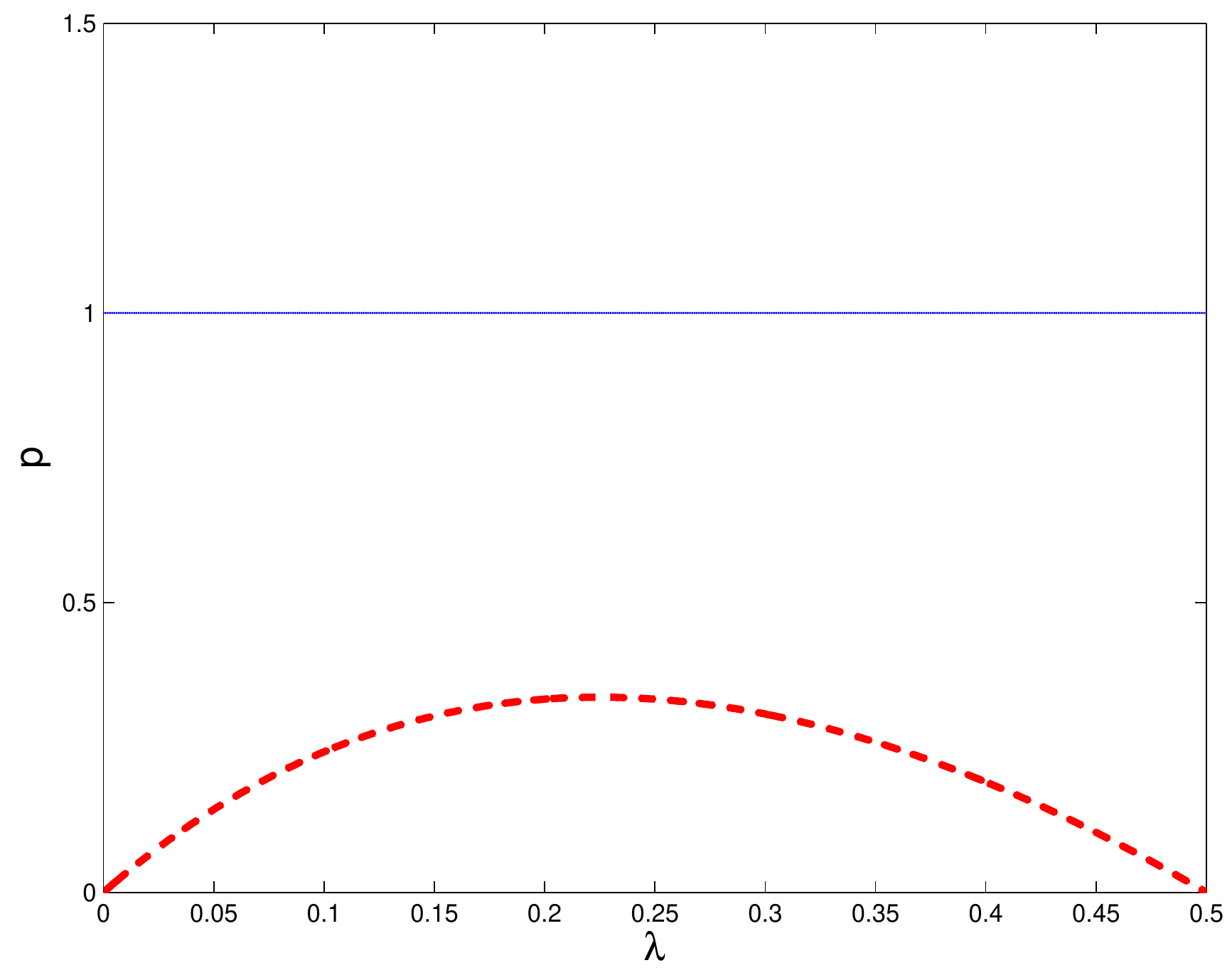}}\hspace{0.3cm}
\subfloat[(ii)]{\includegraphics[scale=0.35]{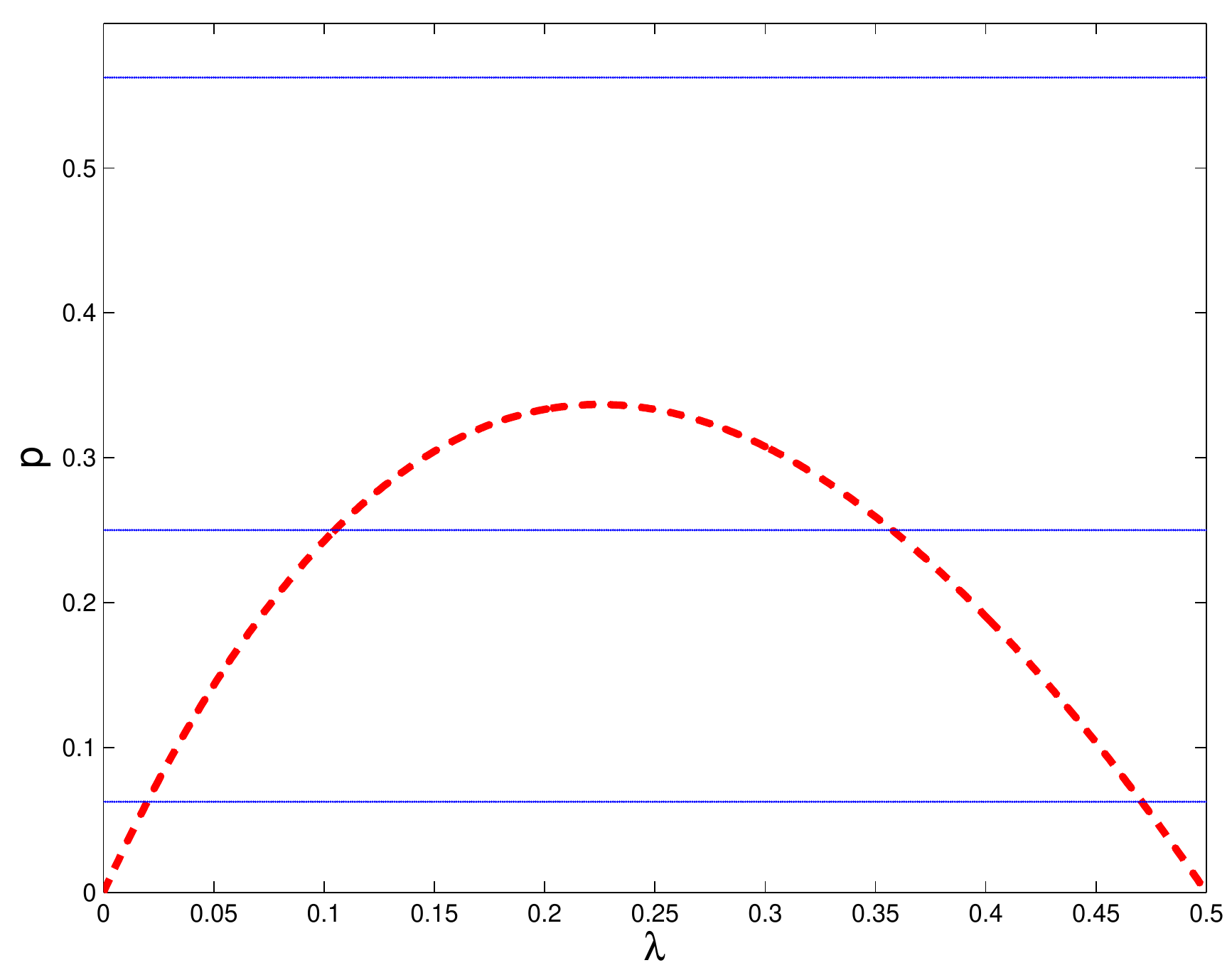}}\quad
\subfloat[(iii)]{\includegraphics[scale=0.35]{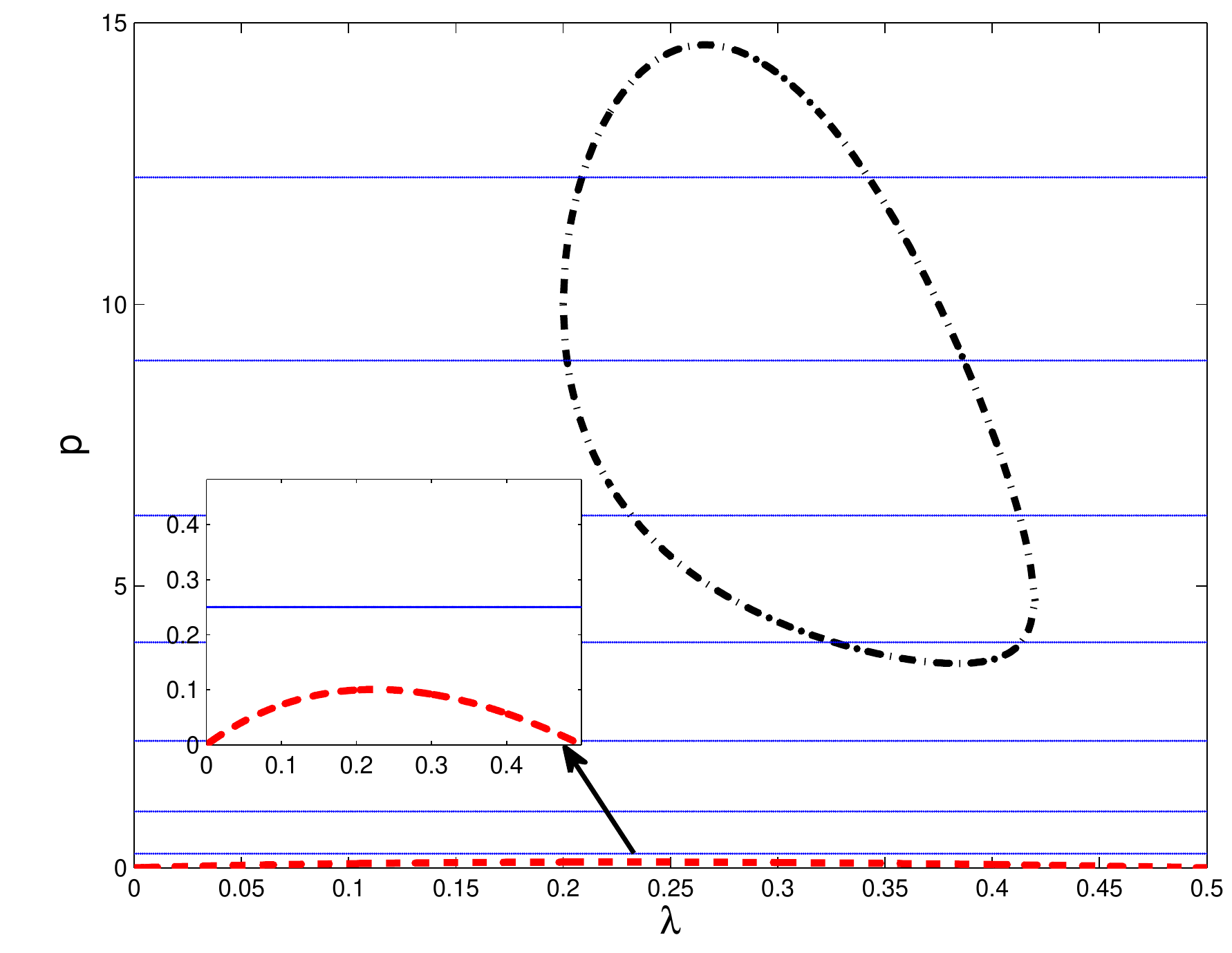}}\hspace{0.3cm}
\subfloat[(iv)]{\includegraphics[scale=0.35]{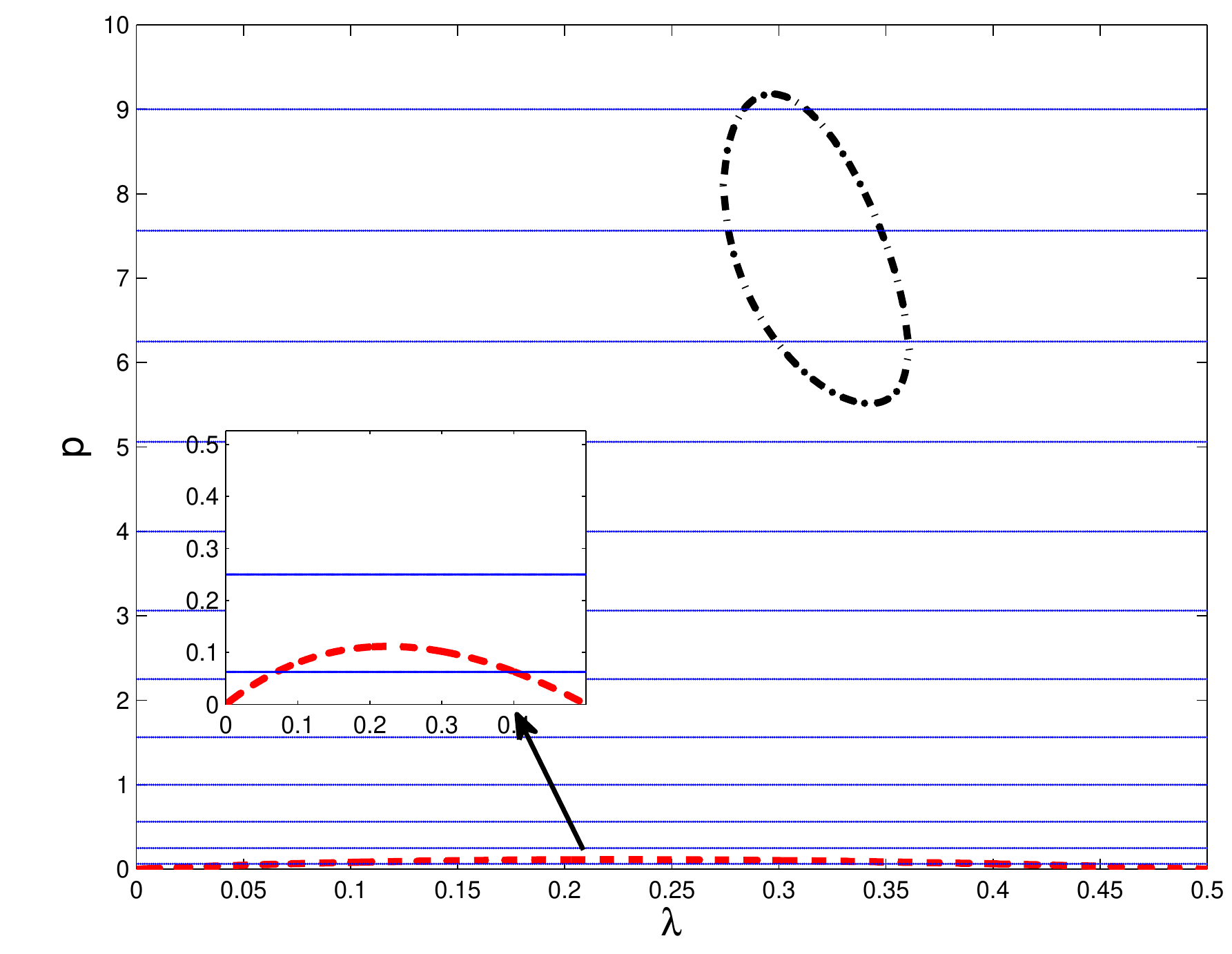}}
  \caption{Bifurcation diagrams corresponding to four scenarios in Fig. \ref{Fig-prescenario}. In each case, the graphs of  $D(\la,p)=0$ (black dash-dot curve), $T(\la,p)=0$ (red dashed curve) are plotted, and the blue solid horizontal lines are $p=i^2/l^2$ with $i\in\mathbb{N}$. The parameters for four diagrams are, respectively: (i) $d_1=0.1,~d_2=0.2,~\theta=1,~k=0.5,~l=1$; (ii) $d_1=0.1,~d_2=0.2,~\theta=1,~k=0.5,~l=4$; (iii) $d_1=0.005,~d_2=1,~\theta=1,~k=0.5,~l=2$; (iv) $d_1=0.006,~d_2=0.9,~\theta=1,~k=0.5,~l=4$.
  \label{Fig-prediagram}}
\end{figure}

Fig. \ref{Fig-prediagram} show the bifurcation diagrams of system \eqref{equ-prenonlocal} by plotting the graphs of zero level set of determinant function $D(\la,p)=0$ and trace function $T(\la,p)=0$. The intersection points of the closed loop $D(\la,p)=0$ and lines $p=i^2/l^2$ determine the steady state bifurcation points, while the intersection points of $T(\la,p)=0$ and $p=i^2/l^2$ outside of the loop $D(\la,p)=0$ determine the Hopf bifurcation points. Cases (i) and (ii) here belong to Case (ii-a1) in Fig. \ref{fig-2}, and the set $\{(\la,p):D(\la,p)=0\}$ is empty; and cases (iii) and (iv) belong to Case (ii-d) in Fig. \ref{fig-2}, and the set $\{(\la,p):D(\la,p)=0\}$ is a closed loop.

Guided by  Corollary \ref{coro-prescenario}, Fig. \ref{Fig-prescenario}, and Fig. \ref{Fig-prediagram}, we use numerical simulations to verify the spatiotemporal pattern formations.

\begin{example}\label{exm1}
When
\begin{equation}\label{equ-preparameter33}
d_1=0.1,~d_2=0.2,~\theta=1,~k=0.5,~l=1,
\end{equation}
the case (i) in Corollary \ref{coro-prescenario} occurs, and it is predicted that neither steady state nor Hopf bifurcation will occur. Fig. \ref{Fig-solupre1} shows that the constant steady state is asymptotically stable for this choice of parameters.
\end{example}

\begin{figure}[htp]
\centering
\subfloat[Prey]{\includegraphics[scale=0.45]{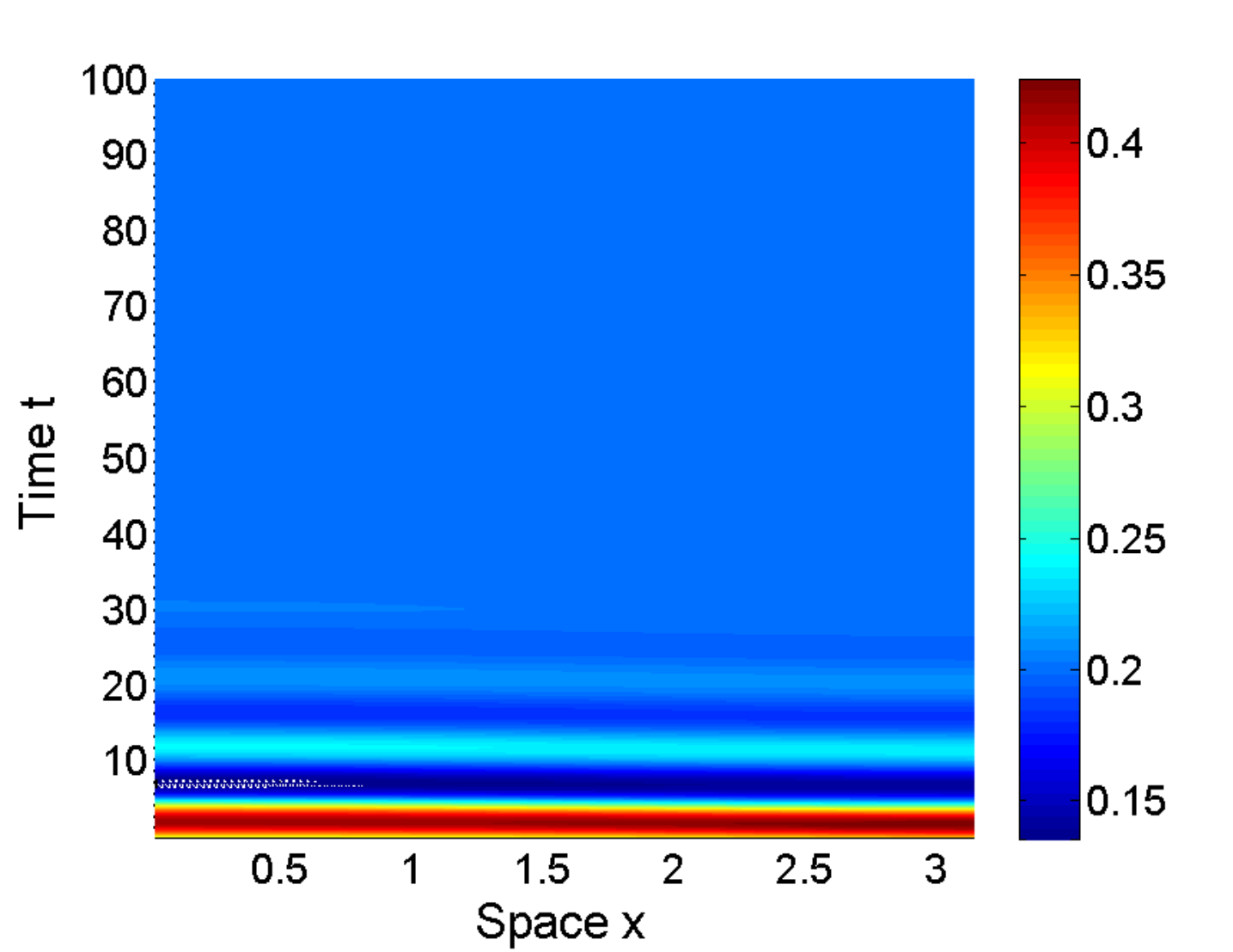}}
\subfloat[Predator]{\includegraphics[scale=0.45]{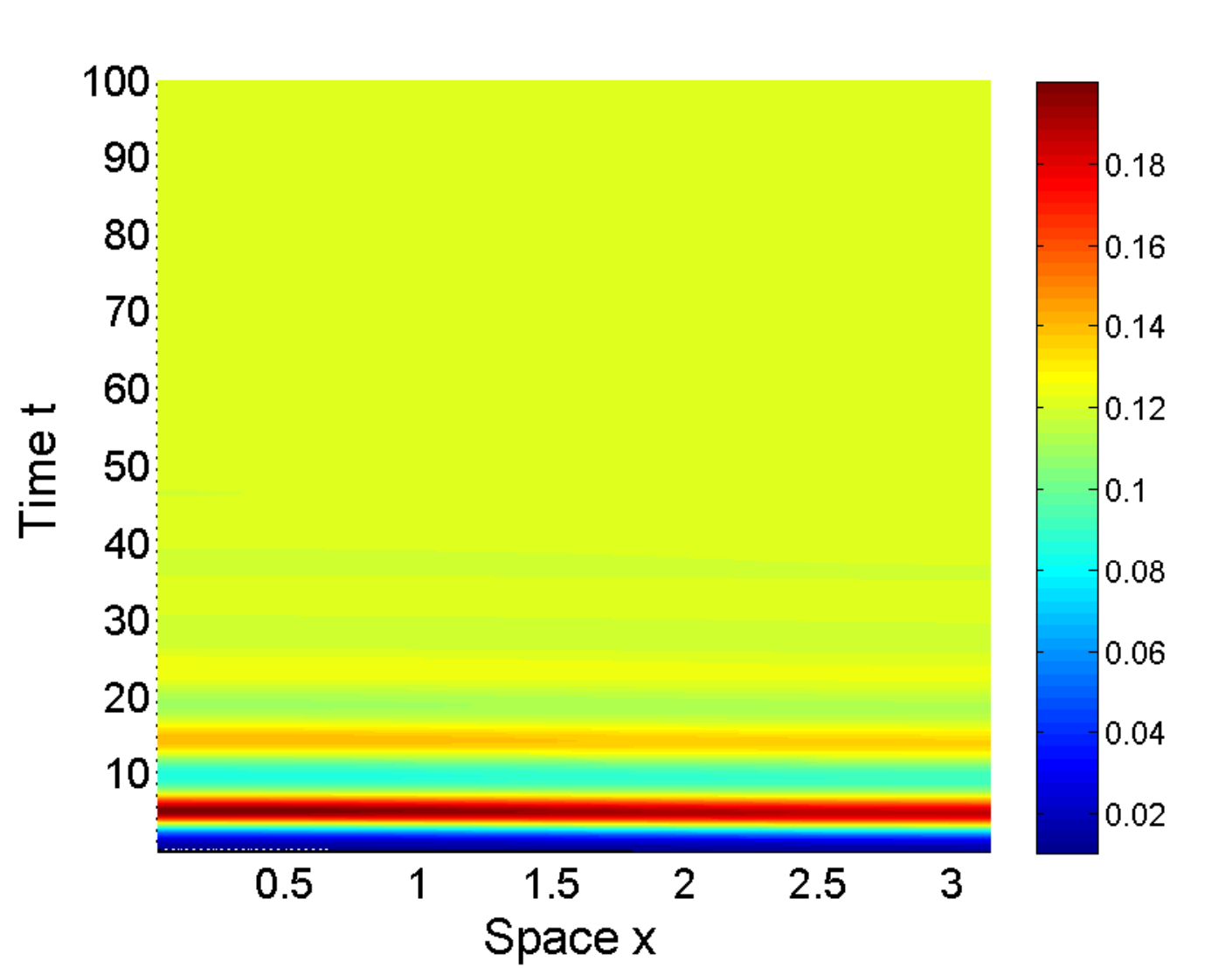}}
\caption{Dynamics of Eq. \eqref{equ-prenonlocal} with parameters in \eqref{equ-preparameter33} and $m=6$ ($\la=0.2$): converges to the constant steady state with initial values $u_0(x)=0.4-0.1\cos(x/4),~v_0(x)=0.02-0.01\cos(x/4)$.  \label{Fig-solupre1}}
\end{figure}

\begin{example}\label{exm2}
When
\begin{equation}\label{equ-preparameter44}
d_1=0.1,~d_2=0.2,~\theta=1,~k=0.5,~l=4,
\end{equation}
it is the case (ii) in Corollary \ref{coro-prescenario}, and only spatially non-homogeneous Hopf bifurcations can occur and the spatially non-homogeneous Hopf bifurcation points are:
\begin{equation*}
\la^{H}_{1,-}=0.0199,~\la^{H}_{1,+}=0.4707,~\la^{H}_{2,-}=0.1049,~\la^{H}_{2,+}=0.3576.
\end{equation*}
Fig. \ref{Fig-solupre2} shows that when $\la=0.2$, both mode-1 and mode-2 spatially non-homogeneous time-periodic patterns are observed with different initial conditions.

\begin{figure}[tp]
\centering
\subfloat[Prey]{\includegraphics[scale=0.45]{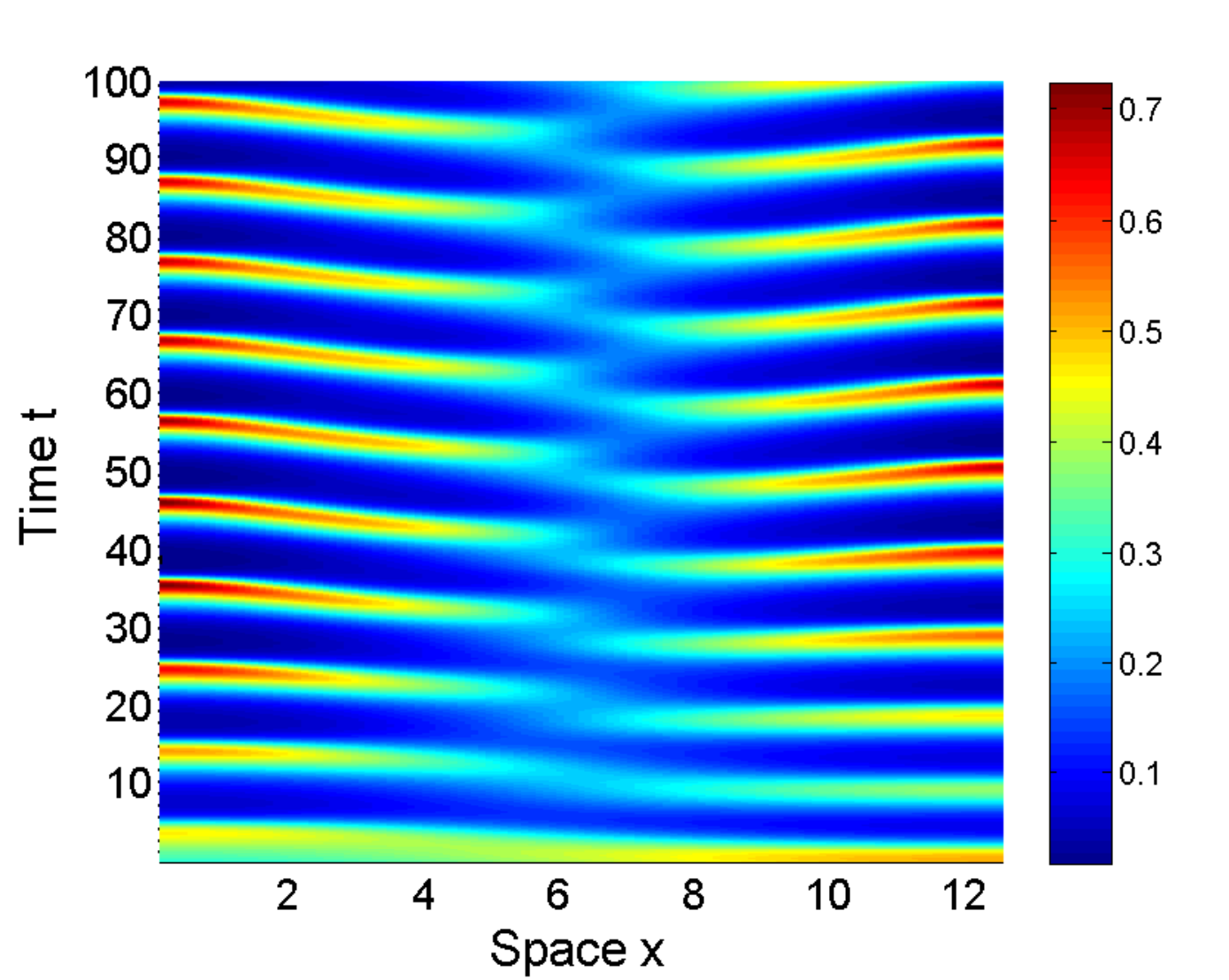}}
\subfloat[Predator]{\includegraphics[scale=0.45]{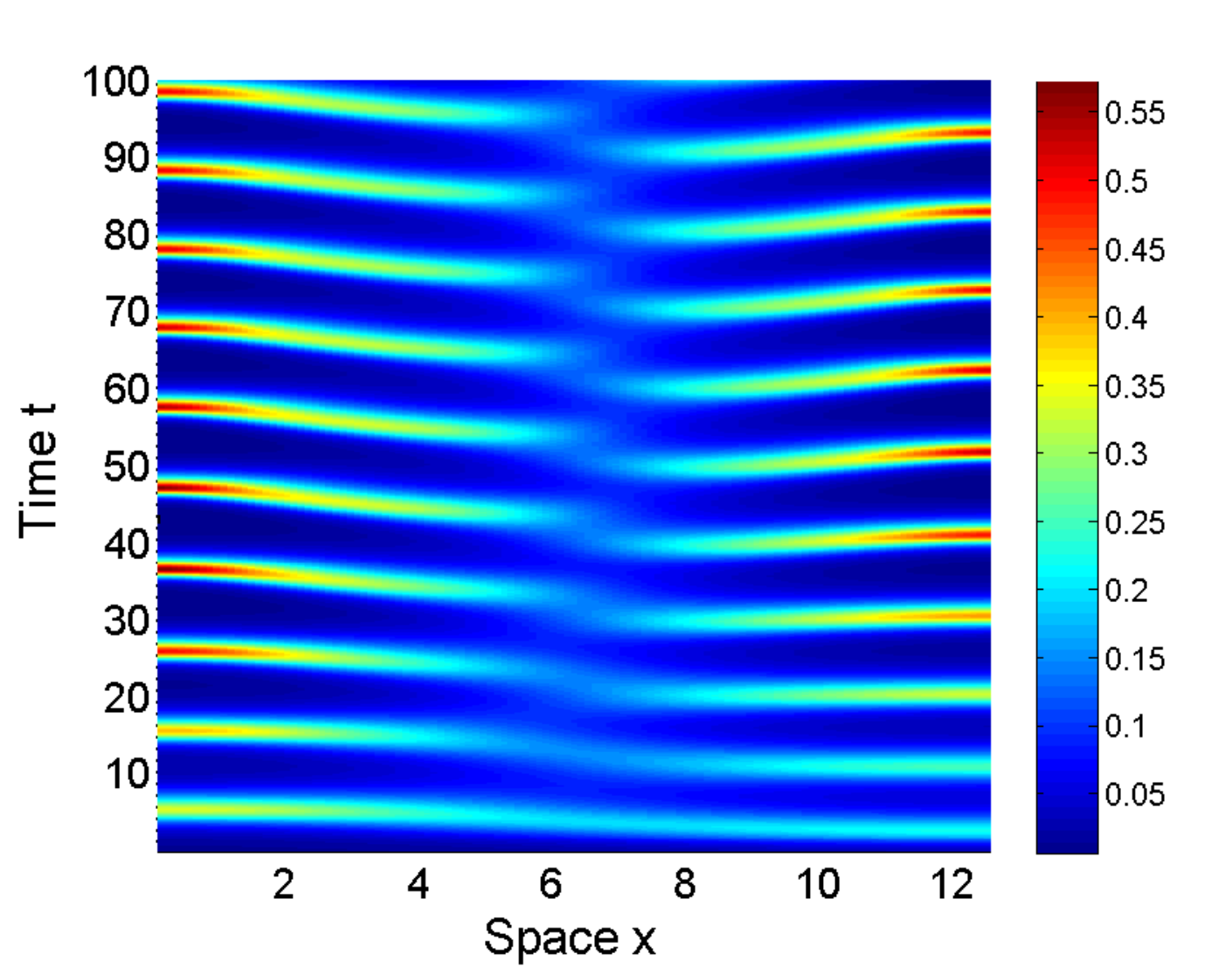}}\quad
\subfloat[Prey]{\includegraphics[scale=0.45]{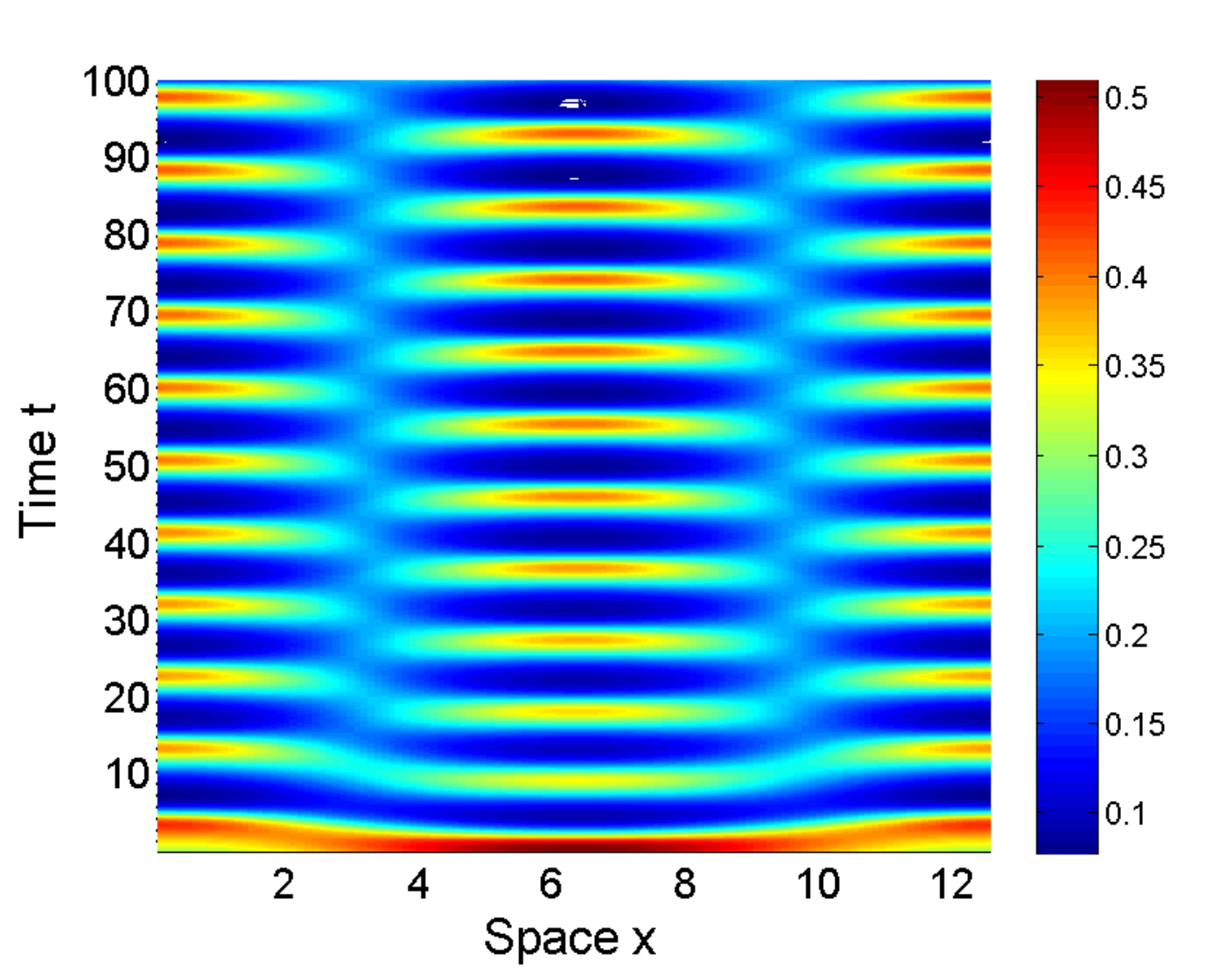}}
\subfloat[Predator]{\includegraphics[scale=0.45]{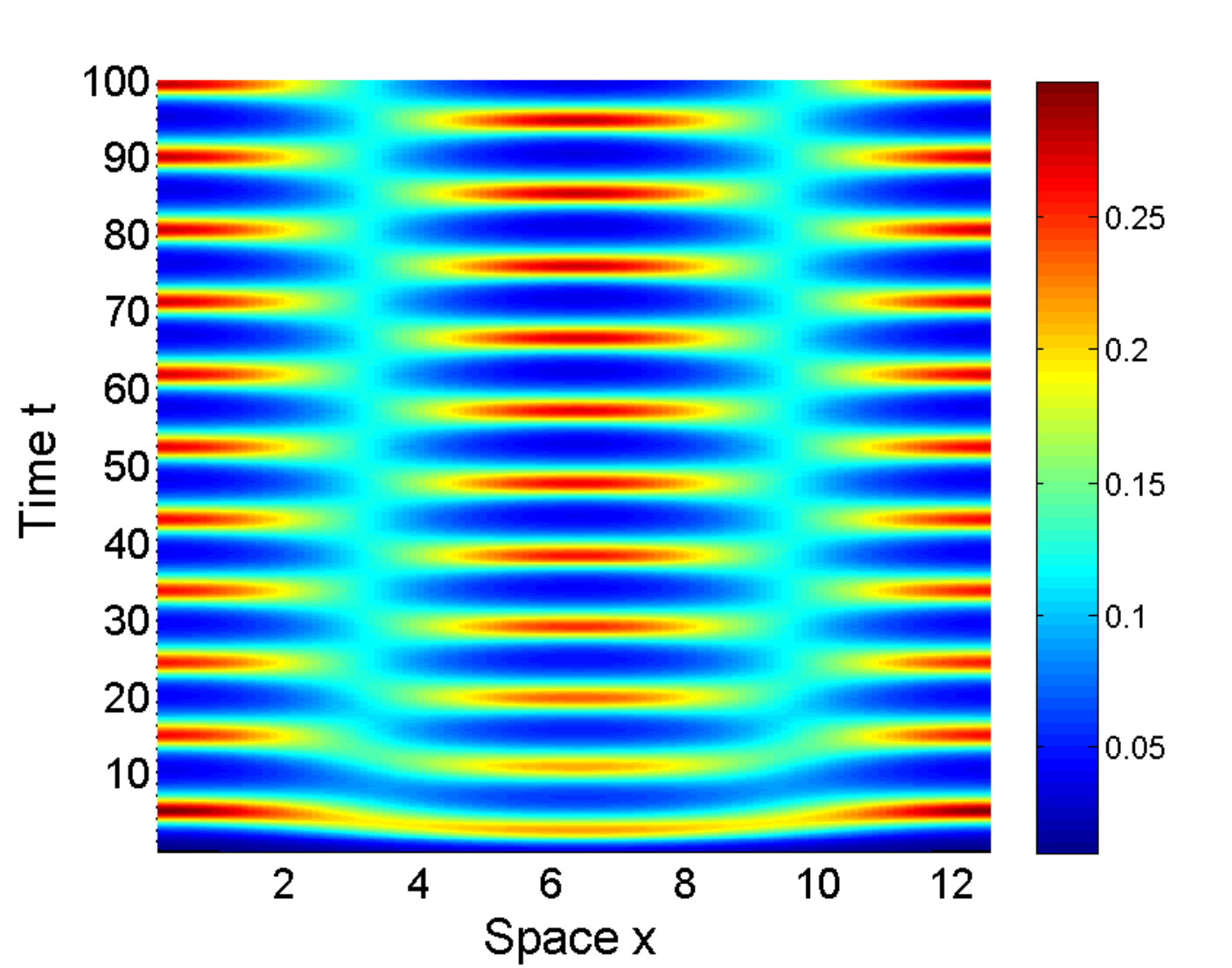}}
  \caption{Dynamics of Eq. \eqref{equ-prenonlocal} with parameters in \eqref{equ-preparameter44} and $m=6$ ($\la=0.2$): (top row) mode-1 spatiotemporal patterns with initial values $u_0(x)=0.4-0.1\cos(x/4),~v_0(x)=0.02-0.01\cos(x/4)$; (bottom row) mode-2 spatiotemporal patterns with initial values $u_0(x)=0.4-0.1\cos(x/2),~v_0(x)=0.02-0.01\cos(x/2)$. \label{Fig-solupre2}}
\end{figure}
\end{example}

\begin{example}\label{example-pre3}
When
\begin{equation}\label{equ-preparameter55}
d_1=0.005,~d_2=1,~\theta=1,~k=0.5,~l=2,
\end{equation}
the graphs of $D_i=0$ and $T_i=0$ are depicted in $(\la,p)$ plane in Fig. \ref{Fig-prediagram} (iii). In this case Hopf bifurcations cannot occur and steady state bifurcations occur. And the steady state bifurcation points can be computed as:
\begin{equation*}
\begin{split}
\la^{S}_{4,-}=0.3264,~\la^{S}_{4,+}=0.4136,~\la^{S}_{5,-}=0.2317,~\la^{S}_{5,+}=0.4126,\\
\la^{S}_{6,-}=0.2018,~\la^{S}_{6,+}=0.3868,~\la^{S}_{7,-}=0.2087,~\la^{S}_{7,+}=0.3423.
\end{split}
\end{equation*}
In Fig.  \ref{Fig-solupre3}, with $\la=0.40$, the mode-4 and mode-5 spatially non-homogeneous steady states are observed  with different initial conditions.  Compared with the local system \eqref{equ-prelocal} in which there is no stable spatial patterns, we can conclude that the spatially non-homogeneous steady states are induced by the nonlocal competition.
\end{example}

\begin{figure}[htp]
\centering
\subfloat[Prey]{\includegraphics[scale=0.45]{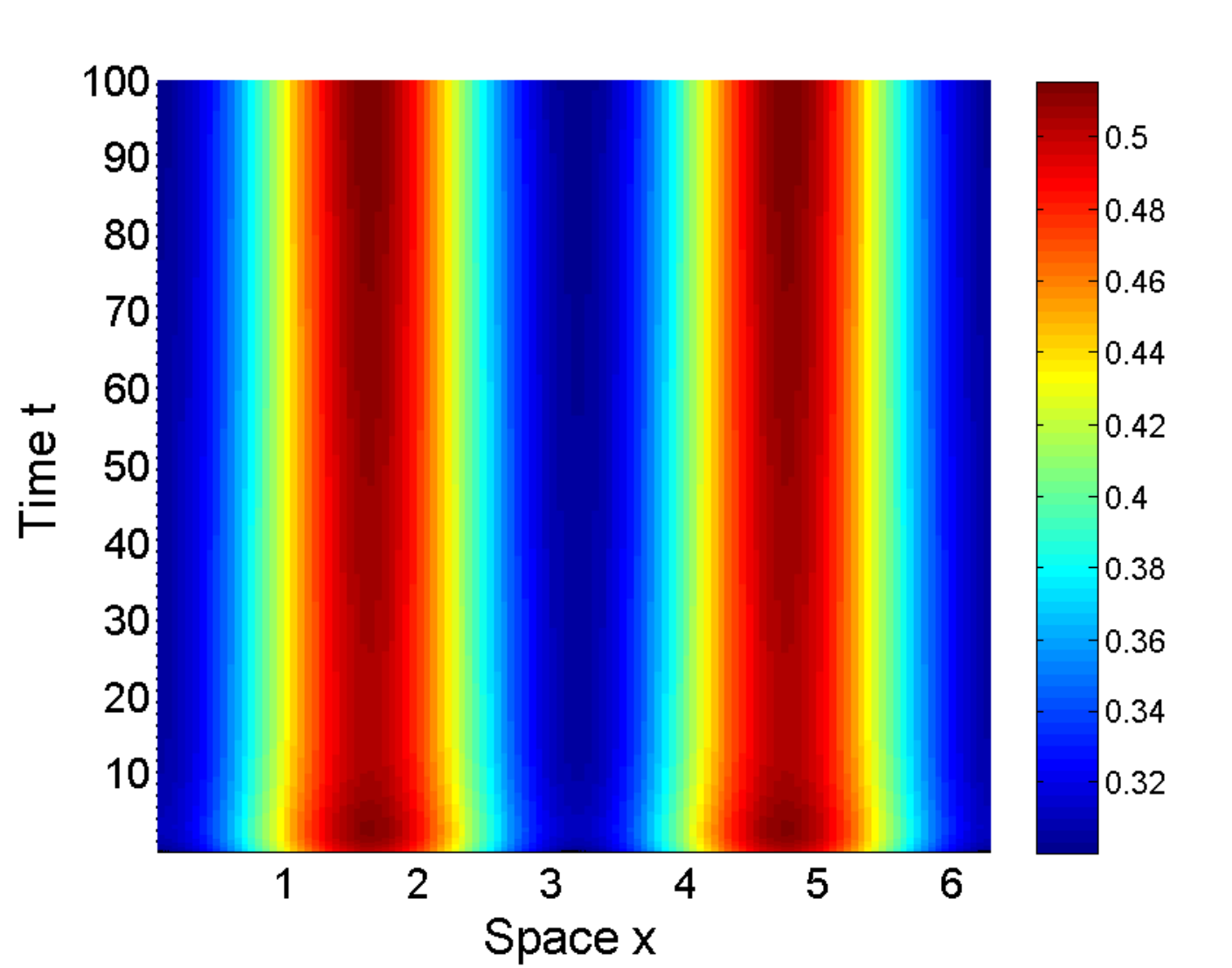}}
\subfloat[Predator]{\includegraphics[scale=0.45]{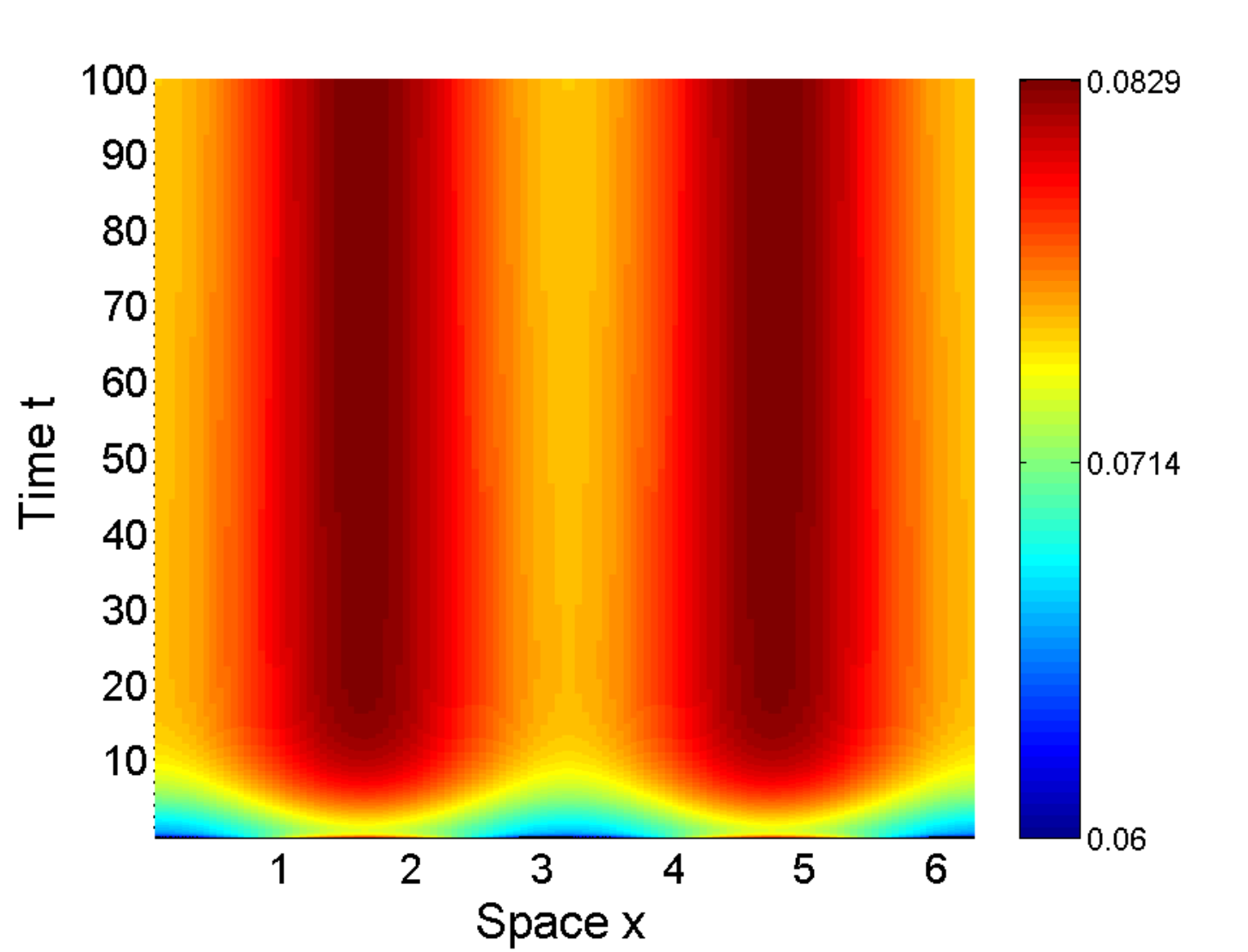}}\quad
\subfloat[Prey]{\includegraphics[scale=0.45]{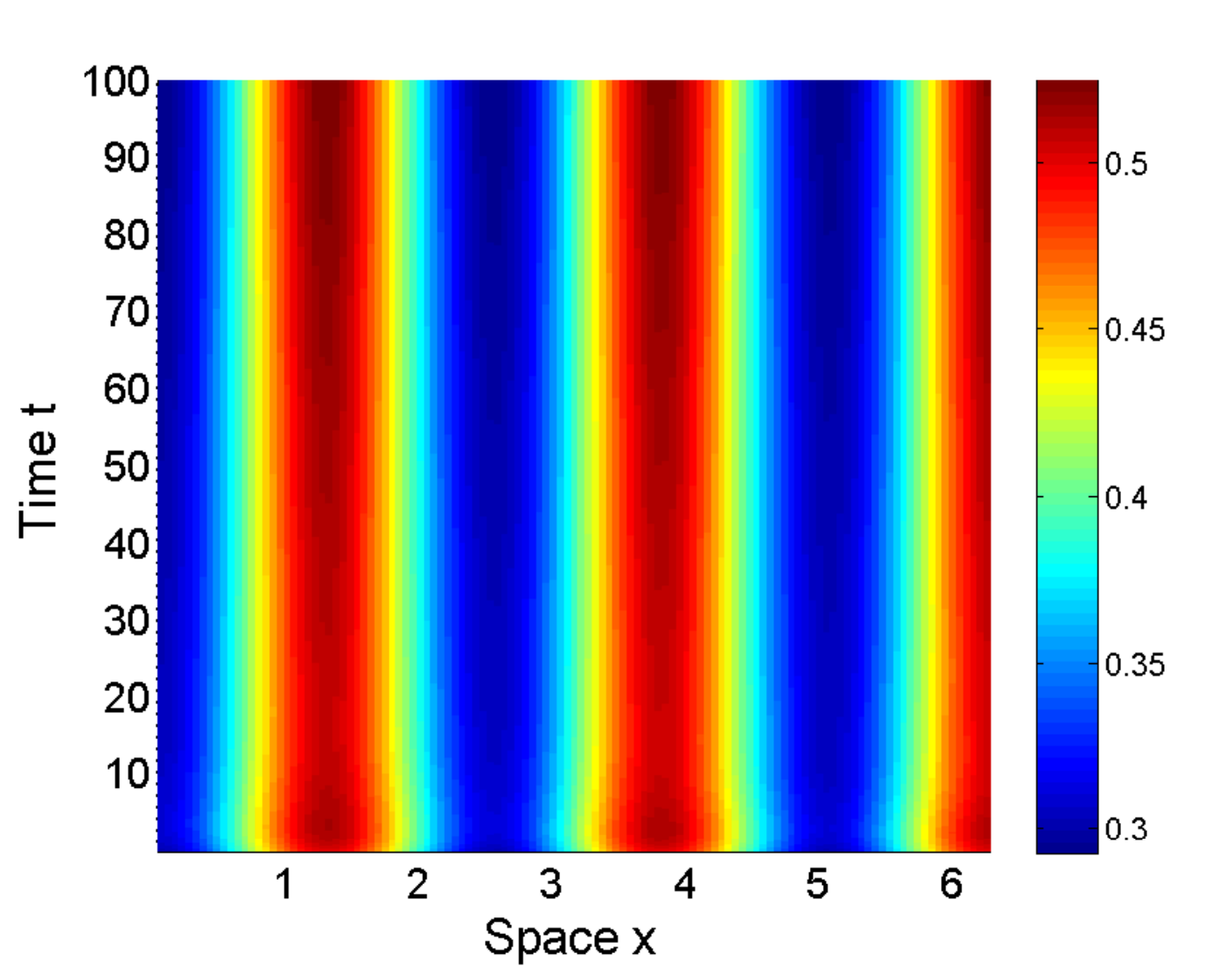}}
\subfloat[Predator]{\includegraphics[scale=0.45]{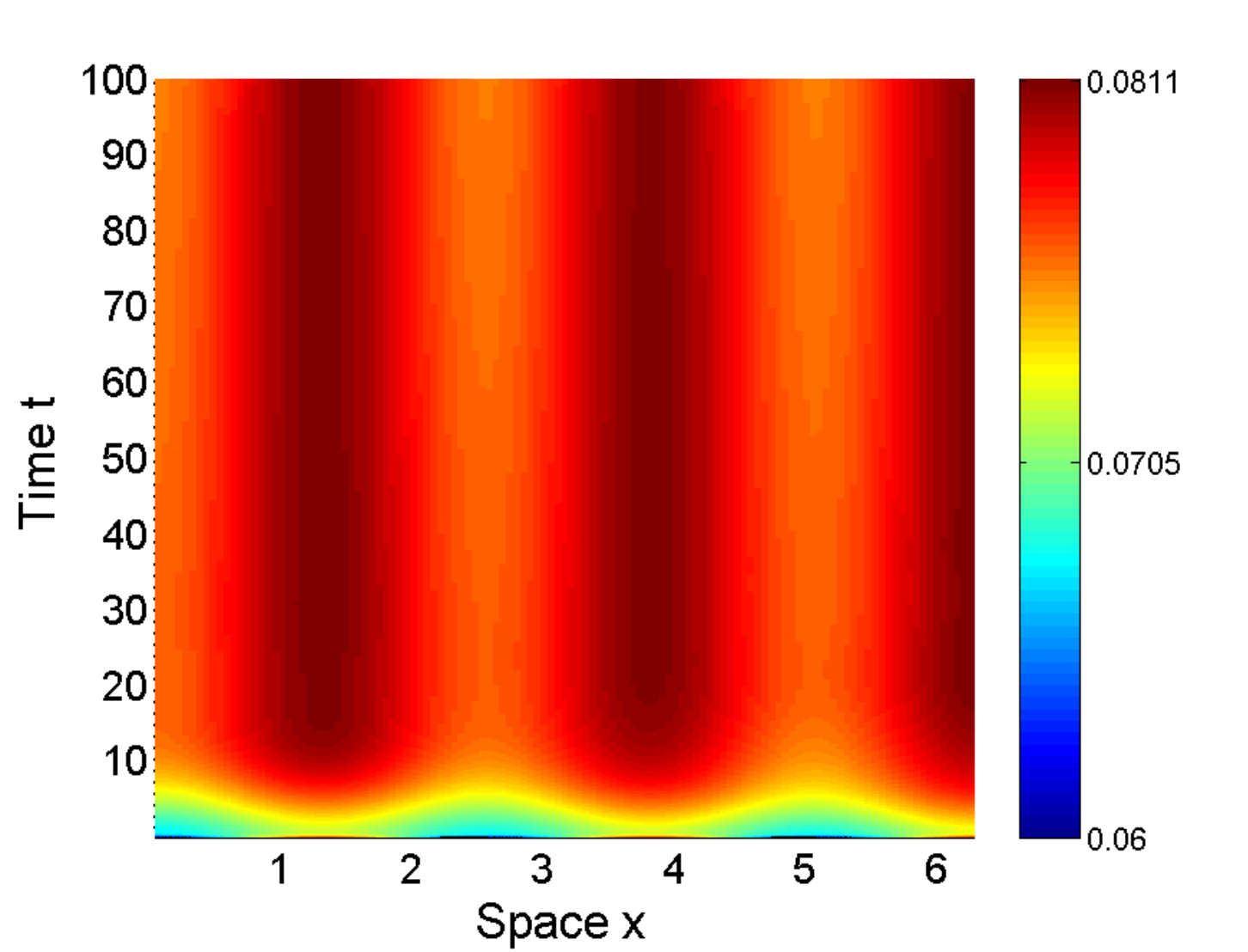}}
  \caption{Dynamics of Eq. \eqref{equ-prenonlocal} with the parameters in \eqref{equ-preparameter55} and $m=3.5$ ($\la=0.4$): (top row) mode-4 spatial patterns with initial values $u_0(x)=0.4-0.1\cos(2x),~v_0(x)=0.07-0.07\cos(2x)$; (bottom row) mode-5 spatial patterns with initial values $u_0(x)=0.4-0.1\cos(2.5x),~v_0(x)=0.07-0.07\cos(2.5x)$.\label{Fig-solupre3}}
\end{figure}

\begin{example}
Finally we take the parameters as
\begin{equation}\label{exm4}
d_1=0.006,~d_2=0.9,~\theta=1,~k=0.5,~l=4.
\end{equation}
The graphs of $D_i=0$ and $T_i=0$ are shown in $(\la,p)$ plane in Fig. \ref{Fig-prediagram} (iv). We have the spatially non-homogeneous Hopf bifurcation points:
\begin{equation*}
\la^{H}_{1,-}=0.0706,~\la^{H}_{1,+}=0.4011,
\end{equation*}
and the steady state bifurcation points are:
\begin{equation*}
\la^{S}_{10,-}=0.2988,~\la^{S}_{10,+}=0.3602,~\la^{S}_{11,-}=0.2765,~\la^{S}_{11,+}=0.3478,~\la^{S}_{12,-}=0.2837,~\la^{S}_{12,+}=0.3128.
\end{equation*}

\begin{figure}[htp]
\centering
\subfloat[Prey]{\includegraphics[scale=0.45]{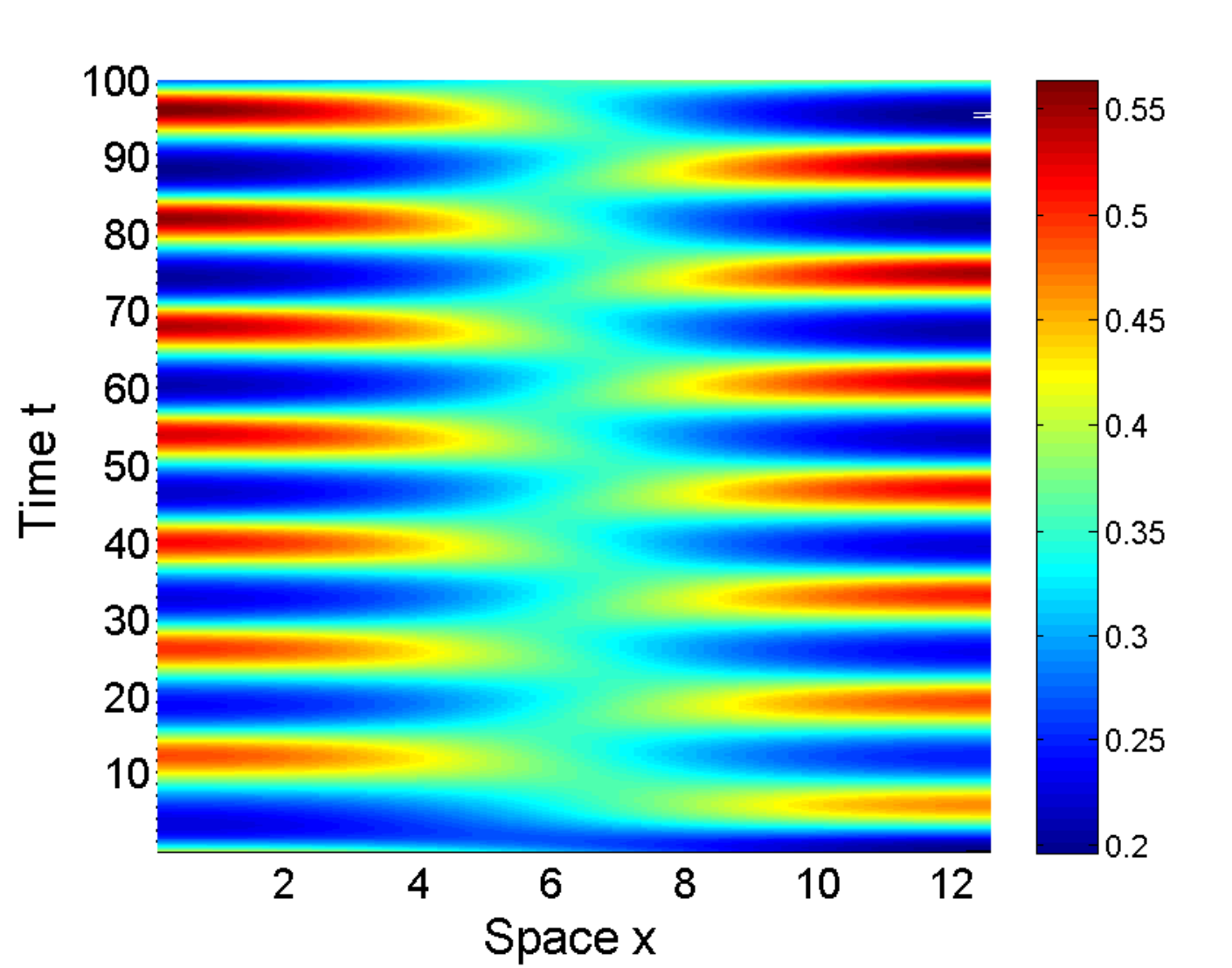}}
\subfloat[Predator]{\includegraphics[scale=0.45]{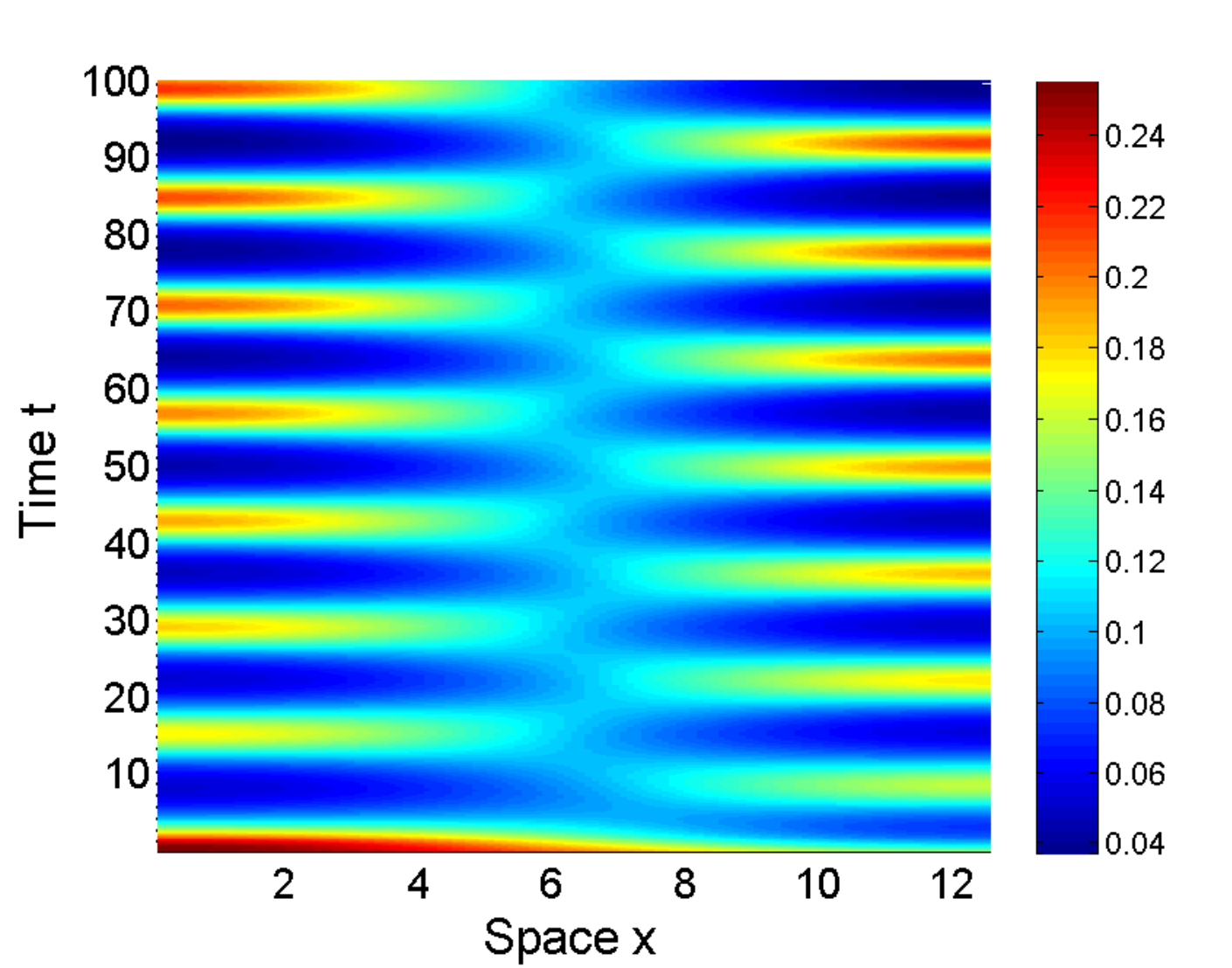}}\quad
\subfloat[Prey]{\includegraphics[scale=0.45]{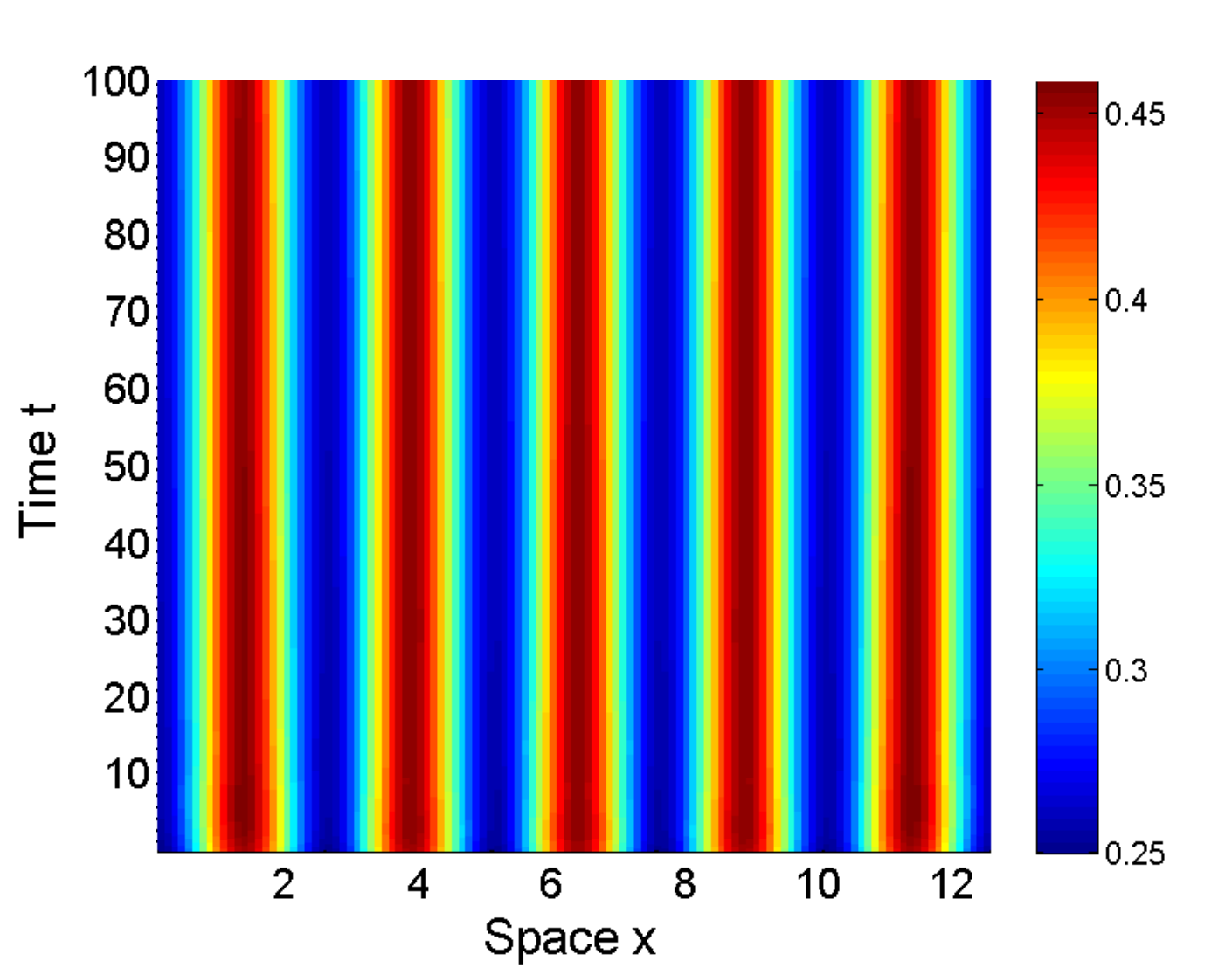}}
\subfloat[Predator]{\includegraphics[scale=0.45]{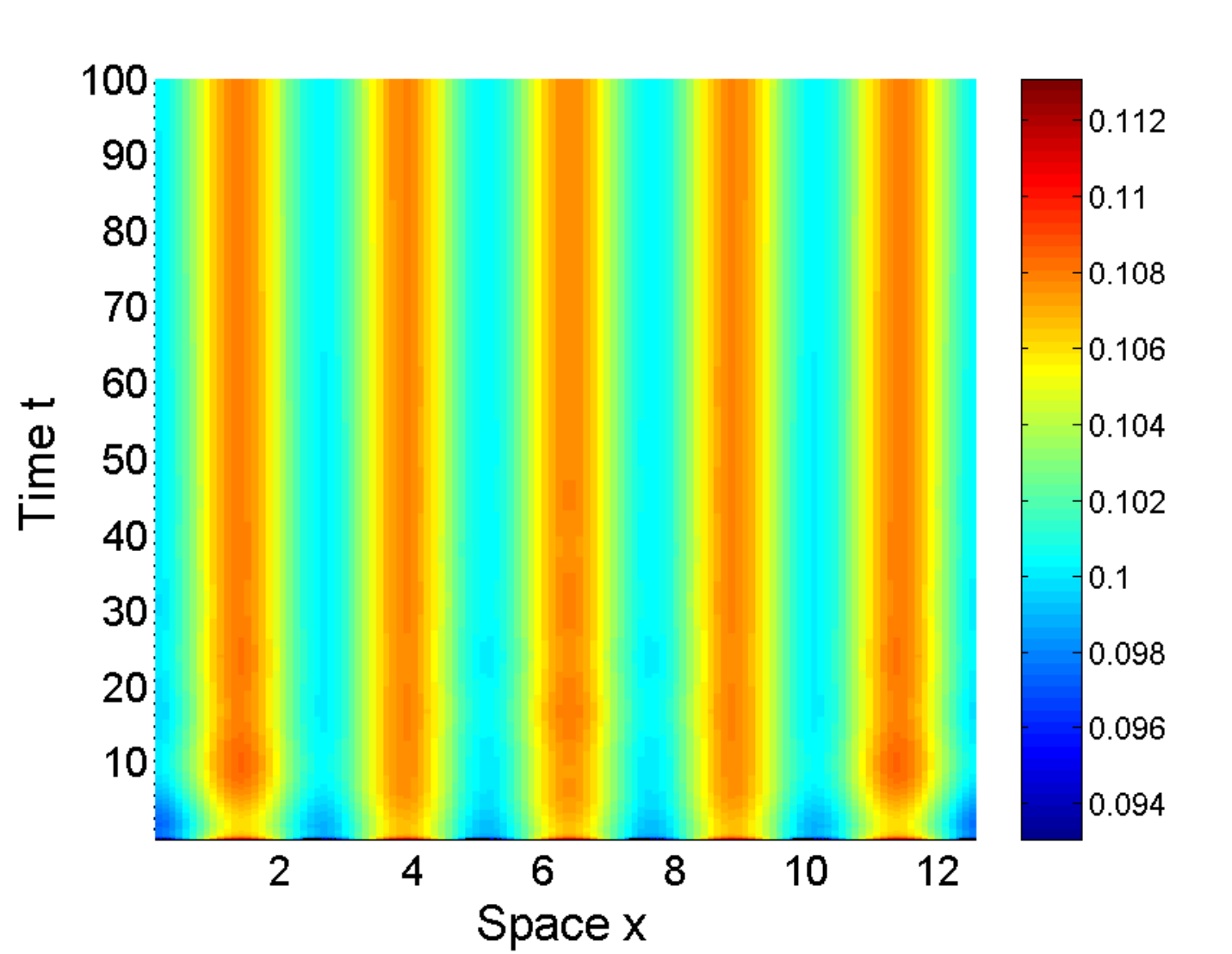}}
  \caption{The dynamics of Eq. \eqref{equ-prenonlocal} with the parameters being \eqref{exm4} and $m=3.857$ ($\la=0.35$): (top row) mode-1 spatiotemporal pattern with initial values $u_0(x)=0.3+0.1\cos(x/4),~v_0(x)=0.2+0.05\cos(x/4)$; (bottom row) mode-10 spatial patterns with initial values $u_0(x)=0.35-0.1\cos(10x/4),~v_0(x)=0.103-0.01\cos(10x/4)$. \label{Fig-solupre4}}
\end{figure}

By using the normal form calculations (see \cite{YiWeiShi2009} and Remark \ref{rem4}), we find that
\begin{equation}\label{equ-prenumericalc1}
\mathcal{R}e(c_1(\la^{H}_{1,-}))=54.6124>0,~\mathcal{R}e(c_1(\la^{H}_{1,+}))=0.0434<0.
\end{equation}
As a consequence of \eqref{equ-prenumericalc1} and the fact that $\la^{H}_{1,-}<\la_*$, $\la^{H}_{1,+}>\la_*$, we have
\begin{equation*}
\mu'(\la^{H}_{1,-})>0,~\mu'(\la^{H}_{1,+})<0.
\end{equation*}
According to \cite{YiWeiShi2009}, we know that the spatially non-homogeneous Hopf bifurcation at $\la=\la^{H}_{1,-}$ and $\la=\la^{H}_{1,+}$ are both supercritical, and the bifurcating periodic orbits near $\la=\la^{H}_{1,-}$ and  $\la=\la^{H}_{1,+}$ are both stable. In Fig.\ref{Fig-solupre4}, with $\la=0.35$,  the mode-1 spatially non-homogeneous time-periodic pattern and mode-10 spatially non-homogeneous steady state are observed with different initial conditions.
\end{example}

\section{Discussion}

In this work, we study the effect of spatial average on the pattern formation of reaction-diffusion systems. For a classical scalar reaction-diffusion equation subject to the homogeneous Neumann boundary condition, spatial pattern formation is impossible on a convex spatial domain. However, when the spatial average is incorporated into the model, stable spatially non-constant steady state can emerge from a symmetry-breaking bifurcation. For a classical two-species reaction-diffusion system with homogeneous Neumann boundary condition, Hopf bifurcation of the corresponding ODE system induces spatially homogeneous periodic orbits, and non-constant steady state can be generated through Turing instability, but stable spatially non-homogeneous time-periodic patterns can only be generated through secondary Turing-Hopf bifurcation which is co-dimension two \cite{Song2016,Song2018}. It is found here that in a two-species reaction-diffusion system with spatial average, spatially non-homogeneous periodic orbits can be generated from a primary (co-dimension one) spatially non-homogeneous Hopf bifurcation from the stable constant steady state.

Another point we want to address is that nonlocality induced instability allows more flexible conditions on the kinetics of the underlying system, and it does not require typical activator-inhibitor interaction between the two species. The diffusive Lotka-Volterra cooperative model with spatial average effect serves as an example to support this view, and the example of diffusive Rosenzweig-MacArthur predator-prey model with spatial average shows how the spatial average can help to generate spatiotemporal patterns in an otherwise stable system with a  unique homogeneous state. Our theory and these examples clearly show that the addition of effect of spatial average in reaction-diffusion systems broadens the range of reaction-diffusion models for spatiotemporal  pattern formation.

Usually spatial heterogeneity increases the complexity of spatial patterns. It is interesting to notice that the mechanism of pattern formation here is to add some partial spatial homogeneity. In some reaction-diffusion systems, stable spatial patterns are not able to be formed. But, when the spatial average is added into the system, spatial patterns can be observed. If we replace all the local terms with the corresponding spatial average terms, the system will be equivalent to an ODE system, spatial pattern formation is also impossible. Thus a combination of locality and nonlocality may be helpful for the formation of spatial patterns.

\section*{Acknowledgement}
This work was done when the first author visited  William \& Mary during the academic year 2016-2018, and she would like to thank Department of Mathematics at  William \& Mary for their support and kind hospitality.

\section*{Appendix}
\appendix
\renewcommand{\theequation}{\thesection A.\arabic{equation}}

\subsection*{The proof of Theorem \ref{theorem-global}}

\begin{proof}
Firstly, we integrate both sides of Eq. \eqref{equ-exampletwogeneral} on $\Om$ and divide by $|\Om|$ which is the spatial domain size, then we obtain a ODE system of $\bar{u}$:
\begin{equation}\label{equ-average}
\bar{u}_t=a-(b+c)\bar{u}-(d+e)\bar{u}^2.
\end{equation}
Then, the equilibrium of Eq. \eqref{equ-average} also satisfies \eqref{equ-exampletwogeneral} which admits a unique positive root $u=u_*$. In addition, for Eq. \eqref{equ-average}, the unique equilibrium $u=u_*$ is globally stable, and all the solutions of \eqref{equ-average} will converge to $u=u_*$ as $t\rightarrow+\infty$. Note that $\bar{u}$ is a function of $t$. Then, we rewrite Eq. \eqref{equ-exampletwogeneral} as:
\begin{equation}
u_t=d\Delta u+A(t)-B(t)u,
\end{equation}
where $A(t)=a-b\bar{u}-d\bar{u}^2$ and $B(t)=c+e\bar{u}$.
Denote $\tilde{A}=\lim\limits_{t\rightarrow+\infty}A(t)>0,~\tilde{B}=\lim\limits_{t\rightarrow+\infty}B(t)>0$, then for any $0<\epsilon\ll1$, there exists $T>0$ such that for arbitrary $t>T$, we have
\begin{equation*}
\begin{cases}
u_t\leq d\Delta u+(\tilde{A}+\epsilon)-(\tilde{B}-\epsilon)u,\\
u_t\geq d\Delta u+(\tilde{A}-\epsilon)-(\tilde{B}+\epsilon)u.
\end{cases}
\end{equation*}
Therefore, we can use the $u_1\geq u(x,t)$ as the upper solution with $u_1$ is the solution of the following equation:
\begin{equation}\label{equ-upper}
\begin{cases}
u_{1}^{\prime}=(\tilde{A}+\epsilon)-(\tilde{B}-\epsilon)u_1,~t>T,\\
u_{1}(0)=\max\limits_{x\in\Om} u(x,T),
\end{cases}
\end{equation}
and the lower solution $u_2\leq u(x,t)$ satisfying
\begin{equation}\label{equ-lower}
\begin{cases}
u_{2}^{\prime}=(\tilde{A}-\epsilon)-(\tilde{B}+\epsilon)u_2,~t>T,\\
u_{2}(0)=\max\limits_{x\in\Om} u(x,T).
\end{cases}
\end{equation}
Moreover, by the theory of ODE, we know the asymptotic behavior of Eqs. \eqref{equ-upper} and \eqref{equ-lower}:
\begin{equation*}
\lim_{t\rightarrow+\infty}u_1(t)=\frac{\tilde{A}+\epsilon}{\tilde{B}-\epsilon},~~\lim_{t\rightarrow+\infty}u_2(t)=\frac{\tilde{A}-\epsilon}{\tilde{B}+\epsilon}.
\end{equation*}
By the arbitrariness of $\epsilon$, we obtain that $\lim\limits_{t\rightarrow+\infty}u(x,t)=\tilde{A}/\tilde{B}=u_*$. We complete the proof.
\end{proof}

\bibliographystyle{plain}%{astron}%{unsrt}% {astron}% {amsplain}

\bibliography{nonlocal}
\end{document}